\documentclass[11pt, twoside, leqno]{article}

\usepackage{amssymb}
\usepackage{amsmath}
\usepackage{amsthm}
\usepackage{color}
\usepackage{mathrsfs}
\usepackage{indentfirst}
\usepackage{txfonts}

\allowdisplaybreaks

\def\red{\color{red}}

\def\rr{{\mathbb R}}
\def\rn{{\mathbb{R}^n}}

\def\zz{{\mathbb Z}}
\def\cc{{\mathbb C}}

\def\nn{{\mathbb N}}

\def\ca{{\mathcal A}}
\def\CB{{\mathcal B}}

\def\cm{{\mathcal M}}

\def\cp{{\mathcal P}}

\def\cs{{\mathcal S}}

\def\cx{{\mathcal X}}

\def\fz{\infty }
\def\az{\alpha}

\def\lz{\lambda}

\def\tz{\theta}
\def\sz{\sigma}

\def\vaz{\varepsilon}

\def\lf{\left}
\def\r{\right}

\def\la{\langle}
\def\ra{\rangle}

\def\ls{\lesssim}

\def\noz{\nonumber}
\def\wz{\widetilde}

\def\st{\subset}
\def\com{\complement}

\def\fin{\mathop\mathrm{\,fin\,}}

\def\loc{{\mathop\mathrm{\,loc\,}}}
\def\supp{\mathop\mathrm{\,supp\,}}

\def\XXint#1#2#3{{\setbox0=\hbox{$#1{#2#3}{\int}$ }
\vcenter{\hbox{$#2#3$ }}\kern-.6\wd0}}

\DeclareMathOperator*{\esssup}{ess\ sup}

\def\unp{\underline{p}}

\def\la{{\langle}}
\def\ra{{\rangle}}
\def\({\left(}
\def \){ \right)}

\def\CB{{\mathcal B}}

\def\one{\mathbf{1}}

\def\Hfin{H_{X,\fin}^{A,q,d}}

\newtheorem{theorem}{Theorem}[section]
\newtheorem{lemma}[theorem]{Lemma}
\newtheorem{corollary}[theorem]{Corollary}
\newtheorem{assumption}[theorem]{Assumption}
\newtheorem{proposition}[theorem]{Proposition}

\theoremstyle{definition}
\newtheorem{remark}[theorem]{Remark}
\newtheorem{definition}[theorem]{Definition}
\renewcommand{\appendix}{\par
\setcounter{section}{0}%
\setcounter{subsection}{0}%
\setcounter{subsubsection}{0}%
\gdef\thesection{\@Alph\c@section}%
\gdef\thesubsection{\@Alph\c@section.\@arabic\c@subsection}%
\gdef\theHsection{\@Alph\c@section.}%
\gdef\theHsubsection{\@Alph\c@section.\@arabic\c@subsection}%
\csname appendixmore\endcsname
}

\numberwithin{equation}{section}
\begin{document}
\title{\bf\Large Anisotropic Ball Campanato-Type
Function Spaces and Their Applications
\footnotetext{\hspace{-0.35cm} 2020 {\it Mathematics Subject Classification}.
Primary 46E30; Secondary 42B30, 42B25, 42B35, 28C20.
\endgraf {\it Key words and phrases.}
expansive matrix, ball quasi-Banach function space,
Hardy space, Campanato-type function space,
duality, Littlewood--Paley function, Carleson measure.
\endgraf This project is partially supported by the National Key Research
and Development Program of China (Grant No. 2020YFA0712900),
the National Natural Science Foundation of China
(Grant Nos. 11971058 and 12071197), and
China Postdoctoral Science Foundation (Grant No. 2022M721024).}}
\date{}
\author{Chaoan Li, Xianjie Yan and Dachun Yang\,\footnote{Corresponding author,
E-mail: \texttt{dcyang@bnu.edu.cn}/{\red{April 24, 2023}}/Final version.}
}
\maketitle
\date{}
\maketitle

\vspace{-0.7cm}
\begin{center}
\begin{minipage}{13cm}
{\small {\bf Abstract}\quad
Let $A$ be a general expansive matrix and let
$X$ be a ball quasi-Banach function space on $\mathbb R^n$,
which supports both a Fefferman--Stein vector-valued
maximal inequality and the boundedness of the powered
Hardy--Littlewood maximal operator on its associate space.
The authors first introduce some anisotropic ball Campanato-type function
spaces associated with both $A$ and $X$, prove that these
spaces are dual spaces of anisotropic
Hardy spaces $H_X^A(\mathbb R^n)$ associated with both $A$ and $X$,
and obtain various anisotropic Littlewood--Paley
function characterizations of $H_X^A(\mathbb R^n)$.
Also, as applications, the authors establish several equivalent
characterizations of anisotropic ball Campanato-type function spaces,
which, combined with the atomic decomposition
of tent spaces associated with both $A$ and $X$, further induces their
Carleson measure characterizations.
All these results have a wide range of generality and,
particularly, even when they are applied to Morrey spaces
and Orlicz-slice spaces, some of the obtained results are also new.
The novelties of this article are reflected in that, to overcome the
essential difficulties caused by the absence of both an explicit
expression and the absolute continuity of quasi-norm $\|\cdot\|_X$,
the authors embed $X$ under consideration into the anisotropic
weighted Lebesgue space with certain special weight and then fully use the known
results of this weighted Lebesgue space.
}
\end{minipage}
\end{center}

\vspace{0.3cm}


\section{Introduction\label{s0}}
Recall that the dual theory of classical Hardy spaces on the
Euclidean space $\rn$ plays an important role in many branches
of analysis, such as harmonic analysis and partial differential
equations, and has been systematically considered and developed
so far; see, for instance, \cite{fs72,mu94}. Indeed, in 1969, Duren
et al. \cite{drs69} first identified the Lipshitz space with the dual
space of the Hardy space $H^p(\mathbb{D})$ of holomorphic functions,
where $p\in(0,1)$ and the \emph{symbol} $\mathbb{D}$ denotes the unit
disc of $\mathbb{C}$. Later, Walsh \cite{w73} further extended this dual
result to the Hardy space on the upper half-plane
$\mathbb{R}^{n+1}_+:=\rn\times(0,\fz)$.
On the real Hardy spaces, the famous dual
theorem, that is, the space
$\mathop{\mathrm{BMO}}(\rn)$
of functions with bounded mean oscillation
is the dual space of the Hardy space $H^1(\rn)$,
is due to Fefferman and Stein \cite{fs72}. Moreover, it is worth pointing
out that the complete dual theory of the Hardy space $H^p(\rn)$ with $p\in(0,1]$
was given by Taibleson and Weiss \cite{tw80}, in which the dual space of
$H^p(\rn)$ proves the special
Campanato space introduced by Campanato \cite{c64}.

Recently, Sawano et al. \cite{shyy17} originally introduced
the ball quasi-Banach function space $X$
which further generalizes the Banach function space
in \cite{bs88} in order to include weighted Lebesgue spaces,
Morrey spaces, mixed-norm Lebesgue spaces, Orlicz-slice spaces,
and Musielak--Orlicz spaces.
Observe that the aforementioned several function spaces
are not quasi-Banach function spaces which were originally
introduced in \cite[Definitions 1.1 and 1.3]{bs88};
see, for instance, \cite{shyy17,st15,wyyz,zwyy}.
In the same article \cite{shyy17}, Sawano et al. also introduced the Hardy space
$H_X(\rn)$, associated with $X$,
and established its various maximal function characterizations
by assuming that the Hardy--Littlewood maximal operator is
bounded on the $p$-convexification of $X$, as well as its
several other real-variable characterizations, respectively,
in terms of atoms, molecules, and Lusin-area functions
by assuming that the Hardy--Littlewood maximal operator satisfies a
Fefferman--Stein vector-valued inequality on $X$ and is
bounded on the associate space
of $X$.

Later, Wang et al. \cite{wyy} further established the
Littlewood--Paley $g$-function and
the Littlewood--Paley $g_{{\lambda}}^\ast$-function
characterizations of both $H_X(\rn)$ and its local version $h_X(\rn)$
and obtained the boundedness of Calder\'on--Zygmund
operators and pseudo-differential operators,
respectively, on $H_X(\rn)$ and $h_X(\rn)$;
Yan et al. \cite{yyy20} established the dual theorem and the
intrinsic square function characterizations of $H_X(\rn)$;
Zhang et al. \cite{zhyy21} introduced some new ball Campanato-type
function space which proves the dual space of $H_X(\rn)$
and established its Carleson measure characterization.
Very recently, on spaces ${\mathcal X}$ of homogeneous type,
Yan et al. \cite{yhyy,yhyy2} introduced ball
quasi-Banach function spaces $Y({\mathcal X})$
and Hardy-type spaces $H_Y(\cx)$, associated with $Y({\mathcal X})$,
and developed a complete real-variable theory of $H_Y(\cx)$.
For more studies about ball quasi-Banach function spaces,
we refer the reader to
\cite{cwyz20,its19,is17,s18,syy1,tyyz21,yyy20b}.

On the other hand, starting from 1970's, there has been an increasing
interesting in extending classical function spaces arising
in harmonic analysis from $\rn$ to various anisotropic
settings and some other domains; see, for instance, \cite{cgp22,fs82,
gkp21,gn18,
hhllyy19,lyz2023,st87,tl15,tribelfs83,tribelfs92}.
The study of function spaces on $\rn$ associated with anisotropic
dilations was originally started from the celebrated articles
\cite{calderon77,ct75,ct77}
of Calder\'on and Torchinsky
on anisotropic Hardy spaces. In 2003,
Bownik \cite{Bownik} introduced and investigated the anisotropic
Hardy space $H^p_{A}(\rn)$ with $p\in(0,\fz)$,
where $A$ is a general expansive matrix on $\rn$.
Since then, various variants of classical Hardy spaces
over the anisotropic Euclidean space have been introduced and
their real-variable theories have been systematically developed.
To be precise, Bownik et al. \cite{blyz08weight} further extended the
anisotropic Hardy space to the weighted setting.
Li et al. \cite{libaode14} introduced the anisotropic
Musielak--Orlicz Hardy space $H^{\varphi}_A(\mathbb{R}^n)$,
where $\varphi$ is an anisotropic Musielak--Orlicz function,
and characterized $H^{\varphi}_A(\mathbb{R}^n)$ by several maximal functions
and atoms.
Liu et al. \cite{lyy16,lyy2018} first introduced the
anisotropic Hardy--Lorentz space $H^{p,q}_A(\mathbb{R}^n)$,
with $p\in(0,1]$ and $q\in(0,\infty]$,
and established their several real-variable characterizations,
respectively,
in terms of atoms
or finite atoms, molecules,
maximal functions, and Littlewood--Paley functions,
which are further applied to obtain the real interpolation theorem of
$H^{p,q}_A(\mathbb{R}^n)$ and the boundedness of anisotropic
Calder\'on--Zygmund operators on $H^{p,q}_A(\mathbb{R}^n)$
including the critical case. Liu et al. \cite{liu2020,lwyy2018}
and Huang et al. \cite{hlyy20} further generalized the corresponding
results in \cite{lyy16,lyy2018}
to variable Hardy spaces and mixed-norm Hardy spaces, respectively.
Recently, Liu et al. \cite{lwyy2019,lyy17} introduced the
anisotropic variable Hardy--Lorentz space $H^{p(\cdot),q}_A(\rn)$,
where
$p(\cdot):\ \rn\rightarrow(0,\fz]$
is a variable
exponent function satisfying the globally log-H\"older
continuous condition and $q\in(0,\fz]$,
and developed a complete real-variable theory of these
spaces including various equivalent
characterizations and the boundedness of sublinear operators.
Independently, Almeida et al. \cite{abr2017} also investigated
the anisotropic variable Hardy--Lorentz space
$H^{p(\cdot),q(\cdot)}(\mathbb{R}^n,A)$, where $p(\cdot)$ and
$q(\cdot)$ are nonnegative measurable functions on $(0,\infty)$.
In \cite{abr2017}, equivalent characterizations of
$H^{p(\cdot),q(\cdot)}(\mathbb{R}^n,A)$ in terms of maximal
functions and atoms were established.
It is remarkable that the anisotropic variable Hardy--Lorentz space
$H^{p(\cdot),q(\cdot)}(\mathbb{R}^n,A)$ in \cite{abr2017} and that
in \cite{lwyy2019,lyy17} can not cover each other because the variable
exponent $p(\cdot)$ in \cite{abr2017} is only
defined on $(0,\infty)$, instead of $\mathbb{R}^n$.
Particularly, Huang et al. \cite{hw22,hyy21} further enriched the
real-variable theory of
anisotropic mixed-norm Campanato spaces
and anisotropic variable Campanato spaces and established the dual
theory of both anisotropic Hardy spaces
$H^{\vec{p}}_A(\mathbb{R}^n)$ and $H^{p(\cdot)}_A(\rn)$
with the full ranges of both $\vec{p}$ and $p(\cdot)$.
For more studies about function spaces on the anisotropic
Euclidean space, we refer the reader to
\cite{Bownik2,blyz10weight,cgn17,cgn19,lby14,lffy15}.

Recall that the anisotropic Hardy space $H_X^A(\rn)$ associated with both $A$
and $X$ was first introduced and studied by Wang et al. \cite{wyy22},
in which they characterized $H_X^A(\rn)$ in terms of maximal functions,
atoms, finite atoms, and molecules and obtained the boundedness of
Calder\'on--Zygmund operators on $H_X^A(\rn)$. Motivated by this
and \cite{zhyy21}, a quite \emph{natural question} arises:
can we prove whether or not the dual space of $H_X^A(\rn)$ is the
anisotropic ball Campanato-type function space and characterize
this space by the Carleson measure? The main target of this
article is to give an affirmative answer to this question.
Indeed, to answer this question and also to enrich the
real-variable theory
of anisotropic Campanato spaces associated with both $A$ and $X$,
in this article, by borrowing some ideas from \cite{zhyy21}, namely
considering finite linear combinations of atoms as a whole
instead of a single atom, we introduce the anisotropic ball
Campanato-type function space and give some applications. Using this
and the additional assumptions that the Hardy--Littlewood maximal operator
satisfies some Fefferman--Stein vector-valued inequality on $X$
and is bounded on the associated space of $X$,
we get rid of the dependence on the concavity of $\|\cdot\|_X$
and prove that the dual space of $H_X^{A}(\rn)$ is just the anisotropic ball
Campanato-type function space. From this, we further deduce several equivalent
characterizations of anisotropic ball Campanato-type function spaces.
Moreover, via embedding $X$ into
a certain anisotropic weighted Lebesgue space,
we overcome the difficulty caused by the absence of both an explicit
expression and the absolute continuity of the quasi-norm $\|\cdot\|_X$
under consideration and
establish the anisotropic Littlewood--Paley characterizations of $H_X^{A}(\rn)$,
which, together with the dual theorem of $H_X^{A}(\rn)$ and the
atomic decomposition of anisotropic tent spaces associated with $X$,
finally implies the Carleson measure characterization of
anisotropic ball Campanato-type function spaces.

It is remarkable that the results obtained in this article
have a wide range of generality because ball quasi-Banach
function spaces include so many specific function spaces.
Particularly, when $X$ becomes the Morrey space,
the Littlewood--Paley function characterizations of
anisotropic Hardy--Morrey spaces are new
while the dual theorem and the Carleson measure
characterization are not applicable because
Morrey spaces do not have any absolutely continuous quasi-norm;
when $X$ becomes the Orlicz-slice space,
the obtained results are completely new;
when $X$ becomes the weighted Lebesgue space or the Orlicz space,
the dual theorem and the Carleson measure characterization are new
while the Littlewood--Paley function characterizations of
anisotropic Hardy-type spaces are obtained in \cite{lfy15,lyy2018};
when $X$ becomes the Lorentz space, the variable Lebesgue space,
or the mixed-norm Lebesgue space, the obtained results coincide
with those in \cite{hlyy20,hw22,hyy21,llh2023,lwyy2018,lyy2018}.
Obviously, due to the flexibility and the operability, more
applications of these results obtained in this article to newfound
function spaces are completely possible.

The remainder of this article is organized as follows.

In Section \ref{s1}, we recall some notation and concepts
which are used throughout this article. More precisely,
we first recall the definitions of the expansive matrix $A$, the step homogeneous
quasi-norm $\rho$, and the ball quasi-Banach function space $X$.
Then we make some mild assumptions on the boundedness of the
Hardy--Littlewood maximal operator on both $X$ and its associate space,
which are needed throughout this article.
Finally, we recall the concept of the non-tangential (grand) maximal function.

The aim of Section \ref{s2} is to give the dual space of the
anisotropic Hardy space $H_X^{A}(\rn)$ (see Theorem \ref{s2t1} below).
To this end, we first introduce
the anisotropic ball Campanato-type function space
$\mathcal{L}_{X,q,d,s}^{A}(\rn)$ (see Definition \ref{LAXqds} below)
and give an equivalent quasi-norm
characterization of $\mathcal{L}_{X,q,d,s}^{A}(\rn)$
(see Proposition \ref{s2p1} below). Using these,
both the known atomic and the known finite
atomic characterizations of $H_X^A(\rn)$,
and the assumptions that the Hardy--Littlewood maximal operator
satisfies some Fefferman--Stein vector-valued
inequality on $X$ and is bounded on the associated space of $X$,
we prove that the dual space of $H_X^A(\rn)$ is just
$\mathcal{L}_{X,q',d,\tz_0}^{A}(\rn)$.
At the end of this section, we also give its invariance
of $\mathcal{L}_{X,q,d,s}^{A}(\rn)$ on different indices
$q$ and $d$;
see Corollary \ref{s2c1} below.

In Section \ref{s3}, by applying the dual result obtained in Theorem \ref{s2t1}
and a key estimate from \cite[Lemma 2.13]{wyy22} (see also Lemma \ref{s3l2} below),
we obtain several equivalent characterizations of
$\mathcal{L}_{X,q,d,\tz_0}^A({\rn})$
(see Theorems \ref{s3t1} and \ref{s3t2} below),
which are further applied to establish the Carleson measure
characterization of $\mathcal{L}_{X,1,d,\tz_0}^{A}({\rn})$
in Section \ref{s5}.

Section \ref{s4} is devoted to establishing the anisotropic
Littlewood--Paley function characterization of $H_X^A(\rn)$,
including the anisotropic Lusin area function,
the anisotropic Littlewood--Paley $g$-function, and
the anisotropic Littlewood--Paley
$g_\lambda^\ast$-function, respectively,
in Theorems \ref{s4t1}, \ref{s4t1'}, and \ref{s4t1''} below.
We first prove Theorem \ref{s4t1}. To this end,
we first show that the quasi-norms in $X$ of the anisotropic
Lusin area functions defined by different Schwartz functions
are equivalent
(see Theorem \ref{s4t2} below). Then,
via borrowing some ideas from \cite{lyy2018}
and the anisotropic Calder\'on
reproducing formula (see Lemma \ref{s4l1} below),
we complete the proof of Theorem \ref{s4t1}.
From this and an approach initiated by Ullrich \cite{U}
and further developed by Liang et al. \cite{lhy12} and Liu et al. \cite{lwyy2019},
together with Fefferman--Stein vector-valued inequality on $X$,
we obtain the anisotropic Littlewood--Paley $g$-function
and the anisotropic Littlewood--Paley
$g_\lambda^\ast$-function characterizations.

In Section \ref{s5}, we establish the Carleson measure
characterization of $\mathcal{L}_{X,1,d,\tz_0}^{A}({\rn})$
(see Theorem \ref{s5t1} below).
Indeed, via using Theorems \ref{s2t1}, \ref{s3t1}, and \ref{s4t1},
as well as the atomic decomposition of anisotropic tent spaces
associated with $X$ (see Lemma \ref{s5l1} below),
we show that a measurable function $h$ belongs to
$\mathcal{L}_{X,1,d,\tz_0}^{A}({\rn})$
if and only if $h$ generates an $X$--Carleson measure $d\mu$. Moreover,
the norm of the $X$--Carleson measure $d\mu$ is equivalent to the
$\mathcal{L}_{X,1,d,\tz_0}^A({\rn})$-norm of $h$.

In Section \ref{s6}, we apply all the main results
obtained in the above sections to several specific
ball quasi-Banach function spaces. Particularly,
the results about Morrey spaces and Orlicz-slice spaces
are completely new and stated, respectively,
in Subsections \ref{s6-appl1} and \ref{s6-appl2};
part of the results about Orlicz spaces and weighted Lebesgue
spaces obtained in this article are new and stated, respectively,
in Subsections \ref{s6-appl6} and \ref{s6-appl7};
the results about Lorentz spaces, variable Lebesgue spaces,
and mixed-norm Lebesgue spaces
coincide with the known ones, which are stated in
Subsections \ref{s6-appl3}, \ref{s6-appl4},
and \ref{s6-appl5} successively.

At the end of this section, we make some conventions on notation.
Let $\nn:=\{1,2,\ldots\}$, $\zz_+:=\nn\cup\{0\}$, $\zz_+^n:=(\zz_+)^n$,
and $\bf{0}$ be the \emph{origin} of $\rn$.
For any multi-index $\az:=(\az_1,\ldots,\az_n)\in\mathbb{Z}_+^n$
and any $x:=(x_1,\ldots,x_n)\in\rn$, let $|\az|:=\az_1+\cdots+\az_n$,
$\partial^{\az}:=(\frac{\partial}{\partial x_1})^{\az_1}\cdots
(\frac{\partial}{\partial x_n})^{\az_n},$ and
$x^\az:=x_1^{\az_1}\cdots x_n^{\az_n}.$
We denote by $C$ a \emph{positive constant}
which is independent of the main parameters, but may vary from
line to line. We use $C_{(\az,\dots)}$ to denote a positive
constant depending on the indicated parameters $\az,\, \dots$.
The symbol $f\ls g$ means $f\leq  Cg$. If $f\ls g$ and $g\ls f$,
we then write $f\sim g$.
If $f \le C g$ and $g=h$ or $g\le h$, we then write $f\ls g = h$ or $f\ls g\le h$.
For any $q\in [1,\infty]$, we denote by $q'$ its \emph{conjugate
index}, that is,
$1/q+1/q'=1$. For any $x\in\rn$,
we denote by $|x|$ the $n$-dimensional
\emph{Euclidean metric} of $x$.
If $E$ is a subset of $\rn$,
we denote by ${\mathbf{1}}_E$ its \emph{characteristic function}
and by $E^\complement$ the set $\rn\backslash E$.
For any $r\in(0,\fz)$ and $x\in\rn$, we denote by $B(x,r)$
the ball centered at $x$ with the radius $r$, that is,
$B(x,r):=\{y\in\rn:\ |x-y|<r\}.$
For any ball $B$, we use $x_B$ to denote its center and $r_B$
its radius, and denote by $\lz B$ for any $\lz\in(0,\fz)$ the
ball concentric with $B$ having the radius $\lz r_B$.
We also use $\epsilon\to 0^{+}$
to denote $\epsilon\in(0,\infty)$ and
$\epsilon \to 0$. Let $X$ and $Y$ be two
normed vector spaces, respectively, with the norm
$\|\cdot\|_X$ and the norm $\|\cdot\|_Y$;
then we use $X\hookrightarrow Y$ to denote
$X\subset Y$ and there exists a positive constant
$C$ such that, for any $f\in X$,
$\|f\|_Y \le C \|f\|_X.$
For any measurable funtion $f$ on $\rn$ and any measurable set
$E\subset\rn$ with $|E|\in(0,\infty)$, let $\fint_E f(x)\,dx
:=\frac{1}{|E|}\int_E f(x)\,dx$.
At last, when we prove a theorem or the like,
we always use the same symbols in the wanted proved
theorem or the like.
\section{Preliminaries\label{s1}}

In this section, we first recall some notation and concepts
on dilations (see, for instance \cite{Bownik,ho03}) as well as ball quasi-Banach
function spaces (see, for instance,  {\cite{shyy17,wyy,wyyz,yyy20,zwyy}}).
We begin with recalling the concept of the expansive matrix from \cite{Bownik}.

\begin{definition}\label{dilation}
A real $ n\times n $ matrix $ A $ is called an
{\it expansive matrix} (shortly, a {\it dilation}) if
\begin{equation*}
\min_{\lz \in \sz(A)} |\lz|>1,
\end{equation*}
here and thereafter, $ \sz (A) $ denotes the {\it set of all eigenvalues of $ A $}.
\end{definition}

Let $A$ be a dilation and
\begin{align}\label{2.14.x1}
b:=|\det A|,
\end{align}
where $\det A$ denotes the determinant of $A$.
Then it follows from \cite[p.\ 6, (2.7)]{Bownik}
that $b\in(1,\infty)$. By the fact that there exists an open and symmetry
ellipsoid $\Delta$, with $ |\Delta|=1 $, and an $ r\in(1,\fz) $ such that
$ \Delta \subset r\Delta \subset A\Delta $ (see \cite[p.\,5, Lemma 2.2]{Bownik}),
we find that, for any $ k\in \zz $,
\begin{equation}\label{B_k}
B_k := A^k \Delta
\end{equation}
is open, $ B_k \subset r B_k \subset B_{k+1} $,
and $ |B_k|=b^k $. For any $ x \in\rn $ and $ k \in \zz $,
the ellipsoid $ x+B_k $ is called a {\it dilated ball}.
In what follows, we always let $ \CB $ be the set of all such dilated balls,
that is,

\begin{equation}\label{ball-B}
\CB:=\{x+B_k :\  x \in \rn,k \in \zz \}
\end{equation}
and let
\begin{equation}\label{tau}
\tau := \inf\left\{ l \in\zz :\  r^l \geq 2\right\}.
\end{equation}

Let $\lz_-,\lz_+ \in (0,\infty)$ satisfy that
\begin{equation}\label{2.21.x1}
1<\lz_- <\min\{ |\lz|:\ \lz \in\sigma(A)\}
\leq\max\{|\lz|:\ \lz\in\sigma(A)\}<\lz_+.
\end{equation}
We point out that, if $A$ is diagonalizable over $\rr$,
then we may let
$$\lz_-:=\min\{|\lz|:\ \lz\in\sz(A)\}\ \mbox{and}\
\lz_+:=\max\{|\lz|:\ \lz\in\sigma(A)\}.$$
Otherwise, we may choose them sufficiently close to
these equalities in accordance with what we need in our arguments.

The following definition of the homogeneous quasi-norm is just
\cite[p.\,6,\ Definition 2.3]{Bownik}.

\begin{definition}\label{quasi-norm}
A {\it homogeneous quasi-norm}, associated with a dilation $A$,
is a measurable mapping $\varrho:\ \rn\rightarrow[0,\infty)$ such that
\begin{enumerate}
\item[{\rm(i)}] $\varrho (x)=0\Longleftrightarrow x=\bf{0}$,
where $\bf{0}$ denotes the origin of $\rn$;

\item[{\rm(ii)}] $\varrho(Ax)=b\varrho(x)$ for any $x\in\rn$;

\item[{\rm(iii)}] there exists an $A_0\in[1,\infty)$ such that,
for any $x,y\in\rn$,
$$\varrho(x+y)\leq A_0\,[\varrho(x)+\varrho(y)].$$
\end{enumerate}
\end{definition}

In the standard Euclidean space case,
let $A:=2\,I_{n\times n}$ and, for any $x\in\rn$, $\varrho(x):=|x|^n$.
Then $\varrho$ is an example of homogeneous quasi-norms
associated with $A$ on $\rn$. Here and thereafter,
$I_{n\times n}$ always denotes the $n\times n$ \emph{unit
matrix} and $|\cdot|$ the \emph{Euclidean norm} in $\rn$.

For a fixed dilation $A$, by \cite[p.\,6,\ Lemma 2.4]{Bownik},
we define the following quasi-norm
which is used throughout this article.

\begin{definition}\label{def-shqn}
Define the {\it step homogeneous quasi-norm $\rho$} on $\rn$,
associated with the dilation $A$, by setting
\begin{equation*}
\rho(x):= \left\lbrace
\begin{aligned}
&b^k\  &&{\rm if}\  x \in B_{k+1}\backslash B_k,\\
&0\  &&{\rm if}\  x=\bf{0},
\end{aligned}
\right.
\end{equation*}
where $b$ is the same as in \eqref{2.14.x1} and,
for any $k\in\zz$, $B_k$ the same as in \eqref{B_k}.
\end{definition}

Then $(\rn,\rho,dx)$ is a space of homogeneous type
in the sense of Coifman and Weiss \cite{CWhomo},
where $dx$ denotes the $n$-dimensional Lebesgue measure.
For more studies on the real-variable theory of function
spaces over spaces of homogeneous type,
we refer the reader to \cite{bdl18,bdl20,bl11,lj10,lj11,lj13}.

Throughout this article, we always let $A$ be a dilation
in Definition \ref{dilation}, $b$ the same as in \eqref{2.14.x1},
$\rho$ the step homogeneous quasi-norm in Definition \ref{def-shqn},
$\CB$ the set of all dilated balls in \eqref{ball-B},
$\mathscr M(\rn)$ the {\it set of all measurable functions} on $\rn$,
and $B_k$ for any $k\in\zz$ the same as in \eqref{B_k}.
Now, we recall the definition of ball quasi-norm
Banach function spaces (see \cite{shyy17}).

\begin{definition}\label{BQBFS}
A quasi-normed linear space $X\subset\mathscr M(\rn)$,
equipped with a quasi-norm $\|\cdot\|$ which
makes sense for the whole $\mathscr M(\rn)$, is called
a \emph{ball quasi-Banach function space} if it satisfies
\begin{enumerate}
\item[{\rm(i)}] for any $f\in\mathscr M(\rn)$, $\|f\|_X=0$
implies that $f=0$ almost everywhere;

\item[{\rm(ii)}] for any $f,g\in\mathscr M(\rn)$, $|g|\le|f|$ almost everywhere
implies that $\|g\|_X\le\|f\|_X$;

\item[{\rm(iii)}] for any $\{f_m\}_{m\in\nn}\subset\mathscr M(\rn)$
and $f\in\mathscr M(\rn)$, $0\le f_m\uparrow f$ as $m\to\infty$
almost everywhere implies that
$\|f_m\|_X\uparrow\|f\|_X$ as $m\to\infty$;

\item[{\rm(iv)}] $\one_B \in X$ for any dilated ball $B \in \CB$.
\end{enumerate}

Moreover, a {ball quasi-Banach function} space $X$ is
called a \emph{ball Banach function space} if it satisfies:
\begin{enumerate}
\item[{\rm(v)}] for any $f,g\in X$, $\|f+g\|_X\leq\|f\|_X+\|g\|_X$;

\item[{\rm(vi)}] for any given dilated ball $B\in\CB$,
there exists a positive constant $C_{(B)}$ such that,
for any $f\in X$,
\begin{equation*}
\int_B |f(x)|\,dx \leq C_{(B)}\|f\|_X.
\end{equation*}
\end{enumerate}
\end{definition}

\begin{remark}\label{s1r1}
\begin{enumerate}
\item[{\rm(i)}]
As was mentioned in \cite[Remark 2.5(i)]{wyy22},
if $f\in\mathscr M(\rn)$,
then $\|f\|_X=0$ if and only if $f=0$ almost everywhere;
if $f,g\in\mathscr M(\rn)$
and $f=g$ almost everywhere, then $\|f\|_X\sim\|g\|_X$
with the positive equivalence constants independent of both $f$ and $g$.

\item[{\rm(ii)}]
As was mentioned in \cite[Remark 2.5(ii)]{wyy22},
if we replace any dilated ball $B\in\CB$ in Definition \ref{BQBFS}
by any bounded measurable set $E$ or by any ball
$B(x,r)$ with $x\in\rn$ and $r\in(0,\fz)$,
we obtain its another equivalent formulation.

\item[{\rm(iii)}]
By \cite[Theorem 2]{dfmn21}, we find that both
(ii) and (iii) of Definition \ref{BQBFS} imply that
any ball quasi-Banach function space is complete.
\end{enumerate}
\end{remark}

Now, we recall the concepts of the $p$-convexification and
the concavity of ball quasi-Banach function spaces,
which is a part of \cite[Definition 2.6]{shyy17}.

\begin{definition}\label{Debf}
Let $X$ be a ball quasi-Banach function space and $p\in(0,\infty)$.
\begin{enumerate}
\item[{\rm(i)}] The \emph{$p$-convexification} $X^p$ of $X$
is defined by setting
$$X^p:=\lf\{f\in\mathscr M(\rn):\ |f|^p\in X\r\}$$
equipped with the \emph{quasi-norm} $\|f\|_{X^p}:=\||f|^p\|_X^{1/p}$.

\item[{\rm(ii)}] The space $X$ is said to be
\emph{concave} if there exists a positive constant
$C$ such that, for any $\{f_k\}_{k\in{\mathbb N}}\subset \mathscr M(\rn)$,
$$\sum_{k=1}^{{\infty}}\|f_k\|_{X}
\le C\left\|\sum_{k=1}^{{\infty}}|f_k|\right\|_{X}.$$
In particular, when $C=1$, $X$ is said to be
\emph{strictly concave}.
\end{enumerate}
\end{definition}

The associate space $X'$ of any given ball
Banach function space $X$ is defined as follows;
see \cite[Chapter 1, Section 2]{bs88} or \cite[p.\,9]{shyy17}.

\begin{definition}
For any given ball Banach function space $X$,
its \emph{associate space} (also called the
\emph{K\"othe dual space}) $X'$ is defined by setting
\begin{equation*}
X':=\lf\{f\in\mathscr M(\rn):\ \|f\|_{X'}
:=\sup_{g\in X,\ \|g\|_X=1}\lf\|fg\r\|_{L^1(\rn)}<\infty\r\},
\end{equation*}
where $\|\cdot\|_{X'}$ is called the \emph{associate norm} of $\|\cdot\|_X$.
\end{definition}

\begin{remark}\label{bbf}
From \cite[Proposition 2.3]{shyy17}, we deduce that, if $X$ is a ball
Banach function space, then its associate space $X'$ is also a ball
Banach function space.
\end{remark}

Next, we recall the concept of absolutely continuous quasi-norms
of $X$ as follows (see \cite[Definition 3.2]{wyy} for the standard
Euclidean space case and \cite[Definition 6.1]{yhyy} for the case
of spaces of homogeneous type).

\begin{definition}
Let $X$ be a ball quasi-Banach function space.
A function $f\in X$ is said to have an {\it absolutely
continuous quasi-norm} in $X$ if $\|f\one_{E_j}\|_X\downarrow0$
whenever $\{E_j\}_{j=1}^\infty$ is a sequence of measurable
sets satisfying $E_j\supset E_{j+1}$ for any $j\in\nn$
and $\bigcap_{j=1}^\infty E_j=\emptyset$. Moreover,
$X$ is said to have an {\it absolutely continuous quasi-norm} if,
for any $f\in X$, $f$ has an absolutely continuous quasi-norm in $X$.
\end{definition}

Now, we recall the concept of the Hardy--Littlewood maximal operator.
Let $L_{\rm loc}^1(\rn)$ denote the
{\it set of all locally integrable functions} on $\rn$.
Recall that the {\it Hardy--Littlewood maximal operator}
$\cm(f)$ of $f \in L_{\rm loc}^1(\rn)$ is defined by setting,
for any $x \in \rn$,
\begin{align*}
\cm(f)(x):=&\,\sup_{k\in\zz}\sup_{y\in x+B_k}
\fint_{y+B_k} |f(z)|\,dz
=\,\sup_{x\in B\in\CB}\fint_B|f(z)|\,dz,
\end{align*}
where $\CB$ is the same as in {\eqref{ball-B}} and the last
supremum is taken over all balls $B\in\CB$.
For any given $\alpha \in (0,\infty)$,
the {\it powered Hardy--Littlewood maximal operator}
$\cm^{(\alpha)}$ is defined by setting,
for any $f \in L^1_{\rm loc}(\rn)$ and $x \in \rn$,
\begin{equation*}
\cm^{(\alpha)}(f)(x)
:=\lf\{\cm\lf(|f|^\alpha\r)(x)\r\}^{\frac{1}{\alpha}}.
\end{equation*}

Throughout this article, we also need the following two
fundamental assumptions about
the boundedness of $\cm$ on the given ball quasi-Banach
function space and its associate space.

\begin{assumption}\label{Assum-1}
Let $X$ be a ball quasi-Banach function space.
Assume that there exists a $p_- \in (0,\infty)$ such that,
for any $p \in (0,p_-)$ and $u \in (1,\infty)$,
there exists a positive constant $C$, depending on both $p$ and $u$,
such that, for any $\{f_k\}_{k=1}^\infty\subset\mathscr M(\rn)$,
\begin{equation*}
\left\| \left\{\sum_{k=1}^\infty\lf[\cm\lf(f_k\r)\r]^u\right\}^
{\frac{1}{u}} \right\|_{X^{\frac{1}{p}}}
\leq C\left\|\left\{\sum_{k=1}^{\infty}|f_k|^u\right\}^
{\frac{1}{u}}\right\|_{X^\frac{1}{p}}.
\end{equation*}
\end{assumption}

In what follows, for any given $p_- \in (0,\infty)$, we always let
\begin{equation} \label{underp}
\underline{p}:=\min \{p_-,1\}.
\end{equation}

\begin{assumption}\label{Assum-2}
Let $p_- \in (0,\infty)$ and $X$ be a
ball quasi-Banach function space.
Assume that there exists a $\theta_0\in(0,\underline{p})$,
with $\underline{p}$ the same as in \eqref{underp},
and a $p_0 \in (\theta_0,\infty)$ such that
$X^{{1}/{\theta_0}}$ is a ball Banach function space and,
for any $f\in(X^{{1}/{\theta_0}})'$,
\begin{equation*}
\left\|\cm^{(({p_0}/{\theta_0})')}(f)\right\|_{(X^{{1}/{\theta_0}})'}
\leq C\|f\|_{(X^{{1}/{\theta_0}})'},
\end{equation*}
where $C$ is a positive constant, independent of $f$, and
$\frac{1}{p_0/\theta_0}+\frac{1}{(p_0/\theta_0)'}=1$.
\end{assumption}
Next, recall that a {\it Schwartz function} is a function
$\varphi \in C^\infty(\rn)$ satisfying that, for any
$k\in\zz_+$ and any multi-index $\alpha\in\zz_+^n$,
\begin{equation*}
\|\varphi\|_{\alpha,k} := \sup_{x \in \rn}
[\rho (x)]^k|\partial^\alpha\varphi(x)|<\infty.
\end{equation*}
Denote by $ \cs(\rn) $ the {\it set of all Schwartz functions},
equipped with the topology determined by
$\{\|\cdot\|_{\alpha,k} \}_{\alpha\in\zz_+^n,k \in \zz_+}$.
Then $\cs'(\rn)$ is defined to be the {\it dual space} of $\cs(\rn)$,
equipped with the weak-$\ast$ topology. For any $N \in \zz_+$,
let
\begin{equation*}
\cs_N(\rn):=\lf\{\varphi\in\cs(\rn):\
\|\varphi\|_{\alpha,k}\leq1,|\alpha|\leq N,k\leq N\r\},
\end{equation*}
equivalently,
\begin{align*}
&\varphi\in\cs_N(\rn)\\
&\quad\Longleftrightarrow\ \|\varphi\|_{\cs_N(\rn)}
:=\sup_{|\alpha|\leq N}
\sup_{x\in\rn}\max\{1,[\rho(x)]^N\}
|\partial^\alpha\varphi(x)|\leq 1.
\end{align*}

In what follows, for any $\varphi\in\cs(\rn)$ and
$k \in \zz$, let $\varphi_k(\cdot):=b^{-k}\varphi(A^{-k}\cdot)$.
\begin{definition}
Let $\varphi\in\cs(\rn)$ and $f\in\cs'(\rn)$.
The {\it non-tangential maximal function} $M_\varphi (f)$
with respect to $\varphi$ is defined by setting,
for any $x\in\rn$,
\begin{equation*}
M_\varphi(f)(x):=\sup_{k\in\zz,\,y\in x+B_k}\lf|f\ast\varphi_k(y)\r|.
\end{equation*}
Moreover, for any given $ N\in\nn$, the {\it non-tangential grand
maximal function} $M_N(f)$ is defined by setting, for any $x\in\rn$,
\begin{equation}\label{M_N}
M_N(f)(x):=\sup_{\varphi\in\cs_N(\rn)}M_\varphi(f)(x).
\end{equation}
\end{definition}

\section{Duality between $H_{X}^{A}(\rn)$ and
$\mathcal{L}_{X,q',d,\tz_0}^{A}(\rn)$\label{s2}}

In this section, we provide a description of the dual space of the anisotropic
Hardy space $H_{X}^{A}(\rn)$ associated with ball quasi-Banach function space $X$.
This description is a
consequence of the definition of the anisotropic ball Campanato-type function space,
the atomic
and the finite atomic characterizations of $H_{X}^{A}(\rn)$ from \cite{wyy22},
as well as some basic tools from functional analysis.
To state the dual theorem, we first present the definition of
$H_{X}^{A}(\rn)$ from \cite{wyy22} as follows. In what follows,
for any $\az\in\rr$, we denote by
$\lfloor\az\rfloor$ the largest integer not greater than $\az$.

\begin{definition}\label{HXA}
Let $A$ be a dilation and $X$ a ball quasi-Banach function space
satisfying both Assumption \ref{Assum-1} with $p_-\in(0,\infty)$
and Assumption \ref{Assum-2} with the same
$p_-$, $\theta_0\in(0,\unp)$, and $p_0\in(\theta_0,\infty)$,
where $\unp$ is the same as in \eqref{underp}. Assume that
\begin{align}\label{3.14.x1}
N\in\nn\cap\lf[\lf\lfloor\lf(\frac{1}{\theta_0}-1\r)
\frac{\ln b}{\ln(\lz_-)}\r\rfloor+2,\infty\r).
\end{align}
The {\it anisotropic Hardy space $H_{X,\,N}^A(\rn)$},
associated with both $A$ and $X$, is defined by setting
\begin{equation*}
H_{X,\,N}^A(\rn):=\lf\{f\in\cs'(\rn):\ \|M_N(f)\|_X<\infty\r\},
\end{equation*}
where $M_N(f)$ is the same as in \eqref{M_N}. Moreover,
for any $f\in H_{X,\,N}^A(\rn)$, let
$$\|f\|_{H_{X,\,N}^A(\rn)}:=\lf\|M_N(f)\r\|_X.$$
\end{definition}

Let $A$ be a dilation and $X$ the same as in Definition \ref{HXA}.
In the remainder of this article, we always let
\begin{align}\label{NXA}
N_{X,\,A}:=\lf\lfloor\lf(\frac{1}{\theta_0}-1\r)
\frac{\ln b}{\ln(\lz_-)}\r\rfloor+2.
\end{align}

\begin{remark}\label{s2r1}
\begin{enumerate}
\item[{\rm(i)}]
As was mentioned in \cite[Remark 2.17(i)]{wyy22},
the space $H_{X,\,N}^A(\rn)$ is independent
of the choice of $N$ as long as $N\in\nn\cap[N_{X,\,A},\infty)$.
In what follows, when $N\in\nn\cap[N_{X,\,A},\infty)$,
we write $H_{X,\,N}^A(\rn)$ simply by $H_{X}^A(\rn)$.

\item[{\rm(ii)}]
We point out that, if $A:=2\,I_{n\times n}$,
then $H_X^A(\rn)$ coincides with $H_X(\rn)$
which was introduced by Sawano et al. in \cite{shyy17}.
\end{enumerate}
\end{remark}
In what follows, for any $d\in\zz_+$, $\mathcal{P}_d(\rn)$
denotes the set of all the polynomials on $\rn$
with degree not greater than $d$; for any ball $B\in\CB$ and
any
locally integrable function $g$ on $\rn$,
we use $P^d_Bg$ to denote the \emph{minimizing polynomial} of
$g$ with degree not greater than $d$, which means
that $P^d_Bg$ is the unique polynomial $f\in\mathcal{P}_d(\rn)$
such that, for any $h\in\mathcal{P}_d(\rn)$,
$$\int_{B}[g(x)-f(x)]h(x)\,dx=0.$$

Next, we introduce the anisotropic ball Campanato-type function space
associated with the ball quasi-Banach function space.
In what follows, we use $L_{\loc}^q(\rn)$ to denote the
\emph{set of all $q$-order locally integrable functions} on
$\rn$.

\begin{definition}\label{LAXqds}
Let $A$ be a dilation, $X$ a ball quasi-Banach function space, $q\in[1,{\infty})$,
$d\in\zz_+$, and $s\in(0,\infty)$.
Then the \emph{anisotropic ball Campanato-type function space}
$\mathcal{L}_{X,q,d,s}^{A}(\rn)$, associated with $X$, is defined to be
the set of all the $f\in L^q_{\rm loc}({{\rr}^n})$ such that
\begin{align*}
\|f\|_{\mathcal{L}_{X,q,d,s}^{A}(\rn)}
&:=\sup
\lf\|\lf\{\sum_{i=1}^m
\lf[\frac{{\lambda}_i}{\|{\one}_{B^{(i)}}\|_X}\r]^{s}
{\one}_{B^{(i)}}\r\}^{\frac1{s}}\r\|_{X}^{-1}
\\
&\quad\times\sum_{j=1}^m\frac{{\lambda}_j|B^{(j)}|}{\|{\one}_{B^{(j)}}
\|_{X}}
\lf[
\fint_{B^{(j)}}\lf|f(x)-P^d_{B^{(j)}}f(x)\r|^q \,dx\r]^\frac1q
\end{align*}
is finite, where the supremum is taken over all
$m\in\nn$, $\{B^{(j)}\}_{j=1}^m\subset \CB$, and
$\{\lambda_j\}_{j=1}^m\subset[0,\infty)$ with $\sum_{j=1}^m\lambda_j\neq0$.
\end{definition}

\begin{remark}\label{s2r2}
Let $A$, $X$, $q$, $d$, and $s$ be the same
as in Definition \ref{LAXqds}.
\begin{enumerate}
\item[{\rm(i)}]
Obviously, $\mathcal{P}_d(\rn)\subset \mathcal{L}_{X,q,d,s}^{A}({{\rr}^n})$.
Indeed, $\|f\|_{\mathcal{L}_{X,q,d,s}^{A}(\rn)}=0$ if and only if
$f\in\mathcal{P}_d(\rn)$. Throughout this
article, we always identify $f\in \mathcal{L}_{X,q,d,s}^{A}(\rn)$
with $\{f+P:\ P\in\mathcal{P}_d(\rn)\}$.
\item[{\rm(ii)}]
For any $f\in L^q_{\rm loc}({{\rr}^n})$, define
\begin{align*}
\||f\||_{\mathcal{L}_{X,q,d,s}^{A}(\rn)}
:&=\sup\inf
\lf\|\lf\{\sum_{i=1}^m
\lf[\frac{{\lambda}_i}{\|{\one}_{B^{(i)}}\|_X}\r]^{s}
{\one}_{B^{(i)}}\r\}^{\frac1{s}}\r\|_{X}^{-1}
\\
&\quad\times\sum_{j=1}^m
\frac{{\lambda}_j|B^{(j)}|}{\|{\one}_{B^{(j)}}
\|_{X}}\lf[
\fint_{B^{(j)}}\lf|f(x)-P(x)\r|^q \,dx\r]^\frac1q,
\end{align*}
where the supremum is taken the same way as in Definition \ref{LAXqds}
and the infimum is taken over all $P\in\mathcal{P}_d(\rn)$.
Then, similarly to the proof of \cite[Lemma 2.6]{yyy20} with using
\cite[p.\,49, (8.9)]{Bownik} instead of \cite[Lemma 2.5]{yyy20},
we easily find that $\||\cdot\||_{\mathcal{L}_{X,q,d,s}^{A}(\rn)}$
is an equivalent quasi-norm of $\mathcal{L}_{X,q,d,s}^{A}({{\rr}^n})$.
\end{enumerate}
\end{remark}

Moreover, for the anisotropic ball Campanato-type function space
$\mathcal{L}_{X,q,d,s}^{A}(\rn)$,
we have the following equivalent quasi-norm.

\begin{proposition}\label{s2p1}
Let $A$, $X$, $q$, $d$, and $s$ be the same
as in Definition \ref{LAXqds}.
For any $f\in L^q_{\rm loc}({{\rr}^n})$, define
\begin{align*}
\widetilde{\|f\|}_{\mathcal{L}_{X,q,d,s}^{A}(\rn)}
&:=\sup
\lf\|\lf\{\sum_{i\in\nn}
\lf[\frac{{\lambda}_i}{\|{\one}_{B^{(i)}}\|_X}\r]^{s}
{\one}_{B^{(i)}}\r\}^{\frac1{s}}\r\|_{X}^{-1}\\
&\quad\times\sum_{j\in\nn}\frac{{\lambda}_j|B^{(j)}|}{\|{\one}_{B^{(j)}}
\|_{X}}
\lf[
\fint_{B^{(j)}}\lf|f(x)-P^d_{B^{(j)}}f(x)\r|^q \,dx\r]^\frac1q,
\end{align*}
where the supremum is taken over all $\{B^{(j)}\}_{j\in\nn}\subset \CB$ and
$\{\lambda_j\}_{j\in\nn}\subset[0,\infty)$ satisfying that
\begin{align}\label{s2e1}
\lf\|\lf\{\sum_{j\in\nn}
\lf[\frac{{\lambda}_j}{\|{\one}_{B^{(j)}}\|_X}\r]^{s}
{\one}_{B^{(j)}}\r\}^{\frac1{s}}\r\|_{X}\in(0,\infty).
\end{align}
Then, for any $f\in L^q_{\rm loc}({{\rr}^n})$,
$$\widetilde{\|f\|}_{\mathcal{L}_{X,q,d,s}^{A}(\rn)}
=\|f\|_{\mathcal{L}_{X,q,d,s}^{A}(\rn)}.$$
\end{proposition}

\begin{proof}
Let $f\in L^q_{\rm loc}({{\rr}^n})$. Obviously,
\begin{align}\label{2.15.x1}
\|f\|_{\mathcal{L}_{X,q,d,s}^{A}(\rn)}
\le \widetilde{\|f\|}_{\mathcal{L}_{X,q,d,s}^{A}(\rn)}.
\end{align}

Conversely, let $\{B^{(j)}\}_{j\in\nn}\subset \CB$ and
$\{\lambda_j\}_{j\in\nn}\subset[0,\infty)$ satisfy \eqref{s2e1}.
From Definition \ref{BQBFS}(iii), it follows that
\begin{align*}
&\lim_{m\to\fz}
\lf\|\lf\{\sum_{i=1}^m
\lf[\frac{{\lambda}_i}{\|{\one}_{B^{(i)}}\|_X}\r]^{s}
{\one}_{B^{(i)}}\r\}^{\frac1{s}}\r\|_{X}^{-1}\\
&\qquad\times\sum_{j=1}^m\frac{{\lambda}_j|B^{(j)}|}{\|{\one}_{B^{(j)}}
\|_{X}}\lf[
\fint_{B^{(j)}}\lf|f(x)-P^d_{B^{(j)}}f(x)\r|^q \,dx\r]^\frac1q\\
&\quad=\lf\|\lf\{\sum_{i\in\nn}
\lf[\frac{{\lambda}_i}{\|{\one}_{B^{(i)}}\|_X}\r]^{s}
{\one}_{B^{(i)}}\r\}^{\frac1{s}}\r\|_{X}^{-1}\\
&\qquad\times\sum_{j\in\nn}\frac{{\lambda}_j|B^{(j)}|}{\|{\one}_{B^{(j)}}\|_{X}}
\lf[\fint_{B^{(j)}}
\lf|f(x)-P^d_{B^{(j)}}f(x)\r|^q \,dx\r]^\frac1q.
\end{align*}
Therefore, for any $\varepsilon\in(0,\fz)$,
there exists an $m_0\in\nn$ such that
$\sum_{j=1}^{m_0}\lambda_j\neq0$
and
\begin{align*}
&\lf\|\lf\{\sum_{i\in\nn}
\lf[\frac{{\lambda}_i}{\|{\one}_{B^{(i)}}\|_X}\r]^{s}
{\one}_{B^{(i)}}\r\}^{\frac1{s}}\r\|_{X}^{-1}\\
&\qquad\times\sum_{j\in\nn}\frac{{\lambda}_j|B^{(j)}|}{\|{\one}_{B^{(j)}}
\|_{X}}\lf[\fint_{B^{(j)}}\lf|f(x)-P^d_{B^{(j)}}f(x)\r|^q \,dx\r]^\frac1q\\
&\quad <\lf\|\lf\{\sum_{i=1}^{m_0}
\lf[\frac{{\lambda}_i}{\|{\one}_{B^{(i)}}\|_X}\r]^{s}
{\one}_{B^{(i)}}\r\}^{\frac1{s}}\r\|_{X}^{-1}\\
&\qquad\times\sum_{i=1}^{m_0}\frac{{\lambda}_j|B^{(j)}|}{\|{\one}_{B^{(j)}}
\|_{X}}\lf[\fint_{B^{(j)}}\lf|f(x)-P^d_{B^{(j)}}f(x)\r|^q \,dx\r]^\frac1q
+\varepsilon\\
&\quad\le\|f\|_{\mathcal{L}_{X,q,d,s}^{A}(\rn)}+\varepsilon,
\end{align*}
which, together with the arbitrariness of
both $\{B^{(j)}\}_{j\in\nn}\subset \CB$ and
$\{\lambda_j\}_{j\in\nn}\subset[0,\infty)$ satisfying \eqref{s2e1}
and $\varepsilon\in(0,\fz)$,
further implies that
$$\widetilde{\|f\|}_{\mathcal{L}_{X,q,d,s}^{A}(\rn)}
\le\|f\|_{\mathcal{L}_{X,q,d,s}^{A}(\rn)}.$$
This, combined with \eqref{2.15.x1},
then finishes the proof of Proposition \ref{s2p1}.
\end{proof}
Then we introduce another anisotropic ball Campanato-type function
space $\mathcal{L}_{X,q,d}^{A}(\rn)$ associated with
the ball quasi-Banach function space $X$.
\begin{definition}\label{cqdA}
Let $A$ be a dilation, $X$ a ball quasi-Banach function space,
$q\in[1,\fz)$, and $d\in\zz_+$. Then the \emph{Campanato space}
$\mathcal{L}_{X,q,d}^{A}(\rn)$, associated with both $A$ and $X$,
is defined to be
the set of all the $f\in L^q_{\rm loc}(\rn)$ such that
\begin{align*}
\|f\|_{\mathcal{L}_{X,q,d}^{A}(\rn)}
:=\sup_{B\in\CB}\frac{|B|}{\|\one_B\|_X}\lf\{
\fint_{B}\lf|f(x)-P^d_Bf(x)\r|^q\,dx\r\}^{\frac{1}{q}}
<\infty,
\end{align*}
where the supremum is taken over all balls $B\in\CB$
and $P^d_Bf$ denotes the minimizing polynomial of
$f$ with degree not greater than $d$.
\end{definition}
\begin{remark}\label{s2r3}
Let $A$, $X$, $q$, $d$, and $s$ be the same as in Definition \ref{LAXqds}.
\begin{enumerate}
\item[{\rm(i)}]
From Definitions \ref{LAXqds} and \ref{cqdA}, it immediately follows that
$\mathcal{L}_{X,q,d,s}^{A}({{\rr}^n})\subset\mathcal{L}_{X,q,d}^{A}(\rn)$
and this inclusion is continuous.

\item[{\rm(ii)}]
For any $f\in L^q_{\rm loc}({{\rr}^n})$, define
\begin{align*}
\||f\||_{\mathcal{L}_{X,q,d}^{A}(\rn)}
:=\sup_{B\in\CB}\inf_{P\in\mathcal{P}_d(\rn)}
\frac{|B|}{\|{\one}_{B}\|_{X}}\lf[
\fint_{B}\lf|f(x)-P(x)\r|^q \,dx\r]^\frac1q.
\end{align*}
Then, similarly to \cite[Lemma 2.6]{yyy20},
we find that $\||\cdot\||_{\mathcal{L}_{X,q,d}^{A}(\rn)}$
is an equivalent quasi-norm
of $\mathcal{L}_{X,q,d}^{A}({{\rr}^n})$.
\end{enumerate}
\end{remark}

Now, we give a basic inequality
which is used throughout this article.

\begin{lemma}\label{basicine}
Let $\{a_i\}_{i\in\nn}\subset[0,\infty)$. If $\alpha\in[1,\fz)$, then
\begin{equation*}
\lf(\sum_{i\in\nn}a_i\r)^\alpha\geq\sum_{i\in\nn}a_i^\alpha.
\end{equation*}
\end{lemma}

The following proposition shows that, if the ball quasi-Banach function space
$X$ is concave and $s\in(0,1]$, then the space
$\mathcal{L}_{X,q,d,s}^{A}({{\rr}^n})$ coincides with
$\mathcal{L}_{X,q,d}^{A}(\rn)$ introduced in Definition \ref{cqdA}.

\begin{proposition}\label{s2p2}
Let $X$ be a concave ball quasi-Banach function space,
$q\in[1,\fz)$, $d\in\zz_+$, and $s\in(0,1]$. Then
\begin{align}\label{2.15.x3}
\mathcal{L}_{X,q,d,s}^{A}({{\rr}^n})=\mathcal{L}_{X,q,d}^{A}(\rn)
\end{align}
with equivalent quasi-norms.
\end{proposition}

\begin{proof}
We first show that
\begin{align}\label{2.15.x2}
\mathcal{L}_{X,q,d}^{A}(\rn)\subset\mathcal{L}_{X,q,d,s}^{A}(\rn)
\end{align}
and the inclusion is continuous.
For this purpose, let $f\in
\mathcal{L}_{X,q,d}^{A}(\rn)$. Then, from the assumptions that
$X$ is concave, Definitions \ref{BQBFS}(ii) and \ref{Debf}(ii),
and $s\in(0,1]$, together with Lemma \ref{basicine},
we deduce that
\begin{align*}
\|f\|_{\mathcal{L}_{X,q,d,s}^{A}(\rn)}
&\ls \sup
\lf(\sum_{i=1}^m{\lambda}_i\r)^{-1}\sum_{j=1}^m
\frac{{\lambda}_j|B^{(j)}|}
{\|{\one}_{B^{(j)}}\|_{X}}\\
&\quad\times\lf[\fint_{B^{(j)}}\lf|
f(x)-P^d_{B^{(j)}}f(x)\r|^q \,dx\r]^\frac1q\\
&\leq \sup\lf(\sum_{i=1}^m\lambda_i\r)^{-1}
\sum_{j=1}^m{\lambda}_j
\|f\|_{\mathcal{L}_{X,q,d}^{A}(\rn)}
=\|f\|_{\mathcal{L}_{X,q,d}^{A}(\rn)},
\end{align*}
where the supremum is taken over all $m\in\nn$, $\{B^{(j)}\}_{j=1}^m\subset \CB$,
and $\{\lambda_j\}_{j=1}^m\subset[0,\infty)$
with $\sum_{j=1}^m\lambda_j\neq0$.
This further implies \eqref{2.15.x2}.
Combining \eqref{2.15.x2} and Remark \ref{s2r3}(i),
we obtain \eqref{2.15.x3}, which completes the proof of
Proposition \ref{s2p2}.	
\end{proof}

Then we establish the relation between $\mathcal{L}_{X,q,d,s}^{A}({{\rr}^n})$
and $H_X^A(\rn)$. To this end, we first recall the definitions of
the anisotropic $(X,q,d)$-atom and the anisotropic finite atomic
Hardy space $H_{X,\fin}^{A,q,d}(\rn)$ from \cite[Definitions 4.1 and 5.1]{wyy22}.

\begin{definition}\label{deffin}
Let $A$ be a dilation and $X$ a ball quasi-Banach function space
satisfying both Assumption \ref{Assum-1} with $p_-\in(0,\infty)$
and Assumption \ref{Assum-2} with the same
$p_-$, $\theta_0\in(0,\unp)$, and $p_0\in(\theta_0,\infty)$,
where $\unp$ is the same as in \eqref{underp}. Assume that
$N\in\nn\cap\lf[N_{X,\,A},\infty\r)$ with $N_{X,\,A}$ the same as in \eqref{NXA}.
Further assume that $q\in(\max\{p_0,1\},\infty]$ and
\begin{equation}\label{def-d}
d\in\lf[\left\lfloor\left(\frac{1}{\theta_0}-1\right)
\frac{\ln b}{\ln(\lambda_-)}\right\rfloor,\fz\r)\cap\zz_+.
\end{equation}
\begin{enumerate}
\item [{\rm(i)}]
An {\it anisotropic $(X,q,d)$-atom} $a$ is a
measurable function on $\rn$ satisfying that
\begin{enumerate}
\item [$\rm (i)_1$]
$\supp a:=\{x\in\rn:\ a(x)\neq0\}\subset B$,
where $B\in\CB$ and $\CB$ is the same as in \eqref{ball-B};
\item[$\rm (i)_2$]
$\|a\|_{L^q(\rn)}\leq |B|^{\frac{1}{q}}\|\one_B\|_X^{-1}$;
\item[$\rm (i)_3$]
$\int_{\rn} a(x)x^\gamma\,dx=0 $ for any
$\gamma\in\zz_+^n$ with $|\gamma|\leq d$,
here and thereafter, for any $\gamma:=\{\gamma_1,\ldots,\gamma_n\}\in\zz_+^n$,
$|\gamma|:=\gamma_1+\cdots+\gamma_n$ and
$x^\gamma:=x_1^{\gamma_1}\cdots x_n^{\gamma_n}$.
\end{enumerate}

\item[{\rm(ii)}]
The {\it anisotropic
finite atomic Hardy space} $H_{X,\fin}^{A,q,d}(\rn)$,
associated with both $A$ and $X$, is defined to be
the set of all the $f\in\cs'(\rn)$ satisfying that
there exists a $K\in\nn$, a sequence $\{\lambda_i\}_{i=1}^K\subset[0,\infty)$,
and a finite sequence $\{a_i\}_{i=1}^K$ of anisotropic $(X,q,d)$-atoms supported,
respectively, in $\{B^{(i)}\}_{i=1}^K\subset\CB$ such that
\begin{equation*}
f=\sum_{i=1}^K\lambda_ia_i.
\end{equation*}
Moreover, for any $f\in H_{X,\fin}^{A,q,d}(\rn)$, define
\begin{equation*}
\|f\|_{\Hfin(\rn)}:=\inf\left\|\left\lbrace\sum_{i=1}^K
\left[\frac{\lz_i\one_{B^{(i)}}}{\|\one_{B^{(i)}}\|_X}\right]^{\tz_0}\right\rbrace
^{\frac{1}{\tz_0}}\right\|_X,
\end{equation*}
where the infimum is taken over all the decompositions of $f$ as above.
\end{enumerate}
\end{definition}

Let $A$ be a dilation and $X$ the same as in Definition \ref{deffin}.
In the remainder of this article, we always let
\begin{align}\label{dxa}
d_{X,\,A}:=\lf\lfloor\lf(\frac{1}{\theta_0}-1\r)
\frac{\ln b}{\ln(\lz_-)}\r\rfloor.
\end{align}

To establish the dual theorem of $H_{X}^{A}(\rn)$,
we need its atomic and its finite atomic characterizations as follows,
which are simple corollaries
of \cite[Theorem 4.2 and Lemma 7.2]{wyy22}
and \cite[Theorem 5.4]{wyy22}, respectively.
\begin{lemma}\label{s2l1}
Let $A$, $X$, $q$, $d$, and $\tz_0$ be the same
as in Definition \ref{deffin}.
Further assume that $X$ has an absolutely continuous quasi-norm,
$\{a_j\}_{j\in\nn}$ is a sequence of anisotropic $(X,q,d)$-atoms supported,
respectively, in the balls $\{B^{(j)}\}_{j\in\nn}\subset\CB$ and
$\{{\lambda}_j\}_{j\in\nn}\subset[0,\infty)$ such that
$$\left\|\left\{\sum_{j\in\nn}
\left[\frac{{\lambda}_j}{\|{\one}_{B^{(j)}}\|_X}\right]^{\theta_0}
{\one}_{B^{(j)}}\right\}^{\frac1{\theta_0}}\right\|_{X}<\fz.$$
Then the series $f:=\sum_{j\in\nn}{\lambda}_ja_j$
converges in $H_{X}^{A}(\rn)$, $f\in H_{X}^{A}(\rn)$, and there
exists a positive constant $C$, independent of $f$, such that
$$\lf\|f\r\|_{H_{X}^{A}(\rn)}\le C\left\|\left\{\sum_{j\in\nn}
\left[\frac{{\lambda}_j}{\|{\one}_{B^{(j)}}\|_X}\right]^{\theta_0}
{\one}_{B^{(j)}}\right\}^{\frac1{\theta_0}}\right\|_{X}.$$
\end{lemma}
\begin{lemma}\label{finatomth}
Let $A$, $X$, $q$, $d$, $\tz_0$, and $p_0$
be the same as in Definition \ref{deffin}.
\begin{enumerate}
\item [{\rm (i)}]
If $q\in(\max\{p_0,1\},\infty)$,
then $\|\cdot\|_{\Hfin(\rn)}$ and $\|\cdot\|_{H_{X}^{A}(\rn)}$
are equivalent quasi-norms on $\Hfin(\rn)$ ;

\item [{\rm (ii)}]
$\|\cdot\|_{H_{X,{\rm fin}}^{A,\infty,d}(\rn)}$ and
$\|\cdot\|_{H_{X}^{A}(\rn)}$ are equivalent quasi-norms
on $H_{X,{\rm fin}}^{A,\infty,d}(\rn)\cap \mathcal{C}(\rn)$,
where $\mathcal{C}(\rn)$
denotes the set of all continuous functions on $\rn$.
\end{enumerate}
\end{lemma}

The following conclusion is also needed for establishing the dual theorem.
\begin{proposition}\label{atomch2}
Let $A$, $X$, and $d$ be the same as in Definition \ref{deffin}.
Then the set $H_{X,{\rm fin}}^{A,\infty,d}(\rn)\cap\mathcal{C}{(\rn)}$ is dense in
$H_{X}^A(\rn)$.
\end{proposition}
\begin{proof}
From \cite[Lemma 7.2]{wyy22}, it easily follows that
$H_{X,\fin}^{A,\infty,d}(\rn)$ is dense in
$H_{X}^{A}(\rn)$.
Thus,
to show that $H_{X,\fin}^{A,\infty,d}(\rn)\cap\mathcal{C}{(\rn)}$ is also
dense in
$H_{X}^{A}(\rn)$,
it suffices to prove that the set
$H_{X,\fin}^{A,\infty,d}(\rn)\cap\mathcal{C}{(\rn)}$ is dense in
$H_{X,\fin}^{A,\infty,d}(\rn)$ with the quasi-norm $\|\cdot\|_{H_{X}^{A}(\rn)}$.
To this end, we only need to show that, for any given anisotropic
$(X,\infty,d)$-atom $a$ supported in the anisotropic ball $B:=x_0+B_{i_0}$
with $x_0\in\rn$ and $i_0\in\zz$,
\begin{equation}\label{phi_ka}
\lim_{k\in(-\fz,0]\cap\zz,\,k\to-\infty}
\lf\|a-\varphi_{k}\ast a\r\|_{H_{X}^{A}(\rn)}=0,
\end{equation}
where $\varphi\in\mathcal{S}(\rn)$ satisfies
$\int_{\rn}\varphi(x)\,dx=1$ and
$\supp\varphi\subset B_{0}$.
Let $s\in(\max\{1,p_0\},\infty)$ with $p_0$ the same as in Definition \ref{HXA}.
Observe that, for any $k\in(-\fz,0]\cap\zz$,
$$
\frac{|B_{\max\{i_0,0\}+\tau}|^{\frac{1}{s}}(a-\varphi_k\ast a)}
{\|\mathbf1_{x_{0}+B_{\max\{i_0,0\}+\tau}}\|_{X}\|a-\varphi_k\ast a\|_{L^s(\rn)}}
$$
is an anisotropic $(X,s,d)$-atom supported in the anisotropic ball
$x_{0}+B_{\max\{i_0,0\}+\tau}$,
which, combined with Lemma \ref{s2l1}, further implies that
\begin{align*}
\lf\|a-\varphi_k\ast a\r\|_{H_{X}^{A}(\rn)}
&\lesssim
\frac{\|\mathbf1_{x_{0}+B_{\max\{i_0,0\}+\tau}}\|_{X}
\|a-\varphi_k\ast a\|_{L^s(\rn)}}
{|B_{\max\{i_0,0\}+\tau}|^{\frac{1}{s}}}\\
&\lesssim\|a-\varphi_k\ast a\|_{L^s(\rn)}.
\end{align*}
From this and \cite[p.15,\,Lemma 3.8]{Bownik},
we deduce \eqref{phi_ka},
which then completes the proof of Proposition \ref{atomch2}.
\end{proof}

The following technical lemma is
just \cite[p.\,49, (8.9)]{Bownik}
(see also \cite[Lemma 3.4]{lhyy20}).

\begin{lemma}\label{s3l1}
Let $f\in L^{1}_{\rm loc}(\rn)$, $d\in\zz_+$, and
$B$ be an anisotropic ball in $\CB$. Then there exists a positive
constant $C$, independent of both $f$ and $B$, such that
\begin{align*}
\sup_{x\in B}\left|P_{B}^df(x)\right|
\le {C}\fint_B|f(x)|\,dx.
\end{align*}
\end{lemma}

Now, we prove that the dual space of $H_X^A(\rn)$ is
$\mathcal{L}_{X,q',d,\tz_0}^{A}(\rn)$.
\begin{theorem}\label{s2t1}
Let $A$, $X$, $q$, $d$, and $\tz_0$ be the same
as in Definition \ref{deffin}.
Further assume that $X$ has an absolutely continuous quasi-norm.
Then the dual space of $H_X^A(\rn)$, denoted by $(H_X^A(\rn))^*$,
is $\mathcal{L}_{X,q',d,\tz_0}^{A}({{\rr}^n})$ with $1/q+1/q'=1$
in the following sense:
\begin{enumerate}
\item[{\rm (i)}] Let $g\in\mathcal{L}_{X,q',d,\tz_0}^{A}({{\rr}^n})$.
Then the linear functional
\begin{align}\label{s2t1e1}
L_g:\ f\rightarrow L_g(f):=\int_{{{\rr}^n}}f(x)g(x)\,dx,
\end{align}
initially defined for any $f\in H_{X,\fin}^{A,q,d}(\rn)$,
has a bounded extension to $H_X^A(\rn)$.

\item[{\rm (ii)}] Conversely, any continuous linear
functional on $H_X^A(\rn)$ arises as in \eqref{s2t1e1}
with a unique $g\in\mathcal{L}_{X,q',d,\tz_0}^{A}({{\rr}^n})$.
\end{enumerate}
Moreover,
$\|g\|_{\mathcal{L}_{X,q',d,\tz_0}^{A}({{\rr}^n})}
\sim\|L_g\|_{(H^A_X({{\rr}^n}))^*}$,
where the positive equivalence constants
are independent of $g$.
\end{theorem}
\begin{proof}
We first show (i) in the case $q\in(\max\{1,p_0\},\infty)$
with $p_0$ the same as in Definition \ref{deffin}.
To this end, let $g\in\mathcal{L}_{X,q',d,\tz_0}^{A}({{\rr}^n})$.
For any $f\in H_{X,\fin}^{A,q,d}(\rn)$,
by Definition \ref{deffin}, we know that there exists
a sequence $\{\lambda_j\}_{j=1}^m\subset[0,\infty)$ and
a sequence $\{a_j\}_{j=1}^m$ of anisotropic $(X,\ q,\ d)$-atoms supported,
respectively, in the balls
$\{B^{(j)}\}_{j=1}^m\subset\CB$ such that $f=\sum_{j=1}^m\lambda_ja_j$ and
$$
\lf\|\lf\{\sum_{j=1}^{m} \lf[\frac{\lambda_j}
{\|\one_{B^{(j)}}\|_{X}} \r]^{\tz_0}\one_{B^{(j)}} \r\}
^{\frac{1}{\tz_0}}\r\|_{X}
\sim\|f\|_{H_{X,\fin}^{A,q,d}(\rn)}.
$$
From these, the vanishing moments of $a_j$,
the H\"older inequality, the size condition of $a_j$,
Remark \ref{s2r2}(ii), and Lemma
\ref{finatomth}(i), it follows that
\begin{align}\label{s2t1e2}
|L_g(f)|&=\lf|\int_{\rn}f(x)g(x)\,dx\r|
\leq\sum_{j=1}^{m}\lambda_j\lf|\int_{B^{(j)}}a_j(x)g(x)\,dx\r|\\
&=\sum_{j=1}^{m}\lambda_j\inf_{P\in \cp_d(\rn)}
\lf|\int_{B^{(j)}}a_j(x)\lf[g(x)-P(x)\r]\,dx\r|\noz\\
&\leq\sum_{j=1}^{m}\lambda_j\|a_j\|_{L^q(\rn)}\inf_{P\in \cp_d(\rn)}
\lf[\int_{B^{(j)}}\lf|g(x)-P(x)\r|^{q'}\,dx\r]^{\frac{1}{q'}}\noz\\
&\le\sum_{j=1}^{m}\frac{\lambda_j|B^{(j)}|}{\|\one_{B^{(j)}}\|_{X}}
\inf_{P\in \cp_d(\rn)}\lf[
\fint_{B^{(j)}}|g(x)-P(x)|^{q'}\,dx\r]^{\frac{1}{q'}}\noz\\
&\ls\lf\|\lf\{\sum_{i=1}^{m}
\lf[\frac{{\lambda}_i}{\|{\one}_{B^{(i)}}\|_X}\r]^{\tz_0}
{\one}_{B^{(i)}}\r\}^{\frac1{\tz_0}}\r\|_{X}
\|g\|_{\mathcal{L}_{X,q',d,\tz_0}^{A}({{\rr}^n})}\noz\\
&\sim\|f\|_{H_{X,\fin}^{A,q,d}(\rn)}
\|g\|_{\mathcal{L}_{X,q',d,\tz_0}^{A}({{\rr}^n})}
\sim\|f\|_{H_X^A(\rn)}\|g\|_{\mathcal{L}_{X,q',d,\tz_0}^{A}({{\rr}^n})}.\noz
\end{align}
Moreover, by \cite[Lemma 7.2]{wyy22} and the assumption that
$X$ has an absolutely continuous quasi-norm, we find that
$H_{X,\fin}^{A,q,d}(\rn)$ is dense in $H_X^A(\rn)$.
This, together with \eqref{s2t1e2} and a standard density argument,
further implies that, when $q\in(\max\{1,p_0\},\infty)$,
(i) holds true and
$$\|L_g\|_{(H^A_X({{\rr}^n}))^*}
\ls\|g\|_{\mathcal{L}_{X,q',d,\tz_0}^{A}({{\rr}^n})}$$
with the implicit positive constant
independent of $g$.

We next prove (i) in the case $q=\fz$. Indeed, using
Proposition \ref{atomch2} and repeating the above proof for
any given
$q\in(\max\{1,p_0\},\infty)$, we then conclude that any
$g\in\mathcal{L}^{A}_{X,1,d,\tz_0}({{\rr}^n})$ induces a
bounded linear
functional on $H_X^A(\rn)$,
which is initially defined on
$H_{X,\fin}^{A,\infty,d}(\rn)\cap\mathcal{C}{(\rn)}$
and given by setting, for any
$\ell\in H_{X,\fin}^{A,\infty,d}(\rn)\cap\mathcal{C}{(\rn)}$,
\begin{equation}\label{s2t1e3}
L_g:\ \ell\mapsto\ L_g(\ell):=\int_{\rn}\ell(x)g(x)\,dx,
\end{equation}
and then
has a bounded linear
extension to $H_X^A(\rn)$.
Let $g\in\mathcal{L}_{X,1,d,\tz_0}^{A}({{\rr}^n})$. Thus,
it remains to show that, for any $f\in H_{X,\fin}^{A,\infty,d}(\rn)$,
\begin{equation}\label{s2t1e4}
L_g(f)=\int_{\rn}f(x)g(x)\,dx.
\end{equation}
To this end, suppose $f\in H_{X,\fin}^{A,\infty,d}(\rn)$
and $\supp f\subset x_0+B_{i_0}$ with $x_0\in\rn$ and $i_0\in\zz$.
Let $\varphi\in\mathcal{S}(\rn)$ satisfy
$\supp\varphi\subset B_{0}$ and $\int_{\rn}\varphi(x)\,dx=1$.
Letting $s\in(\max\{1,p_0\},\infty)$,
by the proof of Proposition \ref{atomch2},
we find that, for any $k\in(-\fz,0]\cap\zz$ and $f\in L^s(\rn)$,
\begin{align}\label{2.16.x1}
\varphi_k\ast f\in H_{X,\fin}^{A,\infty,d}(\rn)\cap\mathcal{C}{(\rn)}
\end{align}
and
\begin{equation}\label{s2t1e5}
\lim_{k\in(-\fz,0]\cap\zz,\,k\to-\infty}
\lf\|f-\varphi_k\ast f\r\|_{L^s(\rn)}=0.
\end{equation}
From this and the Riesz lemma (see, for instance, \cite[Theorem 2.30]{Folland}),
it follows that
there exists a subsequence
$\{k_h\}_{h\in\nn}\subset(-\fz,0]\cap\zz$ such that $\lim_{h\to\infty}k_h=-\infty$
and, for almost every $x\in\rn$,
\begin{align}
\lim_{h\to\infty}\varphi_{k_h}\ast f(x)=f(x).\noz
\end{align}
By \eqref{s2t1e5} and an argument similar to that used in
the proof of Proposition \ref{atomch2}, we conclude that
$\lim_{h\to\infty}\|f-\varphi_{k_h}\ast f\|_{H_X^A(\rn)}=0$,
which, combined with Lemma \ref{s2l1}, \eqref{2.16.x1},
\eqref{s2t1e3}, the fact that
$$\lf|\lf(\varphi_{k_h}\ast f\r)g\r|
\le\|f\|_{L^\infty(\rn)}\mathbf1_{x_0+B_{\max\{i_0,0\}+\tau}}|g|\in L^1(\rn),$$
and the Lebesgue dominated convergence theorem
(see, for instance, \cite[Theorem 2.24]{Folland}),
further implies that
\begin{align*}
L_g(f)&=\lim_{h\to\infty}L_g(\varphi_{k_h}\ast f)
=\lim_{h\to\infty}\int_{\rn}\varphi_{k_h}\ast f(x)g(x)\,dx\\
&=\int_{\rn}f(x)g(x)\,dx.
\end{align*}
This finishes the proof of \eqref{s2t1e4} and hence (i) in the case
$q=\infty$. Moreover, repeating the proof in \eqref{s2t1e2},
we obtain, for any $q\in(\max\{1,p_0\},\infty]$,
\begin{align}\label{4.14.x2}
    \|L_g\|_{(H^A_X({{\rr}^n}))^*}
\ls\|g\|_{\mathcal{L}_{X,q',d,\tz_0}^{A}({{\rr}^n})}
\end{align}
with the implicit positive constant independent of $g$.

We next show (ii). For this purpose,
let $\pi_B :\ L^1(B)\rightarrow \mathcal{P}_d(\rn)$, with $B\in\CB$, be
the natural projection
such that, for any $f\in L^1(B)$ and $Q\in\mathcal{P}_d(\rn)$,
\begin{align}\label{2.16.y3}
\int_{B}\pi_B(f)(x)Q(x)\,dx=\int_{B}f(x)Q(x)\,dx.
\end{align}

For any $q\in(\max\{1,p_0\},\infty]$ and any ball $B\in\CB$,
the \emph{closed subspace $L^q_0(B)$} of $L^q(B)$
is defined
by setting
\begin{align*}
L^q_0(B): = \lf\{f\in L^q(B):\ \pi_B(f)=0\ {\rm and}\
f\neq0\ \mathrm{almost\ everywhere}\r\},
\end{align*}
where $L^q(B)$ is the subspace
of $L^q(\rn)$ consisting
of all the measurable functions on $\rn$ vanishing outside $B$.
Therefore,
$$\frac{|B|^{\frac1q}}{\|\one_B\|_{X}}\|f\|_{L^q(\rn)}^{-1}f$$
is an anisotropic $(X,\ q,\ d)$-atom for any $f\in L^q_0(B)$. From this
and Lemma \ref{s2l1}, it follows that
\begin{align}\label{s2t1e6}
\lf\|\frac{|B|^{1/q }}{\|\one_B\|_{X}}\|f\|
_{L^q(\rn)}^{-1}f\r\|_{H_X^A(\rn)}\ls1.
\end{align}
Now, suppose $L \in(H_X^A(\rn))^*$. Then, by \eqref{s2t1e6},
we find that, for any $f\in L^q_0(B)$,
\begin{equation}\label{s2t1e7}
|L(f)| \le\|L\|_{(H_X^A(\rn))^*}\frac{\|\one_B\|_{X}}{|B|^{1/q}}\|f\|_{L^q(\rn)}.
\end{equation}
Therefore, $L$ provides a bounded linear functional on $L^q_0(B)$.
Thus, applying the Hahn--Banach theorem
(see, for instance, \cite[Theorem 5.6]{Folland}),
we find that there exists a linear functional
$L_B$, which extends $L$ to the whole space $L^q(B)$ without increasing its
norm.

When $q\in(\max\{1,p_0\},\infty)$, by the duality $(L^q(B))^* = L^{q'}(B)$,
we find that there exists an $h_B \in L^{q'}(B)\subset L^{1}(B)
$ such that, for any $f\in L^q_0(B)$,
\begin{align}\label{s2t1e8}
L(f)=L_B(f)=\int_{B}f(x)h_B(x)\,dx.
\end{align}
In the case $q=\infty$, let $\widetilde{q}\in(\max\{1,p_0\},\infty)$.
Then there exists an
$h_B \in L^{\widetilde{q}'}(B)\subset L^{1}(B)$
such that, for any $f\in L^\infty_0(B)
\subset L^{\widetilde{q}}(B)$, $L(f)=\int_{B}f(x)h_B(x)\,dx$.
Altogether, we find that,
for any $q\in(\max\{1,p_0\},\infty]$,
there exists an $h_B \in L^{q'}(B)$ such that, for any $f\in L^q_0(B)$,
\begin{align}\label{2.16.y2}
L(f)=\int_{B}f(x)h_B(x)\,dx.
\end{align}

Next, we prove that such an $h_B\in L^{q'}(B)$ is unique in the sense
of modulo $\cp_d(\rn)$. Indeed, assume that  $\wz{h_B}$ is
another element of $L^{q'}(B)$ such that
\begin{align}\label{2.16.y1}
L(f)=\int_{B}f(x)\wz{h_B}(x)\,dx
\end{align}
for any $f\in L^q_0(B)$.
Then, from \eqref{2.16.y2}, \eqref{2.16.y1}, and \eqref{2.16.y3},
we infer that, for any $f\in L^\infty(B)$, $f-\pi_B(f)\in L^\infty_0(B)$ and
\begin{align*}
0&=\int_B\lf[f(x)-\pi_B(f)(x)\r]\lf[h_B(x)-\wz{h_B}(x)\r]\,dx\\
&=\int_Bf(x)\lf[h_B(x)-\wz{h_B}(x)\r]\,dx
-\int_B\pi_B(f)(x)\pi_B(h_B - \wz{h_B})(x)\,dx\\
&=\int_Bf(x)\lf[h_B(x)-\wz{h_B}(x)\r]\,dx
-\int_Bf(x)\pi_B(h_B-\wz{h_B})(x)\,dx\\
&=\int_B f(x)\lf[h_B(x)-\wz{h_B}(x)-\pi_B(h_B - \wz{h_B})(x)\r]\,dx.
\end{align*}
The arbitrariness of $f$ further implies that
$h_B(x)-\wz{h_B}(x) = \pi_B(h_B - \wz{h_B})(x)$ for almost every $x\in B$.
Therefore, after changing values
of $h_B$ (or $\wz{h_B}$) on a set of measure zero,
we have $h_B - \wz{h_B}\in \cp_d(\rn)$.
Thus, for any $q\in(\max\{1,p_0\},\infty]$ and $f\in L^{q}_0(B)$, there exists
a unique $h_B\in L^{q'}(B)/{\cp_d(B)}$ such that \eqref{s2t1e8} holds true.

For any $j\in\rn$ and $f\in L^q_0(B_j)$, let $g_j$ be the unique
element of $L^{q'}(B_j)/{\cp_d(B_j)}$ such that
$$L(f) =\int_{B_j}f(x)g_j(x)\,dx.$$
Therefore, we can define a local $L^{q'}(\rn)$ function $g$
by setting $g(x):= g_j(x)$ whenever $x\in B_j$.
Assume that $f$ is a finite linear combination of anisotropic $(X,\ q,\ d)$-atoms.
It is easy to show that there exists an $x_0\in\rn$ and a $k_0\in\zz$
such that $\supp f\st x_0+B_{k_0}$. Let
$$j_0:=\frac{\ln A_0+\ln[b^{k_0-1}+\rho(x_0)]}{\ln b}+1.$$
Then, by Definition \ref{def-shqn}, we conclude that
$\supp f\st x_0+B_{k_0}\st B_{j_0}$. Thus, $f\in L^q_0(B_{j_0})$ and
$$L(f)=\int_{B_{j_0}}f(x)g_{j_0}(x)\,dx=\int_{\rn}f(x)g(x)\,dx.$$
From this and \eqref{s2t1e7}, we deduce that,
for any ball $B\in\CB$,
\begin{equation}\label{s2t1e9}
\|g\|_{(L^q_0(B))^*}
\leq\frac{\|\one_B\|_{X}}{|B|^{1/q}}
\|L\|_{(H_X^A(\rn))^*}.
\end{equation}
Moreover, it is known that
\begin{equation*}
\|g\|_{(L^q_0(B))^*}=\inf_{P\in \cp_d(\rn)}\|g-P\|_{L^{q'}(B)}
\end{equation*}
(see, for instance, \cite[p.\,52, (8.12)]{Bownik}), which, combined with
Remark \ref{s2r3}(ii) and \eqref{s2t1e9}, further implies that
\begin{equation}\label{4.14}
\|g\|_{\mathcal{L}_{X,q',d}^A({{\rr}^n})}\sim\sup_{B\in\CB}
\frac{|B|^{\frac{1}{q}}}{\|\one_{B}\|_{X}}\|g\|_{(L^q_0(B))^*}
\leq\|L\|_{(H_X^A(\rn))^*}.
\end{equation}
Thus, $g\in\mathcal{L}_{X,q',d}^{A}({{\rr}^n})$ and,
for any finite linear combination $f$ of anisotropic
$(X,\ q,\ d)$-atoms,
$$L(f) =\int_{\rn}f(x)g(x)\,dx.$$

Now, we show that $g\in\mathcal{L}_{X,q',d,\tz_0}^{A}({{\rr}^n})$
and $\|g\|_{\mathcal{L}_{X,q',d,\tz_0}^{A}({{\rr}^n})}\ls
\|L\|_{(H_X^A(\rn))^*}.$
To this end, for any $m\in\nn$, $\{B^{(j)}\}_{j=1}^m\subset \CB$, and
$\{\lambda_j\}_{j=1}^m\subset[0,\infty)$ with $\sum_{j=1}^m\lambda_j\neq0$,
let $h_j\in L^q(B^{(j)})$ with $\|h_j\|_{L^q(B^{(j)})}=1$ be such that
\begin{align}\label{s2t1e10}
&\lf[\int_{B^{(j)}}\lf|g(x)-P^d_{B^{(j)}}g(x)\r|^{q'} \,dx\r]^\frac1{q'}\\
&\quad=\int_{B^{(j)}}\lf[g(x)-P^d_{B^{(j)}}g(x)\r]h_j(x)\,dx\noz
\end{align}
and, for any $x\in\rn$, define
$$
a_j(x):=\frac{|B^{(j)}|^{\frac{1}{q}}
[h_j(x)-P^d_{B^{(j)}}h_j(x)]{\one}_{B^{(j)}}(x)}
{\|{\one}_{B^{(j)}}\|_{X}
\|h_j-P^d_{B^{(j)}}h_j\|_{L^q(B^{(j)})}}.
$$
Then it is easy to find that, for any $j\in\{1,\ldots,m\}$,
$a_j$ is an anisotropic $(X,\ q,\ d)$-atom. From this,
and Lemma \ref{s2l1}, it follows that
$\sum_{j=1}^m \lambda_j a_j\in H_X^A(\rn)$ and
\begin{align}\label{h_j}
\lf\|\sum_{j=1}^m{\lambda}_j a_j\r\|_{H_X^A(\rn)}
\ls\lf\|\lf\{\sum_{j=1}^m
\lf[\frac{{\lambda}_j}{\|{\one}_{B^{(j)}}\|_X}\r]^{\tz_0}
{\one}_{B^{(j)}}\r\}^{\frac1{\tz_0}}\r\|_{X}.
\end{align}
Moreover, by the Minkowski inequality,
the assumption that $\|h_j\|_{L^q(B^{(j)})}=1$,
Lemma \ref{s3l1}, and the H\"older inequality, we find that
\begin{align*}
\lf\|h_j-P^d_{B^{(j)}}h_j\r\|_{L^q(B^{(j)})}
&\leq\lf\|h_j\r\|_{L^q(B^{(j)})}+
\lf\|P^d_{B^{(j)}}h_j\r\|_{L^q(B^{(j)})}\\
&\ls1+\lf|B^{(j)}\r|^{\frac{1}{q}}\fint_{B^{(j)}}\lf|h_j(x)\r|\,dx\noz\\
&=1+\frac{1}{|B^{(j)}|^{\frac{1}{q'}}}\int_{B^{(j)}}\lf|h_j(x)\r|\,dx\noz\\
&\leq1+\lf\|h_j\r\|_{L^q(B^{(j)})}\ls1\noz.
\end{align*}
This, together with \eqref{s2t1e10},
the assumption that $L \in(H_X^A(\rn))^*$,
and \eqref{h_j},
further implies that
\begin{align*}
&\sum_{j=1}^m\frac{{\lambda}_j|B^{(j)}|}{\|{\one}_{B^{(j)}}\|_{X}}\lf[
\fint_{B^{(j)}}
\lf|g(x)-P^d_{B^{(j)}}g(x)\r|^{q'} \,dx\r]^\frac1{q'}\\
&\quad=\sum_{j=1}^m\frac{{\lambda}_j|B^{(j)}|^{\frac1q}}
{\|{\one}_{B^{(j)}}\|_{X}}\int_{B^{(j)}}
\lf[g(x)-P^d_{B^{(j)}}g(x)\r]h_j(x)\,dx\noz\\
&\quad=\sum_{j=1}^m\frac{{\lambda}_j|B^{(j)}|^{\frac1q}}
{\|{\one}_{B^{(j)}}\|_{X}}
\int_{B^{(j)}}\lf[h_j(x)-P^d_{B^{(j)}}h_j(x)\r]g(x){\one}_{B^{(j)}}(x)\,dx\noz\\
&\quad\lesssim
\sum_{j=1}^m{\lambda}_j
\int_{B^{(j)}}a_j(x)g(x)\,dx
= \sum_{j=1}^m{\lambda}_j L(a_j)
= L\lf(\sum_{j=1}^m{\lambda}_j a_j\r)\noz\\
&\quad\ls\lf\|\sum_{j=1}^m{\lambda}_j a_j\r\|_{H_X^A(\rn)}
\ls\lf\|\lf\{\sum_{j=1}^m
\lf[\frac{{\lambda}_j}{\|{\one}_{B^{(j)}}\|_X}\r]^{\tz_0}
{\one}_{B^{(j)}}\r\}^{\frac1{\tz_0}}\r\|_{X}.\noz
\end{align*}
Using this and Definition \ref{LAXqds}, we find
$g\in\mathcal{L}_{X,q',d,\tz_0}^{A}({{\rr}^n})$.
Moreover, from $g\in\mathcal{L}_{X,q',d,\tz_0}^{A}({{\rr}^n})$,
Proposition \ref{s2p2}, and \eqref{4.14}, we infer that
$$\|g\|_{\mathcal{L}_{X,q',d,\tz_0}^{A}({{\rr}^n})}\sim
\|g\|_{\mathcal{L}_{X,q',d}^A({{\rr}^n})}\ls
\|L\|_{(H_X^A(\rn))^*}.$$
This finishes the proof of (ii) and hence Theorem \ref{s2t1}.
\end{proof}
As a consequence of Theorem \ref{s2t1}, we have the following
equivalence of the anisotropic ball Campanato-type function space
$\mathcal{L}_{X,q,d,s}^{A}(\rn)$;
we omit the details.

\begin{corollary}\label{s2c1}
Let $A$, $X$, $d$, $\tz_0$, and $p_0$ be
the same as in Theorem \ref{s2t1}
and $q\in[1,\fz)$ when $p_0\in(0,1)$, or $q\in[1,p_0')$ when $p_0\in[1,\fz)$.
Then
$$
\mathcal{L}_{X,1,d_{X,A},\tz_0}^{A}(\rn)=\mathcal{L}_{X,q,d,\tz_0}^{A}(\rn)
$$
with equivalent quasi-norms,
where $d_{X,A}$ is the same as in \eqref{dxa}.
\end{corollary}

\begin{remark}\label{3.24.x1}
\begin{enumerate}
\item [{\rm(i)}] If $A:=2\,I_{n\times n}$, then Theorem \ref{s2t1}
and Corollary \ref{s2c1} were obtained in
\cite[Theorem 3.14 and Corollary 3.15]{zhyy21},
respectively.
\item [{\rm(ii)}] Recently, Yan et al. \cite[Theorem 6.6]{yhyy}
obtained the dual theorem of the Hardy space $H_Y(\cx)$ associated
with the ball quasi-Banach function space $Y(\cx)$
on a given space $\cx$ of homogeneous type.
We point out that, since there exists no linear structure in
a general space $\cx$ of homogeneous type, one can not introduce
the Schwartz function and the polynomial on $\cx$.
Indeed, any atom in \cite{yhyy} only has zero degree vanishing moment,
while the atom in Theorem \ref{s2t1} has vanishing moments up to order
$d\in[d_{X,\,A},\fz)\cap\nn$ with $d_{X,\,A}$ the same as in \eqref{dxa}.
Thus, although $(\rn,\rho,dx)$ is a space of homogeneous type,
Theorem \ref{s2t1} can not be deduced from by \cite[Theorem 6.6]{yhyy}
and, actually, they can not cover each other.
\end{enumerate}
\end{remark}

\section{Equivalent Characterizations of
$\mathcal{L}_{X,q,d,\tz_0}^A({\rn})$\label{s3}}
In this section, applying the dual theorem obtained in Section \ref{s2},
we establish several equivalent characterizations for the anisotropic
ball Campanato-type function space $\mathcal{L}_{X,q,d,\tz_0}^{A}({\rn})$.
This plays an important role in establishing the Carleson measure
characterization of $\mathcal{L}_{X,1,d,\tz_0}^{A}({\rn})$
in Section \ref{s5} below.
\begin{theorem}\label{s3t1}
Let $A$, $X$, $q$, $d$, and $\tz_0$ be the same
as in Corollary \ref{s2c1}
and
\begin{align}\label{2.19.y2}
\vaz\in\lf(\frac{\ln b}{\ln(\lambda_-)}
\lf[\frac2s+d\frac{\ln(\lambda_+)}{\ln b}\r],\fz\r)
\end{align}
for some $s\in(0,\tz_0)$. Then the following statements are mutually equivalent:
\begin{enumerate}
\item[{\rm(i)}]
$f\in\mathcal{L}_{X,q,d,\tz_0}^{A}(\rn)$;

\item[{\rm(ii)}]
for any $f\in L^q_{\rm loc}(\rn)$,
\begin{align}\label{2.19.y1}
\|f\|_{\mathcal{L}_{X,1,d,\tz_0}^{A,\vaz}(\rn)}
:=&\,\sup
\lf\|\lf\{\sum_{i=1}^m
\lf(\frac{{\lambda}_i}{\|{\one}_{x_i+B_{l_i}}\|_X}\r)^{\tz_0}
{\one}_{x_i+B_{l_i}}\r\}^{\frac1{\tz_0}}\r\|_{X}^{-1}\\
&\quad\times\sum_{j=1}^m
\frac{{\lambda}_j|x_j+B_{l_j}|}{\|{\one}_{x_j+B_{l_j}}\|_{X}}\noz\\
&\quad\times
\int_{\rn}\frac{b^{\vaz l_j\frac{\ln(\lambda_-)}{\ln b}}
|f(x)-P^d_{x_j+B_{l_j}}f(x)|}
{b^{l_j[1+\vaz\frac{\ln(\lambda_-)}
{\ln b}]}+[\rho(x-x_j)]^{1+\vaz\frac{\ln(\lambda_-)}{\ln b}}}\,dx\noz\\
<&\,\infty,\noz
\end{align}
where the supremum is taken over all $m\in\nn$,
$\{x_j+B_{l_j}\}_{j=1}^m\subset \CB$,
with both
$\{x_j\}_{j=1}^{m}\subset\rn$ and $\{l_j\}_{j=1}^m\subset\zz$,
and $\{\lambda_j\}_{j=1}^m\subset[0,\infty)$ with $\sum_{j=1}^m\lambda_j\neq0$.
\end{enumerate}

Moreover, for any $f\in L^q_{\rm loc}(\rn)$,
$$\|f\|_{\mathcal{L}_{X,q,d,\tz_0}^{A}(\rn)}\sim
\|f\|_{\mathcal{L}_{X,1,d,\tz_0}^{A,\vaz}(\rn)}$$
with the positive equivalence constants
independent of $f$.
\end{theorem}

To show Theorem \ref{s3t1}, we need the following technical lemma
whose proof is a slight modification of \cite[Lemma 2.13]{wyy22};
we omit the details.

\begin{lemma}\label{s3l2}
Let $X$ be a ball quasi-Banach function space satisfying
Assumption \ref{Assum-1} with $p_-\in(0,\infty)$,
$\ell\in\zz_+$, and $s\in(0,\min\{p_-,1\})$.
Then there exists a positive constant $C$, independent of both
$\ell$ and $s$,
such that, for any sequence $\{x_j\}_{j\in\nn}\subset\rn$
and any sequence $\{k_j\}_{j\in\nn}\subset\zz$,
\begin{align*}
\left\|\sum_{j\in\nn}\one_{x_j+B_{k_j+\ell}}\r\|_X
&\le Cb^{\frac\ell s}\left\|\sum_{j\in\nn}\one_{x_j+B_{k_j}}\r\|_{X},
\end{align*}
where, for any $j\in\nn$, $B_{k_j}$ is the same as in \eqref{B_k}.
\end{lemma}

Now, we show Theorem \ref{s3t1}.
\begin{proof}[Proof of Theorem \ref{s3t1}]
According to Corollary \ref{s2c1},
to prove the present theorem, we only need to show that,
for any $f\in L^q_{\rm loc}(\rn)$,
\begin{align}\label{2.18.x1}
\|f\|_{\mathcal{L}_{X,1,d,\tz_0}^{A}(\rn)}\sim
\|f\|_{\mathcal{L}_{X,1,d,\tz_0}^{A,\vaz}(\rn)}.
\end{align}

We first prove
\begin{align}\label{2.17.x1}
\|f\|_{\mathcal{L}_{X,1,d,\tz_0}^{A}(\rn)}\ls
\|f\|_{\mathcal{L}_{X,1,d,\tz_0}^{A,\vaz}(\rn)}.
\end{align}
Indeed, by Definition \ref{def-shqn}, we find that,
for any $m\in\nn$, $\{x_j+B_{l_j}\}_{j=1}^m\subset \CB$
with both
$\{x_j\}_{j=1}^{m}\subset\rn$ and $\{l_j\}_{j=1}^m\subset\zz$,
$\{\lambda_j\}_{j=1}^m\subset[0,\infty)$ with $\sum_{j=1}^m\lambda_j\neq0$,
$\vaz\in(0,\fz)$, and $j\in\{1,2,\ldots,m\}$,
\begin{align*}
&\int_{\rn}\frac{b^{\vaz l_j\frac{\ln(\lambda_-)}{\ln b}}
|f(x)-P^d_{x_j+B_{l_j}}f(x)|}{b^{l_j[1+\vaz\frac{\ln(\lambda_-)}{\ln b}]}
+[\rho(x-x_j)]^{1+\vaz\frac{\ln(\lambda_-)}{\ln b}}}\,dx\\
&\quad\ge\int_{x_j+B_{l_j}}\frac{b^{\vaz l_j\frac{\ln(\lambda_-)}{\ln b}}
|f(x)-P^d_{x_j+B_{l_j}}f(x)|}{b^{l_j[1+\vaz\frac{\ln(\lambda_-)}{\ln b}]}
+[\rho(x-x_j)]^{1+\vaz\frac{\ln(\lambda_-)}{\ln b}}}\,dx\\
&\quad\sim\int_{x_j+B_{l_j}}\frac{b^{\vaz l_j
\frac{\ln(\lambda_-)}{\ln b}}|f(x)-P^d_{x_j+B_{l_j}}f(x)|}
{b^{l_j[1+\vaz\frac{\ln(\lambda_-)}{\ln b}]}}\,dx\\
&\quad=\fint_{x_j+B_{l_j}}\lf|f(x)-P^d_{x_j+B_{l_j}}f(x)\r|\,dx,
\end{align*}
which, together with Definition \ref{LAXqds}
and \eqref{2.19.y1}, further implies \eqref{2.17.x1}.

Conversely, from Definitions \ref{def-shqn} and \ref{LAXqds},
we deduce that, for any $m\in\nn$, $\{x_j+B_{l_j}\}_{j=1}^m\subset \CB$
with both $\{x_j\}_{j=1}^{m}\subset\rn$ and $\{l_j\}_{j=1}^m\subset\zz$,
$\{\lambda_j\}_{j=1}^m\subset[0,\infty)$ with $\sum_{j=1}^m\lambda_j\neq0$,
\begin{align}\label{s3e1}
&\sum_{j=1}^m\frac{{\lambda}_j|x_j+B_{l_j}|}{\|{\one}_{x_j+B_{l_j}}\|_{X}}
\int_{\rn}\frac{b^{\vaz l_j\frac{\ln(\lambda_-)}{\ln b}}|f(x)-P^d_{x_j+B_{l_j}}f(x)|}
{b^{l_j[1+\vaz\frac{\ln(\lambda_-)}{\ln b}]}
+[\rho(x-x_j)]^{1+\vaz\frac{\ln(\lambda_-)}{\ln b}}}\,dx\\
&\quad=\sum_{j=1}^m\frac{{\lambda}_j|x_j+B_{l_j}|}{\|{\one}_{x_j+B_{l_j}}\|_{X}}
\lf[\int_{x_j+B_{l_j}}
+\sum_{k=0}^\fz\int_{x_j+B_{l_j+k+1}\setminus x_j+B_{l_j+k}}\r]\noz\\
&\qquad\times
\frac{b^{\vaz l_j\frac{\ln(\lambda_-)}{\ln b}}|f(x)-P^d_{x_j+B_{l_j}}f(x)|}
{b^{l_j[1+\vaz\frac{\ln(\lambda_-)}{\ln b}]}
+[\rho(x-x_j)]^{1+\vaz\frac{\ln(\lambda_-)}{\ln b}}}\,dx\noz\\
&\quad\le\sum_{j=1}^m
\frac{{\lambda}_j}{\|{\one}_{x_j+B_{l_j}}
\|_{X}}\int_{x_j+B_{l_j}}\lf|f(x)-P^d_{x_j+B_{l_j}}f(x)\r|\,dx\noz\\
&\qquad+\sum_{j=1}^m\frac{{\lambda}_j}{\|{\one}_{x_j+B_{l_j}}\|_{X}}
\sum_{k=0}^\fz b^{-k[1+\vaz\frac{\ln(\lambda_-)}{\ln b}]}\noz\\
&\qquad\times\int_{x_j+B_{l_j+k+1}}\lf|f(x)-P^d_{x_j+B_{l_j}}f(x)\r|\,dx\noz\\
&\quad\leq\lf\|\lf\{\sum_{i=1}^m
\lf(\frac{{\lambda}_i}{\|{\one}_{x_i+B_{l_i}}\|_X}\r)^{\tz_0}
{\one}_{x_i+B_{l_i}}\r\}^{\frac1{\tz_0}}\r\|_{X}
\|f\|_{\mathcal{L}_{X,1,d,\tz_0}^{A}(\rn)}+{\rm I},\noz
\end{align}
where
\begin{align}\label{2.18.x2}
{\rm I}:&=\sum_{j=1}^m\frac{{\lambda}_j}{\|{\one}_{x_j+B_{l_j}}\|_{X}}
\sum_{k=0}^\fz b^{-k[1+\vaz\frac{\ln(\lambda_-)}{\ln b}]}
\int_{x_j+B_{l_j+k+1}}\lf|f(x)-P^d_{x_j+B_{l_j}}f(x)\r|\,dx\\
&\ls\sum_{j=1}^m\frac{{\lambda}_j}{\|{\one}_{x_j+B_{l_j}}\|_{X}}
\sum_{k\in\nn} b^{-k[1+\vaz\frac{\ln(\lambda_-)}{\ln b}]}
\int_{x_j+B_{l_j+k}}\lf|f(x)-P^d_{x_j+B_{l_j+k}}f(x)\r|\,dx\noz\\
&\quad+\sum_{j=1}^m\frac{{\lambda}_j}{\|{\one}_{x_j+B_{l_j}}\|_{X}}
\sum_{k\in\nn}b^{-k[1+\vaz\frac{\ln(\lambda_-)}{\ln b}]}\noz\\
&\quad\times\int_{x_j+B_{l_j+k}}
\lf|P^d_{x_j+B_{l_j+k}}f(x)-P^d_{x_j+B_{l_j}}f(x)\r|\,dx.\noz
\end{align}
Note that, on the one hand, by the definition of minimizing polynomials,
\eqref{2.21.x1}, \cite[Lemma 2.19]{jtyyz2022}, and Lemma \ref{s3l1},
we find that, for any $k\in\nn$ and $x\in x_j+B_{l_j+k+1}$,
\begin{align}\label{2.19.x1}
&\lf|P^d_{x_j+B_{l_j+k}}f(x)-P^d_{x_j+B_{l_j}}f(x)\r|\\
&\quad\le\sum_{\nu=1}^{k}
\lf|P^d_{x_j+B_{l_j+\nu}}f(x)-P^d_{x_j+B_{l_j+\nu-1}}f(x)\r|\noz\\
&\quad=\sum_{\nu=1}^{k}
\lf|P^d_{x_j+B_{l_j+\nu-1}}\lf(f-P^d_{x_j+B_{l_j+\nu}}f\r)(x)\r|\noz\\
&\quad\le\sum_{\nu=1}^{k}\lf\|P^d_{x_j+B_{l_j+\nu-1}}
\lf(f-P^d_{x_j+B_{l_j+\nu}}f\r)\r\|_{L^{\fz}(B(x_j,\lz_+^{l_j+k}))}\noz\\
&\quad\ls\sum_{\nu=1}^{k}\lf(\frac{\lz_+^{l_j+k}}{\lz_-^{l_j+v-1}}\r)^{d}
\lf\|P^d_{x_j+B_{l_j+\nu-1}}
\lf(f-P^d_{x_j+B_{l_j+\nu}}f\r)\r\|_{L^{\fz}(B(x_j,\lz_-^{l_j+\nu-1}))}\noz\\
&\quad\ls\lambda_+^{kd}\sum_{\nu=1}^{k}\frac{1}{|x_j+B_{l_j+\nu-1}|}
\int_{x_j+B_{l_j+\nu}}\lf|f(y)-P^d_{x_j+B_{l_j+\nu}}f(y)\r|\,dy;\noz
\end{align}
on the other hand, from Definition \ref{BQBFS}(ii),
the fact that $s\in(0,\tz_0)$, and Lemma \ref{s3l2},
we infer that, for any $j\in\{1,2,\ldots,m\}$,
\begin{align}\label{2.19.x3}
\frac{1}{\|{\one}_{x_j+B_{l_j}}\|_{X}}
&\ls b^{\frac{k}{s}}\frac{1}{\|{\one}_{x_j+B_{l_j+k}}\|_{X}}
\end{align}
and, for any $k\in\nn$,
\begin{align}\label{2.19.x2}
&\lf\|\lf\{\sum_{i=1}^m
\lf(\frac{{\lambda}_i}{\|{\one}_{x_i+B_{l_i+k}}\|_X}\r)^{\tz_0}
{\one}_{x_i+B_{l_i+k}}\r\}^{\frac1{\tz_0}}\r\|_{X}\\
&\quad\le \lf\|\lf\{\sum_{i=1}^m
\lf(\frac{{\lambda}_i}{\|{\one}_{x_i+B_{l_i}}\|_X}\r)^{\tz_0}
{\one}_{x_i+B_{l_i+k}}\r\}^{\frac1{\tz_0}}\r\|_{X}\noz\\
&\quad\ls b^{\frac{k}{s}}\lf\|\lf\{\sum_{i=1}^m
\lf(\frac{{\lambda}_i}{\|{\one}_{x_i+B_{l_i}}\|_X}\r)^{\tz_0}
{\one}_{x_i+B_{l_i}}\r\}^{\frac 1{\tz_0}}\r\|_{X}.\noz
\end{align}
Combining \eqref{2.18.x2}, \eqref{2.19.x1}, \eqref{2.19.x3},
\eqref{2.19.x2}, and Lemma \ref{s3l2}, we conclude that
\begin{align*}
&\lf\|\lf\{\sum_{i=1}^m
\lf(\frac{{\lambda}_i}{\|{\one}_{x_i+B_{l_i}}\|_X}\r)^{\tz_0}
{\one}_{x_i+B_{l_i}}\r\}^{\frac 1{\tz_0}}\r\|_{X}^{-1}\times{\rm I}\\
&\quad\ls\sum_{k\in\nn} b^{-k[1-\frac{2}{s}+\vaz\frac{\ln(\lambda_-)}{\ln b}]}
\lf\|\lf\{\sum_{i=1}^m\lf(\frac{{\lambda}_i}{\|{\one}_{x_i+B_{l_i+k}}\|_X}\r)
^{\tz_0}
{\one}_{x_i+B_{l_i+k}}\r\}^{\frac1{\tz_0}}\r\|_{X}^{-1}\noz\\
&\qquad\times\sum_{j=1}^m\frac{{\lambda}_j}{\|{\one}_{x_j+B_{l_j+k}}\|_{X}}
\int_{x_j+B_{l_j+k}}\lf|f(x)-P^d_{x_j+B_{l_j+k}}f(x)\r|\,dx\noz\\
&\qquad+\sum_{k\in\nn}\lf(\frac{\lambda_+^{d}}{\lambda_-^{\vaz }}b\r)^k
\sum_{\nu=1}^{k}b^{\nu(\frac{2}{s}-1)}
\lf\|\lf\{\sum_{i=1}^m
\lf(\frac{{\lambda}_i}{\|{\one}_{x_i+B_{l_i+\nu}}\|_X}\r)^{\tz_0}
{\one}_{x_i+B_{l_i+\nu}}\r\}^{\frac1{\tz_0}}\r\|_{X}^{-1}\noz\\
&\qquad\times\sum_{j=1}^m
\frac{{\lambda}_j}{\|{\one}_{x_j+B_{l_j+\nu}}\|_{X}}
\int_{x_j+B_{l_j+\nu}}\lf|f(y)-P^d_{x_j+B_{l_j+\nu}}f(y)\r|\,dy\noz\\
&\quad\ls\|f\|_{\mathcal{L}_{X,1,d,\tz_0}^{A}(\rn)}\lf\{\sum_{k\in\nn}
b^{-k[1-\frac{2}{s}+\vaz\frac{\ln(\lambda_-)}{\ln b}]}
+\sum_{k\in\nn}\lf(\frac{\lambda_+^{d}}{\lambda_-^{\vaz }}b\r)^k
\sum_{\nu=1}^{k}b^{\nu(\frac{2}{s}-1)}\r\}\noz\\
&\quad\ls\|f\|_{\mathcal{L}_{X,1,d,\tz_0}^{A}(\rn)}\lf\{
\sum_{k\in\nn}b^{-k[1-\frac{2}{s}+\vaz\frac{\ln(\lambda_-)}{\ln b}]}
+\sum_{k\in\nn}b^{-k[-\frac{2}{s}+\vaz\frac{\ln(\lambda_-)}{\ln b}
-d\frac{\ln(\lambda_+)}{\ln b}]}\r\}\noz\\
&\quad\sim\|f\|_{\mathcal{L}_{X,1,d,\tz_0}^{A}(\rn)}
\sum_{k\in\nn}b^{-k[-\frac{2}{s}+\vaz\frac{\ln(\lambda_-)}{\ln b}
-d\frac{\ln(\lambda_+)}{\ln b}]},\noz
\end{align*}
which, together with \eqref{s3e1}, \eqref{2.19.y1},
\eqref{2.19.y2}, and the arbitrariness of $m\in\nn$,
$\{x_j+B_{l_j}\}_{j=1}^m\subset \CB$
with both
$\{x_j\}_{j=1}^{m}\subset\rn$ and $\{l_j\}_{j=1}^m\subset\zz$,
and $\{\lambda_j\}_{j=1}^m\subset[0,\infty)$ with $\sum_{j=1}^m\lambda_j\neq0$,
further implies that
\begin{align*}
\|f\|_{\mathcal{L}_{X,1,d,\tz_0}^{A,\vaz}(\rn)}
&\ls\|f\|_{\mathcal{L}_{X,1,d,\tz_0}^{A}(\rn)}
\sum_{k\in\nn}b^{-k[-\frac{2}{s}+\vaz\frac{\ln(\lambda_-)}{\ln b}
-d\frac{\ln(\lambda_+)}{\ln b}]}  \\
&\sim\|f\|_{\mathcal{L}_{X,1,d,\tz_0}^{A}(\rn)}.\noz
\end{align*}
This, combined with \eqref{2.17.x1},
proves \eqref{2.18.x1} and hence finishes the proof of Theorem \ref{s3t1}.
\end{proof}

We can obtain one more equivalent characterization of
$\mathcal{L}_{X,q,d,\tz_0}^{A}(\rn)$ as follows,
whose proof is a slight modification of Theorem \ref{s3t1};
we omit the details.

\begin{theorem}\label{s3t2}
If $A$, $X$, $q$, $d$, $\tz_0$, and $\vaz$ are
the same as in Theorem \ref{s3t1},
then the conclusion of Theorem \ref{s3t1} with $m$
replaced by $\fz$ still holds true,
where the supremum therein is taken over all
$\{x_j+B_{l_j}\}_{j\in\nn}\subset \CB$
with both $\{x_j\}_{j\in\nn}\subset\rn$
and $\{l_j\}_{j\in\nn}\subset\zz$ and
$\{\lambda_j\}_{j\in\nn}\subset[0,\infty)$ with
$\sum_{j\in\nn}\lambda_j\neq0$
satisfying
\begin{align*}
\lf\|\lf\{\sum_{j\in\nn}
\lf(\frac{{\lambda}_j}{\|{\one}_{x_j+B_{l_j}}\|_X}\r)^{\tz_0}
{\one}_{x_j+B_{l_j}}\r\}^{\frac1{\tz_0}}\r\|_{X}\in(0,\infty).
\end{align*}
\end{theorem}

\begin{remark}\label{3.24.x2}
If $A:=2\,I_{n\times n}$, then
Theorems \ref{s3t1} and \ref{s3t2} were obtained in
\cite[Theorems 4.1 and 4.4]{zhyy21}, respectively.
\end{remark}

\section{Littlewood--Paley Function Characterizations of $H^A_X(\rn)$\label{s4}}

In this section, we establish the characterizations of $H_X^A(\rn)$
in terms of the anisotropic Lusin area function,
the anisotropic Littlewood--Paley $g$-function, or
the anisotropic Littlewood--Paley $g_\lambda^\ast$-function.
These are the consequence of the atomic and the
finite atomic characterizations of $H_{X}^{A}(\rn)$ obtained
in \cite{wyy22}
and play important roles in establishing
the Carleson measure characterization
of $\mathcal{L}_{X,1,d,\tz_0}^{A}({\rn})$ in Section \ref{s5}.
First, we recall the concepts of both the radial maximal
function and the radial grand maximal function, which
were introduced in \cite{Bownik}.

\begin{definition}\label{radialM}
Let $\varphi\in\cs(\rn)$ and $f\in\cs'(\rn)$.
The {\it radial maximal function} $M_\varphi^0(f)$
of $f$ with respect to $\varphi$ is defined by setting,
for any $x\in\rn$,
\begin{equation*}
M_\varphi^0 (f)(x):=\sup_{k \in \zz}|f\ast\varphi_k(x)|.
\end{equation*}
Moreover, for any given $N\in\nn$,
the {\it radial grand maximal function} $M_N^0(f)$
of $f\in\cs'(\rn)$ is defined by setting, for any $x\in\rn$,
\begin{equation*}
M_N^0(f)(x):=\sup_{\varphi\in\cs_N(\rn)} M_\varphi^0(f)(x).
\end{equation*}
\end{definition}
In what follows, for any $\varphi\in\cs(\rn)$, $\widehat{\varphi}$
is defined by setting, for any $\xi\in\rn$,
\begin{align*}
\widehat \varphi(\xi) := \int_{\rn} \varphi(x) e^{-2\pi\imath x \cdot \xi} \,dx,
\end{align*}
where $\imath:=\sqrt{-1}$ and $x\cdot\xi :=\sum_{i=1}^{n}x_i \xi_i$
for any $x:=(x_1,\ldots,x_n),\xi:=(\xi_1,\ldots,\xi_n) \in \rn$.
For any $f\in\cs'(\rn)$, $\widehat f$ is defined by setting,
for any $\varphi\in\mathcal{S}(\rn)$,
$\la\widehat f,\varphi\ra:=\la f,\widehat\varphi\ra$.

Recall that $f\in\cs'(\rn)$ is said to \emph{vanish weakly at infinity} if,
for any $\phi\in\cs(\rn)$,
$f\ast\phi_k\to0$ in $\cs'(\rn)$ as $k\to\infty$
(see, for instance, \cite[p.\,50]{fs82}).
Let $\mathcal{C}_{\rm c}^{\infty}(\rn)$ denote the collection of
all the infinitely differentiable functions with compact support
on $\rn$.
The following Calder\'on reproducing formula is just
\cite[Proposition 2.14]{blyz08weight}.
\begin{lemma}\label{s4l1}
Let $d \in \mathbb{Z}_{+}$ and $A$ be a dilation. Assume that
$\phi\in\mathcal{C}_{\rm c}^{\infty}(\rn)$ satisfies
\begin{align}\label{s4l1e1}
\supp \phi \subset B_0,
\ \int_{{{\rr}^n}}x^\gamma\phi(x)\,dx=0
\ \mbox{for any}\ \gamma\in\zz_+^n\
\mbox{with}\ |\gamma|\leq d,
\end{align}
and there exists a positive constant $C$ such that
\begin{align}\label{s4l1e2}
\lf|\widehat{\phi}(\xi)\r|\ge C \
when \ \xi\in\left\{x\in\rn :\ (2||A||)^{-1}\leq\rho(x)\leq1 \right\},
\end{align}
where the dilation $A:=(a_{i, j})_{1\leq i, j\leq n}$
and $||A||:=(\sum_{i, j=1}^n|a_{i, j}|^2)^{1/2}$.
Then there exists a $\psi \in \mathcal{S}\left(\mathbb{R}^{n}\right)$ such that
\begin{enumerate}
\item[\rm(i)] $\supp\widehat{\psi}$ is compact and away from the origin;
\item[\rm(ii)] for any $\xi \in \mathbb{R}^{n} \backslash\{\bf{0}\}$,
$$\sum_{j \in \mathbf{Z}} \widehat{\psi}\left(\left(A^{*}\right)^{j} \xi\right)
\widehat{\phi}\left(\left(A^{*}\right)^{j} \xi\right)=1,$$
where $A^{*}$ denotes the adjoint matrix of $A$.
\end{enumerate}
Moreover, for any $f \in \mathcal{S}^{\prime}\left(\mathbb{R}^{n}\right)$,
if $f$ vanishes weakly at infinity, then
$$f=\sum_{j \in \mathbb{Z}} f \ast \psi_{j} * \phi_{j} \  in
\  \mathcal{S}^{\prime}\left(\mathbb{R}^{n}\right).$$
\end{lemma}

The following definitions of the anisotropic Lusin area function,
the anisotropic Littlewood--Paley $g$-function,
and the anisotropic Littlewood--Paley $g_\lambda^\ast$-function
were introduced in \cite[Definition 2.6]{lyy2018}.
\begin{definition}\label{de4.1}
Let $\phi\in\cs(\rn)$ be the same as in Lemma \ref{s4l1}.
For any $f\in\cs'(\rn)$, the \emph{anisotropic Lusin area function} $S(f)$,
the \emph{anisotropic Littlewood--Paley $g$-function $g(f)$},
and the \emph{anisotropic Littlewood--Paley} $g_\lambda^\ast$-\emph{function}
$g_\lambda^\ast(f)$ with
any given $\lambda\in(0,\infty)$ are defined, respectively, by setting,
for any $x\in\rn$,
\begin{equation}\label{deaf}
S(f)(x):=\left[\sum_{k\in\zz}b^{-k}\int_{x+B_k}\lf|f\ast \phi_k(y)\r|^2\,dy\right]
^\frac12,
\end{equation}
\begin{equation*}
g(f)(x):=\left[\sum_{k\in\zz}\lf|f\ast\phi_k(x)\r|^2\right]^\frac12,
\end{equation*}
and
\begin{equation*}
g_\lambda^\ast(f)(x):=\left\{\sum_{k\in\zz}b^{-k}\int_{\rn}
\left[\frac{b^k}{b^k+\rho(x-y)}\right]^\lambda\lf|f\ast\phi_k(y)\r|^2\,dy\right\}
^\frac12.
\end{equation*}
\end{definition}

We characterize the space $H_X^{A}(\rn)$, respectively, in terms of the
anisotropic Lusin area function, the anisotropic Littlewood--Paley $g$-function,
and the anisotropic Littlewood--Paley $g_\lambda^\ast$-function as follows.
\begin{theorem}\label{s4t1}
Let $A$ be a dilation and $X$ a ball quasi-Banach function space satisfying
both Assumption \ref{Assum-1} with $p_-\in(0,\fz)$ and Assumption \ref{Assum-2}
with the same $p_-$, $\tz_0\in(0,\underline{p})$, and $p_0\in(\tz_0,\fz)$,
where $\underline{p}$ is the same as in \eqref{underp}.
Then $f\in H_X^A(\rn)$ if and only if $f\in\cs'(\rn)$,
$f$ vanishes weakly at infinity, and $\|S(f)\|_X<\infty$.
Moreover, for any $f \in H_{X}^{A}\left(\mathbb{R}^{n}\right)$,
$$
\|S(f)\|_{X}  \sim \|f\|_{H_{X}^{A}\left(\rn\right)},
$$
where the positive equivalence constants are independent of $f$.
\end{theorem}
\begin{theorem}\label{s4t1'}
Let $A$ and $X$ be the same as in Theorem \ref{s4t1}.
Then $f\in H_X^A(\rn)$ if and only if $f\in\cs'(\rn)$,
$f$ vanishes weakly at infinity, and $\|g(f)\|_X<\infty$.
Moreover, for any $f \in H_{X}^{A}\left(\mathbb{R}^{n}\right)$,
$$
\|g(f)\|_{X}  \sim \|f\|_{H_{X}^{A}\left(\rn\right)},
$$
where the positive equivalence constants are independent of $f$.
\end{theorem}
Moreover, by Theorems \ref{s4t1} and \ref{s4t1'}
and an argument similar to that used in the proof of
\cite[Theorem 5.3]{yhyy2}, we easily obtain the following result;
we omit the details.
\begin{theorem}\label{s4t1''}
Let $A$, $X$, and $\tz_0$ be the same as in Theorem \ref{s4t1},
Then $f\in H_X^A(\rn)$ if and only if $f\in\cs'(\rn)$,
$f$ vanishes weakly at infinity, and $\|g_{\lambda}^{*}(f)\|_X<\infty$.
Moreover, for any $f \in H_{X}^{A}\left(\mathbb{R}^{n}\right)$,
$$
\|g_{\lambda}^{*}(f)\|_{X}  \sim \|f\|_{H_{X}^{A}\left(\rn\right)},
$$
where the positive equivalence constants are independent of $f$,
$\lambda \in(\max\{1,2/r_+\}, \infty)$, and
\begin{align}\label{3.25.x1}
r_+&:=\sup\lf\{\tz_0\in(0,\infty):\r.X\,\text{satisfies Assumption \ref{Assum-2}
for this}\ \tz_0\ \\
&\qquad\qquad\quad\lf.\text{and some}\ p_0\in(\tz_0,\infty)\r\}.\noz
\end{align}
\end{theorem}
To prove Theorem \ref{s4t1}, we first present the following conclusion
which shows
that the quasi-norm $\|\cdot\|_X$ of the anisotropic Lusin area functions
defined by different $\phi$ as in Lemma \ref{s4l1} are equivalent.
\begin{theorem} \label{s4t2}
Let $A$ and $X$ be the same as in Theorem \ref{s4t1} and $\phi,
\psi\in\mathcal{C}_{\rm c}^{\infty}(\rn)$ satisfy
both \eqref{s4l1e1}
and \eqref{s4l1e2}.
Then, for any $f\in\cs'(\rn)$ vanishing weakly at infinity,
$$\|S_\phi(f)\|_{X}\sim\|S_\psi(f)\|_{X},$$
where $S_\phi(f)$ is the same as in \eqref{deaf},
$S_\psi(f)$ is the same as in \eqref{deaf} with $\phi$
replaced by $\psi$,
and the positive equivalence
constants are independent of $f$.
\end{theorem}
To prove Theorem \ref{s4t2}, we need the following lemma
which is just \cite[Lemma2.3]{blyz08weight} and originates from
\cite[Theorem 11]{christ90}.
\begin{lemma}\label{defdya}
Let $A$ be a dilation. Then there exists a collection
$$\mathcal{Q} :=\left\{Q_{\alpha}^{k}
\subset \mathbb{R}^{n}:\ k \in \mathbb{Z}, \alpha \in I_{k}\right\}$$
of open subsets, where $I_{k}$ is certain index set, such that

\begin{enumerate}
\item[{\rm (i)}] $|\mathbb{R}^{n}
\setminus\bigcup_{\alpha}Q_{\alpha}^{k}|=0$
for each fixed $k$ and $Q_{\alpha}^{k}\cap Q_{\beta}^{k}=\emptyset$
for any $\alpha \neq \beta$;

\item[{\rm (ii)}] for any $\alpha, \beta, k$, $\ell$ with $\ell \geq k$,
either $Q_{\alpha}^{k} \cap Q_{\beta}^{\ell}=\emptyset$
or $Q_{\alpha}^{\ell} \subset Q_{\beta}^{k}$;

\item[{\rm (iii)}] for each $(\ell, \beta)$ and each $k<\ell$,
there exists a unique $\alpha$ such that $Q_{\beta}^{\ell} \subset Q_{\alpha}^{k}$;

\item[{\rm (iv)}] there exist certain negative integer $v$
and positive integer $u$ such that, for any $Q_{\alpha}^{k}$ with
both $k \in \mathbb{Z}$
and $\alpha \in I_{k}$, there exists an
$x_{Q_{\alpha}^{k}} \in Q_{\alpha}^{k}$
satisfying that, for any $x \in Q_{\alpha}^{k}$,
$$x_{Q_{\alpha}^{k}}+B_{v k-u} \subset Q_{\alpha}^{k} \subset x+B_{v k+u}.$$
\end{enumerate}
\end{lemma}
In what follows, for convenience, we call
$\mathcal{Q}:=\{Q_\alpha^k\}_{k\in\zz, \alpha\in I_k}$
in Lemma \ref{defdya} \emph{dyadic cubes} and $k$ the \emph{level},
denoted by $\ell(Q_{\alpha}^{k})$,
of the dyadic cube $Q_{\alpha}^{k}$ with both
$k\in \zz$ and $\alpha\in I_k$.

The following technical lemma is also necessary,
which is just \cite[Lemma 6.9]{hlyy20}.

\begin{lemma}\label{s4l2}
Let $d$ be the same
as in \eqref{def-d}, $v$ and $u$ the same
as in Lemma \ref{defdya}(iv), and
$$\eta \in\left(\frac{\ln b}{\ln b+(d+1) \ln \lambda_{-}}, 1\right].$$
Then there exists a positive constant $C$ such that,
for any $k,i\in\mathbb{Z}$,
$\left\{c_{Q}\right\}_{Q\in\mathcal{Q}}\subset[0, \infty)$ with $\mathcal{Q}$
in Lemma \ref{defdya}, and $x\in\mathbb{R}^{n}$,
\begin{align*}
&\sum_{\ell(Q)=\left\lceil\frac{k-u}{v}\right\rceil}|Q|
\frac{b^{(k\vee i)(d+1)\frac{\ln\lambda_{-}}{\ln b}}}{[b^{(k \vee i)}+\rho(x-z_Q)]
^{(d+1)\frac{\ln\lambda_{-}}{\ln b}+1}}\\
&\qquad\leq C b^{-[k-(k \vee i)](\frac{1}{\eta}-1)}
\left\{\cm\left[\sum_{\ell(Q)=\left\lceil\frac{k-u}{v}\right\rceil}
\left(c_{Q}\right)^{\eta} \one_{Q}\right](x)\right\}^{\frac{1}{\eta}},
\end{align*}
where $\ell(Q)$ denotes the level of $Q$, $z_Q\in Q$, and,
for any $k,i \in \mathbb{Z}$, $k \vee i:=$ $\max \{k, i\}$.
\end{lemma}

We now prove Theorem \ref{s4t2}.
\begin{proof}[Proof of Theorem \ref{s4t2}]
By symmetry, to show the present theorem, we only need to prove that,
for any $f\in\cs'(\rn)$ which vanishes weakly at infinity,
\begin{equation}\label{s4e1}
\left\|S_{\phi}(f)\right\|_{X} \lesssim\left\|S_{\psi}(f)\right\|_{X} \text {. }
\end{equation}
To this end, for any $i \in \mathbb{Z}, x \in \mathbb{R}^{n}$,
and $y \in x+B_{i}$, let
$$
J_{\phi}^{(i)}(f)(y):=f \ast \phi_{i}(y).
$$
Then, by Lemma \ref{s4l1} and the Lebesgue dominated convergence theorem,
we find that, for any $i \in \mathbb{Z}, x \in \mathbb{R}^{n}$, and $y \in x+B_{i}$,
\begin{align}\label{s4e2}
J_{\phi}^{(i)}(f)(y)
&=\sum_{k \in \mathbb{Z}} f \ast \psi_{k} \ast \phi_{k} \ast \phi_{i}(y)\\
&=\sum_{k \in \mathbb{Z}}
\int_{\rn} f \ast \psi_{k}(z) \phi_{k} \ast \phi_{i}(y-z)\,dz\noz\\
&=\sum_{k \in \mathbb{Z}}
\sum_{\ell(Q)=\left\lceil\frac{k-u}{v}\right\rceil}
\int_{Q} f \ast \psi_{k}(z) \phi_{k} \ast \phi_{i}(y-z) \,dz\noz
\end{align}
in $\mathcal{S}^{\prime}\left(\mathbb{R}^{n}\right)$,
where all the symbols are the same as in Lemma \ref{s4l2}.
On the other hand, by \cite[Lemma 5.4]{blyz10weight},
we conclude that, for any $k,i \in \mathbb{Z}$ and $x \in \mathbb{R}^{n}$,
$$
\left|\phi_{k} \ast \phi_{i}(x)\right|
\lesssim b^{-(d+1)|k-i| \frac{\ln \lambda_{-}}{\ln b}}
\frac{b^{(k \vee i)(d+1) \frac{\ln \lambda_{-}}{\ln b}}}{\left[b^{(k \vee i)}
+\rho(x)\right]^{(d+1) \frac{\ln \lambda_{-}}{\ln b}+1}}.
$$
This further implies that, for any $Q \in \mathcal{Q}$ with
\begin{equation}\label{s4e3}
\ell(Q)= \left\lceil\frac{k-u}{v}\right\rceil,
\end{equation}
there exists some $z_{Q} \in Q$ such that, for any
$k,i \in \mathbb{Z}, x \in \mathbb{R}^{n}, y \in x+B_{i}$, and $z \in Q$,
\begin{align}\label{s4e4}
&\left|\phi_{k} \ast \phi_{i}(y-z)\right|
\lesssim b^{-(d+1)|k-i| \frac{\ln \lambda_{-}}{\ln b}}
\frac{b^{(k \vee i)(d+1) \frac{\ln \lambda_{-}}{\ln b}}}{\left[b^{(k \vee i)}
+\rho\left(x-z_{Q}\right)\right]^{(d+1) \frac{\ln \lambda_{-}}{\ln b}+1}}.
\end{align}
Moreover, for any $Q \in \mathcal{Q}$ satisfying \eqref{s4e3},
we have $B_{v \ell(Q)+u} \subset B_{k}$.
From this, the H\"older inequality,
and Lemma \ref{defdya}(iv), we deduce that, for any $z \in Q$,

\begin{align*}
\frac{1}{|Q|}\left|\int_{Q} f \ast \psi_{k}(y) \,dy\right|
& \leq\left[\fint_{Q}\left|f \ast \psi_{k}(y)\right|^{2} \,dy\right]
^{\frac{1}{2}} \\
& \leq\left[\frac{1}{|B_{v \ell(Q)-u}|} \int_{z+B_{v \ell(Q)+u}}
\left|f \ast \psi_{k}(y)\right|^{2} \,dy\right]^{\frac{1}{2}} \\
& \lesssim\left[b^{-k} \int_{z+B_{k}}
\left|f \ast \psi_{k}(y)\right|^{2} \,dy\right]^{\frac{1}{2}}
\sim Y_{\psi}^{(k)}(f)(z),
\end{align*}
where, for any $k \in \mathbb{Z}$ and $z \in \mathbb{R}^{n}$,
$$
Y_{\psi}^{(k)}(f)(z):=\left[b^{-k}
\int_{z+B_{k}}\left|f \ast \psi_{k}(y)\right|^{2}  \,dy\right]^{\frac{1}{2}}.
$$
Thus, for any $k \in \mathbb{Z}$ and $Q \in \mathcal{Q}$ satisfying \eqref{s4e3},
$$
\frac{1}{|Q|}\left|\int_{Q} f \ast \psi_{k}(y)  \,dy\right|
\lesssim \inf _{z \in Q} Y_{\psi}^{(k)}(f)(z).
$$
By this, \eqref{s4e2}, \eqref{s4e4}, and Lemma \ref{s4l2},
we conclude that, for any given
$\eta \in\left(\frac{\ln b}{\ln b+(d+1) \ln \lambda_{-}}, 1\right]$
and for any $i \in \mathbb{Z}, x \in \mathbb{R}^{n}$, and $y \in x+B_{i}$,
\begin{align}\label{eqJei}
\left|J_{\phi}^{(i)}(f)(y)\right|
&\lesssim \sum_{k \in \mathbb{Z}}
b^{-(d+1)|k-i| \frac{\ln \lambda_{-}}{\ln b}}
\\
&\quad\times\sum_{\ell(Q)=\left\lceil\frac{k-u}{v}\right\rceil}|Q|
\frac{b^{(k \vee i)(d+1) \frac{\ln \lambda_{-}}{\ln b}}}{[b^{(k \vee i)}
+\rho\left(x-z_{Q}\right)]^{(d+1) \frac{\ln \lambda_{-}}{\ln b}+1}}
\inf _{z \in Q} Y_{\psi}^{(k)}(f)(z)\noz\\
&\lesssim \sum_{k \in \mathbb{Z}} b^{-(d+1)|k-i| \frac{\ln \lambda_{-}}{\ln b}}
b^{-[k-(k \vee i)](\frac{1}{\eta}-1)}\noz\\
&\quad\times\left\{\cm\left(\sum_{\ell(Q)=\left\lceil\frac{k-u}{v}\right\rceil}
\inf _{z \in Q}\left[Y_{\psi}^{(k)}(f)(z)\right]^{\eta} \one_{Q}\right)(x)\right\}
^{\frac{1}{\eta}}\notag\\
&=: J_{(\eta, i)}(x).\notag
\end{align}
Using \eqref{def-d}, we are able to
choose an $\eta \in\left(\frac{\ln b}{\ln b+(d+1)
\ln \lambda_{-}}, \tz_0\right)$. Therefore, from \eqref{eqJei},
it follows that, for such an $\eta$ and any $x \in \mathbb{R}^{n}$,
$$
\left[S_{\phi}(f)(x)\right]^{2}
=\sum_{i \in \mathbb{Z}} b^{-i}
\int_{x+B_{i}}\left|J_{\phi}^{(i)}(f)(y)\right|^{2} \,dy
\lesssim \sum_{i \in \mathbb{Z}}\left[J_{(\eta, i)}(x)\right]^{2}.
$$
This, together with the H\"older inequality and the choice that
$\eta>\frac{\ln b}{\ln b+(d+1) \ln \lambda_{-}}$,
further implies that, for such an $\eta$ and
any $x \in \mathbb{R}^{n}$,
\begin{align*}
\left[S_{\phi}(f)(x)\right]^{2} &\lesssim \sum_{i \in \mathbb{Z}}
\sum_{k \in \mathbb{Z}} \left\{b^{-(d+1)|k-i| \frac{\ln \lambda_{-}}{\ln b}}
b^{-[k-(k \vee i)](\frac{1}{\eta}-1)}\right\}^2\\
&\quad\times\sum_{k \in \mathbb{Z}}
\left\{\cm\left(\sum_{\ell(Q)=\left\lceil\frac{k-u}{v}\right\rceil}\inf_{z \in Q}
\left[Y_{\psi}^{(k)}(f)(z)\right]^{\eta} \one_{Q}\right)(x)\right\}
^{\frac{2}{\eta}}\\
&\lesssim \sum_{k \in \mathbb{Z}}
\left\{\cm\left(\sum_{\ell(Q)=\left\lceil\frac{k-u}{v}\right\rceil}
\inf_{z \in Q}\left[Y_{\psi}^{(k)}(f)(z)\right]^{\eta} \one_{Q}\right)(x)\right\}
^{\frac{2}{\eta}}\\
&\leq \sum_{k \in \mathbb{Z}}\left\{\cm\left(\left[Y_{\psi}^{(k)}(f)\right]
^{\eta}
\right)(x)\right\}^{\frac{2}{\eta}} .
\end{align*}
Thus, by the fact that $\eta<\tz_0$ and Assumption \ref{Assum-1}, we find that
\begin{align*}
\left\|S_{\phi}(f)\right\|_{X}
& \lesssim\left\|\left(\sum_{k \in \mathbb{Z}}
\left\{\cm\left(\left[Y_{\psi}^{(k)}(f)\right]^{\eta}\right)(x)\right\}
^{\frac{2}{\eta}}
\right)^{^{\frac{\eta}{2}}}\right\|_{{X}^{\frac{1}{\eta}}}^{\frac{1}{\eta}} \\
& \lesssim\left\|\left(\sum_{k \in \mathbb{Z}}
\left[Y_{\psi}^{(k)}(f)\right]^{2}\right)^{\frac{1}{2}}\right\|_{X} = \left\|S_{\psi}(f)
\right\|_{X},
\end{align*}
which further implies that
\eqref{s4e1} holds true and hence completes the proof of
Theorem \ref{s4t2}.
\end{proof}

Now, we recall the concept of the anisotropic weight
class of Muckenhoupt, associated with a dilation $A$,
which was introduced in \cite[Definition 2.4]{bh06}.

\begin{definition}
Let $A$ be a dilation, $p\in[1,\infty)$,
and $w$ be a nonnegative measurable function on $\rn$.
The function $w$ is said to belong to the
{\it anisotropic weight class of Muckenhoupt},
$\ca_p(A):=\ca_p(\rn,A)$,
if there exists a positive constant $C$ such that,
when $p\in(1,\fz)$,
\begin{align*}
&\sup_{x\in \rn}\sup_{k\in\zz}
\lf\{\fint_{x+B_k}w(y)\,dy\r\}
\lf\{\fint_{x+B_k}\lf[w(y)\r]^
{-\frac{1}{p-1}}\,dy\r\}^{p-1}\leq C
\end{align*}
or, when $p=1$,
\begin{equation*}
\sup_{x\in\rn}\sup_{k \in \zz}\lf\{
\fint_{x+B_k}w(y)\,dy\r\}
\lf\{\esssup_{y\in x+B_k}\lf[w(y)\r]^{-1}\r\}\leq C.
\end{equation*}
Moreover, the minimal constants $C$ as above are denoted by $C_{p,A,n}(w)$.
\end{definition}

It is easy to prove that, if $1\leq p\leq q \leq\infty$,
then $\ca_p(A)\subset \ca_q(A)$.
Let $$\ca_{\fz}(A):=\bigcup_{q\in[1,\fz)}\ca_{q}(A).$$
For any given $w\in\ca_\infty(A)$,
define the {\it critical index} $q_w$ of $w$ by setting
\begin{equation}\label{qw}
q_w:=\inf\lf\{p\in[1,\infty):\  w\in\ca_p(A)\r\}.
\end{equation}
Obviously, $q_w\in[1,\infty)$. By the reverse H\"older inequality
(see, for instance, \cite[Theorem 1.2]{hpr12}), we conclude that,
for any $p\in(1,\fz)$ and $w\in\ca_p(A)$, there exists an
$\epsilon\in(0,p-1]$ such that $w\in\ca_{p-\epsilon}(A)$.
Thus, if $q_w\in(1,\fz)$, then $w\notin\ca_{q_w}(A)$. Moreover,
Johnson and Neugebauer \cite[p.\,254]{jn87}
gave an example of $w\notin\ca_1(A)$ with $A=2I_{n\times n}$
such that $q_w=1$.

In what follows, for any nonnegative local integrable function $w$
and any Lebesgue measurable set $E$, let $$w(E):=\int_Ew(x)\,dx.$$
For any given $p\in(0,\infty)$, denote by $L_w^p(\rn)$ the
{\it set of all the measurable functions} $f$ on $\rn$ such that
\begin{equation*}
\|f\|_{L_w^p(\rn)}
:=\lf\{\int_{\rn}|f(x)|^pw(x)\,dx\r\}^\frac{1}{p}<\infty.
\end{equation*}
Moreover, let $L_w^\infty(\rn):=L^\infty(\rn)$.
Obviously, $L_w^p(\rn)$ is a ball quasi-Banach function space,
which even may not be a quasi-Banach function space
(see, for instance, \cite[p. 86]{shyy17}).

To show Theorem \ref{s4t1}, we need the following several
technical lemmas.
Lemma \ref{embed} is a direct corollary of \cite[Lemma 4.9]{yhyy}
because $(\rn,\rho,dx)$ is a special space of  homogeneous type;
Lemma \ref{inclu} is similar to \cite[p.\,21, Theorem 4.5]{Bownik} and
we omit the details.

\begin{lemma}\label{embed}
Let $A$, $X$, and $\tz_0$ be the same as in Theorem \ref{s4t1}.
Assume that $x_0\in\rn$. Then there exists an $\epsilon\in(0,1)$
such that $X$ continuously embeds into $L_w^{\tz_0}(\rn)$,
where $w:=[\cm(\mathbf1_{x_0+B_0})]^\epsilon$ and $B_0$ is the same
as in \eqref{B_k} with $k=0$.
\end{lemma}

\begin{lemma}\label{inclu}
Let $A$ and $X$ be the same as in Theorem \ref{s4t1}.
Then $H_X^{A}(\rn)\subset \cs'(\rn)$ and the inclusion is continuous.
\end{lemma}

Combining Lemmas \ref{embed} and \ref{inclu},
we obtain the following property of $H_X^A(\rn)$.

\begin{lemma}\label{HXAvanish}
Let $A$ and $X$ be the same as in Theorem \ref{s4t1}
and $f\in H_X^A(\rn)$. Then $f$ vanishes weakly at infinity.
\end{lemma}

\begin{proof}
Let $N\in\nn$ be the same
as in \eqref{3.14.x1}. By Lemma \ref{inclu}, we find that,
for any $k\in\zz$, $\varphi\in\cs(\rn)$, $x\in\rn$, and $y\in x+B_k$,
$|f\ast\varphi_{k}(x)|\lesssim M_N(f)(y)$.
Thus, there exists a positive constant $C_1$ such that,
for any $k\in\zz$, $\varphi\in\cs(\rn)$, and $x\in\rn$,
$$x+B_k\subset\lf\{y\in\rn:\  M_N(f)(y)>C_1|f\ast\varphi_{k}(x)|\r\}.$$
By this and Lemma \ref{embed}, we conclude that,
for any $k\in\zz$, $\varphi\in\cs(\rn)$, and $x\in\rn$,
\begin{align*}
|f\ast\varphi_{k}(x)|
&=[w(B_k)]^{-\frac{1}{\tz_0}}[w(B_k)]^{\frac{1}{\tz_0}}|f\ast\varphi_{k}(x)|\\
&\le[w(B_k)]^{-\frac{1}{\tz_0}}\lf[w\lf(\lf\{y\in\rn:\
M_N(f)(y)>C_1|f\ast\varphi_{k}(x)|\r\}\r)\r]^{\frac{1}{\tz_0}}\\
&\quad\times|f\ast\varphi_{k}(x)|\\
&\lesssim[w(B_k)]^{-\frac{1}{\tz_0}}\lf\|M_N(f)\r\|_{L_w^{\tz_0}(\rn)}
\lesssim[w(B_k)]^{-\frac{1}{\tz_0}}\lf\|M_N(f)\r\|_{X}\\
&= [w(B_k)]^{-\frac{1}{\tz_0}}||f||_{H_X^A(\rn)}\rightarrow 0
\end{align*}
as $k\rightarrow \infty$, which further implies that $f$ vanishes weakly at infinity.
This finishes the proof of Lemma \ref{HXAvanish}.
\end{proof}

To show Theorem \ref{s4t1}, we also need the following lemma
whose proof is similar to that of \cite[Lemma 4.2]{lhy20}; we omit the details here.
\begin{lemma}\label{s4l3}
Let $A$, $X$, $\tz_0$, and $p_0$ be the same
as in Theorem \ref{s4t1},
$q\in(\max\{p_0,1\},\infty]$, $k_0\in\zz$, and $\varepsilon\in(0,\infty)$.
Assume that $\{\lambda_i\}_{i\in\nn}\subset[0,\infty)$,
$\{B^{(i)}\}_{i\in\nn}\subset\CB$,
and $\{m_i^{(\varepsilon)}\}_{i\in\nn}\subset L^q(\rn)$ satisfy that,
for any $\varepsilon\in(0,\infty)$ and $i\in\nn$,
$$
\supp m_i^{(\varepsilon)}:=\left\{x\in\rn:\ m_i^{(\varepsilon)}
\neq 0\right\}\subset A^{k_0}B^{(i)},
$$
$$
\|m_i^{(\varepsilon)}\|_{L^q(\rn)}
\leq\frac{|B^{(i)}|^{\frac{1}{q}}}{\|\one_{B^{(i)}}\|_X},
$$
and
$$
\left|\left|\left\{\sum_{i \in \mathbb{N}}
\left[\frac{\lambda_i\one_{B^{(i)}}}{||\one_{B^{(i)}}||_X}\right]^{\tz_0}
\right\}^{\frac{1}{\tz_0}}\right|\right|_X<\infty.
$$
Then
$$
\left|\left|\liminf_{\varepsilon\rightarrow0^+}\left[
\sum_{i \in \mathbb{N}}\lf|\lambda_i m_i^{(\varepsilon)}\r|^{\tz_0}\right]^
{\frac{1}{\tz_0}}
\right|\right|_X\leq C \left|\left|\left\{\sum_{i \in \mathbb{N}}
\left[\frac{\lambda_i\one_{B^{(i)}}}
{||\one_{B^{(i)}}||_X}\right]^{\tz_0}\right\}^{\frac{1}{\tz_0}}
\right|\right|_X,
$$
where $C$ is a positive constant independent of $\lambda_{i}$, $B^{(i)}$,
$m_i^{(\varepsilon)}$, and $\varepsilon$.
\end{lemma}

Now, we prove Theorem \ref{s4t1}.
\begin{proof}[Proof of Theorem \ref{s4t1}]
Let $\tau$ be the same
as in \eqref{tau} and $u$ and $v$ the same
as in Lemma \ref{defdya}(iv).
We first show the necessity of the present theorem.
To this end, let $f\in H_X^A(\rn)$. Then,
by Lemma \ref{HXAvanish}, we find that $f$ vanishes weakly at infinity.
On the other hand, it follows from \cite[Theorem 4.3]{wyy22}
that there exists a sequence
$\left\{\lambda_{i}\right\}_{i \in \mathbb{N}} \subset [0,\infty)$
and a sequence $\left\{a_{i}\right\}_{i \in \mathbb{N}}$
of anisotropic $(X, q, d)$-atoms
supported,
respectively, in $\{B^{(i)}\}_{i \in \mathbb{N}} \subset \CB$ such that
$$
f=\sum_{i \in \mathbb{N}} \lambda_{i} a_{i} \quad \text { in }
\mathcal{S}^{\prime}\left(\mathbb{R}^{n}\right)
$$
and
$$
\|f\|_{H_{X}^A\left(\mathbb{R}^{n}\right)}
\sim\left\|\left\{\sum_{i \in \mathbb{N}}
\left[\frac{\lambda_{i}
\one_{B^{(i)}}}{\|\one_{B^{(i)}}\|_{X}}\right]^{\tz_0}\right\}
^{\frac{1}{\tz_0}}\right\|_{X}.
$$
Let $a$ be an $(X,q,d)$-atom supported in a dyadic cube $Q$.
Let $w:=u-v+2 \tau$ and, for any $j\in\nn$,
$U_j:=x_Q+(B_{v[\ell(Q)-j-1]+2\tau}\setminus B_{v[\ell(Q)-j]+2\tau})$.
Then, by Lemma \ref{defdya}(iv), we conclude
that, for any $x\in(A^wQ)^\com$,
there exists some $j_0\in\nn$ such that $x\in U_{j_0}$.
For this $j_0$, choose an $N\in\nn$ lager enough such that
$$
(N-\beta)vj_0+\lf(\frac{1}{q}-\beta\r)u<0,
$$
where $\beta:=
\left(\frac{\ln b}{\ln \lambda_{-}}+d+1\right)
\frac{\ln \lambda_{-}}{\ln b}>\frac{1}{\tz_0}$.
By this and an argument similar to that used in the proof of
\cite[(3.3)]{lyy2018},
we find that, for any $x\in(A^wQ)^\com$,
$$
S(a)(x)\lesssim b^{Nvj_0}b^{-\frac{v\ell(Q)}{q}}||a||_{L^q(Q)}.
$$
From this, the size condition of $a$, and Lemma $\ref{defdya}$(iv),
we deduce that, for any $x\in(A^wQ)^\com$,
\begin{align*}
S(a)(x)&\lesssim b^{Nvj_0}b^{-\frac{v\ell(Q)}{q}}||\one_Q||_{X}^{-1}
\lf|B_{v\ell(Q)+u}\r|^\frac{1}{q}\\
&\leq b^{(N-\beta)vj_0+(\frac{1}{q}-\beta)u}||\one_Q||_{X}^{-1}
\frac{|Q|^\beta}{b^{[\ell(Q)-j_0]v\beta}}\\
&\lesssim||\one_Q||_{X}^{-1}\left[\frac{|Q|}{\rho(x-x_Q)}\right]^\beta
\leq||\one_Q||_{X}^{-1}\lf[\cm(\one_Q)(x)\r]^\beta.
\end{align*}
Using this, we obtain, for any $x \in \mathbb{R}^{n}$,
\begin{align}\label{s4e5}
S(f)(x)&\leq\sum_{i \in \mathbb{N}}\left|\lambda_{i}\right|
S(a_{i})(x)\one_{A^{w} B^{(i)}}(x)
+\sum_{i \in \mathbb{N}}\left|\lambda_{i}\right|
S(a_{i})(x)
\one_{\left(A^{w} B^{(i)}\right)^{\com}}(x)\\
&\lesssim\left\{\sum_{i \in \mathbb{N}}\left[\left|\lambda_{i}\right|S(a_{i})(x)
\one_{A^{w}B^{(i)}}(x)\right]^{\tz_0}\right\}^{\frac{1}{\tz_0}}\noz\\
&\quad+\sum_{i\in\mathbb{N}}
\frac{|\lambda_{i}|}{\|\one_{B^{(i)}}\|_{X}}
\left[\cm\left(\one_{B^{(i)}}\right)(x)\right]^{\beta}\notag.
\end{align}
By \eqref{s4e5}, Assumptions \ref{Assum-1} and \ref{Assum-2},
and an argument similar to
that used in the proof of \cite[Theorem 4.3]{wyy22},
we further conclude that
$$\|S(f)\|_{X} \lesssim\|f\|_{H_{X}^{A}\left(\mathbb{R}^{n}\right)},$$
which completes the proof of the necessity of the present theorem.

Next, we show the sufficiency of the present theorem.
Let $\psi$ and $\phi$ be the same
as in Lemma \ref{s4l1} with $d$ in \eqref{def-d},
$f$ vanish weakly at infinity, and $\|S(f)\|_X<\infty$.
Then, from Theorem \ref{s4t2}, we infer that $S_{\psi}(f) \in X$.
Thus, to show the sufficiency of the present theorem,
we need to prove that $f \in H_{X}^{A}\left(\mathbb{R}^{n}\right)$ and
\begin{equation}\label{s4e6}
\|f\|_{H_{X}^{A}\left(\mathbb{R}^{n}\right)} \lesssim\left\|S_{\psi}(f)\right\|_{X}.
\end{equation}
To this end, for any $k \in \mathbb{Z}$,
let $\Omega_{k}:=\{x \in \mathbb{R}^{n}: S_{\psi}(f)(x)>2^{k}\}$ and
$$
\mathcal{Q}_{k}:=
\left\{Q \in \mathcal{Q}:\ \left|Q \cap \Omega_{k}\right|>\frac{|Q|}{2}
\text { and }\left|Q \cap \Omega_{k+1}\right| \leq \frac{|Q|}{2}\right\} .
$$
Clearly, for any $Q \in \mathcal{Q}$, there exists a unique $k \in \mathbb{Z}$
such that $Q \in \mathcal{Q}_{k}$.
Let $\{Q_{i}^{k}\}_{i}$ be the set of all maximal dyadic cubes in $\mathcal{Q}_{k}$,
that is, there exists no $Q \in \mathcal{Q}_{k}$
such that $Q_{i}^{k}\subsetneqq Q$ for any $i$.

For any $Q \in \mathcal{Q}$, let
\begin{align}\label{s4e7}
\widehat{Q}:=&\left\{(y, t) \in \mathbb{R}_{+}^{n+1}:=\mathbb{R}^{n}
\times(0, \infty):\r.\\
&\lf.\ y \in Q,\, b^{v \ell(Q)+u+\tau}\leq t<b^{v [\ell(Q)-1]+u+\tau}
\right\}.\noz
\end{align}
Obviously, $\{\widehat{Q}\}_{Q \in \mathcal{Q}}$ are mutually disjoint and
\begin{equation}\label{s4e8}
\mathbb{R}_{+}^{n+1}=\bigcup_{k \in \zz} \bigcup_{i} B_{k, i},
\end{equation}
where, for any $k \in \mathbb{Z}$ and
$i, B_{k, i}:=\bigcup_{Q \subset Q_{i}^{k}, Q \in \mathcal{Q}_{k}} \widehat{Q}$.
Then, by Lemma \ref{defdya}(ii) and \eqref{s4e7},
we easily find that $\left\{B_{k, i}\right\}_{k \in \mathbb{Z}, i}$
are also mutually disjoint.

On the other hand, $\psi$ has the vanishing moments up to order $d$.
From Lemma \ref{def-d}, the properties of tempered distributions
(see, for instance, \cite[Theorem 2.3.20]{GTM249}), and \eqref{s4e8},
we deduce that, for any $f\in\cs'(\rn)$
vanishing weakly at infinity and satisfying $\|S(f)\|_X<\infty$
and for any $x \in \mathbb{R}^{n}$, we have
\begin{align}\label{s4e9}
f(x)&=\sum_{k \in \mathbb{Z}} f \ast \psi_{k} \ast
\phi_{k}(x)\\
&=\int_{\mathbb{R}_{+}^{n+1}} (f \ast \psi_{t})(y)
\phi_{t}(x-y)\,dy\,dm(t)\noz
\end{align}
in $\cs'(\rn)$,
where $m(t)$ denotes the counting measure on $\mathbb{R}$,
that is, for any set $E\subset\rr$, $m(E)$ is the number of
integers contained in $E$ if $E$ has only finitely many elements,
or else $m(E):=\infty$. For any $k \in \mathbb{Z}$, $i$,
and $x \in \mathbb{R}^{n}$, let
$$
h_{i}^{k}(x):=\int_{B_{k, i}} (f \ast \psi_{t})(y)\phi_{t}(x-y)\,dy\,dm(t).
$$
Next, we prove the sufficiency of the present theorem in three steps.

Step (1) The target of this step is to show that
\begin{equation}\label{s4e10}
\sum_{k \in \zz} \sum_{i} h_{i}^{k} \text { converges in }
\mathcal{S}^{\prime}\left(\mathbb{R}^{n}\right).
\end{equation}
To this end, following the proofs of assertions (i) and (ii)
in the proof of the sufficiency of \cite[Theorem 3.4(i)]{lhy20}
with some slight modifications, we conclude that,
for any given $q\in(\max\{p_0,1\},\infty)$,
\begin{enumerate}
\item[{\rm (i)}]
for any $k \in \mathbb{Z}$, $i$, and $x \in \mathbb{R}^{n}$,
\begin{equation*}
h_{i}^{k}(x)=\sum_{Q \subset Q_{i}^{k}, Q \in \mathcal{Q}_{k}} \int_{\widehat{Q}}
(f \ast \psi_{t})(y) \phi_{k}(x-y)\,dy\,dm(t)
\end{equation*}
holds true in $L^{q}\left(\mathbb{R}^{n}\right)$ and
hence also in $\mathcal{S}^{\prime}\left(\mathbb{R}^{n}\right)$;
\item[{\rm (ii)}]
for any $k \in \mathbb{Z}$ and $i$,
$h_{i}^{k}=\lambda_{i}^{k} a_{i}^{k}$ is a multiple of an anisotropic
$(X, q, d)$-atom, where, for any $k \in$ $\mathbb{Z}$ and $i$,
$\lambda_{i}^{k} \sim 2^{k}||\one_{B_{i}^{k}}||_{X}$ with the positive
equivalence constants independent of both
$k$ and $i$, and $a_{i}^{k}$ is an
anisotropic $(X, q, d)$-atom satisfying,
for any $q\in(\max\{p_0,1\},\infty)$,
$k \in \mathbb{Z}$, $i$, and $\gamma \in \mathbb{Z}_{+}^{n}$,
$$
\supp a_{i}^{k} \subset B_{i}^{k}:=
x_{Q_{i}^{k}}+B_{v\left[\ell\left(Q_{i}^{k}\right)-1\right]+u+3 \tau},$$
$$
\|a_{i}^{k}\|_{L^{q}\left(\mathbb{R}^{n}\right)}
\leq\|\one_{B_{i}^{k}}\|_{X}^{-1}
|B_{i}^{k}|^{\frac{1}{r}},\ \mathrm{and}\
\int_{\mathbb{R}^{n}}a_{i}^{k}(x)x^{\gamma}\,dx=0.
$$
\end{enumerate}

To show \eqref{s4e10}, we next consider two cases:
$i \in \mathbb{N}$ and $i \in\{1, \ldots, I\}$ with some $I \in \mathbb{N}$.

Case 1) $i \in \mathbb{N}$.
In this case, to prove \eqref{s4e10},
by Lemma \ref{inclu}, it suffices to show that
\begin{equation}\label{s4e11}
\lim _{l \rightarrow \infty}
\left\|
\sum_{l \leq|k| \leq m} \sum_{l \leq i \leq m} \lambda_{i}^{k} a_{i}^{k}
\right\|_{H_{X}^{A}\left(\mathbb{R}^{n}\right)}=0.
\end{equation}
Indeed, for any $k \in \mathbb{Z}$ and $i \in \mathbb{N}$,
by the estimate that $\left|Q_{i}^{k} \cap \Omega_{k}\right|
\geq \frac{\left|Q_{i}^{k}\right|}{2}$, we find that,
for any $x \in \mathbb{R}^{n}$,
$$
\cm\left(\one_{Q_{i}^{k} \cap \Omega_{k}}\right)(x)
\gtrsim \fint_{Q_{i}^{k}}
\one_{Q_{i}^{k} \cap \Omega_{k}}(y) \,dy = \frac{|Q_{i}^{k} \cap \Omega_{k}|}
{|Q_{i}^{k}|} \geq \frac{1}{2}.
$$
This, together with Assumption \ref{Assum-1},
further implies that, for any $l,m \in \mathbb{N}$,
\begin{align}\label{s4e12}
\left\|\sum_{l \leq|k| \leq m } \sum_{l\leq i\leq m}
\left(2^{k}\one_{B_{i}^{k}}\right)
^{\tz_0} \right\|_{X^{\frac{1}{\tz_0}}}^{\frac{1}{\tz_0}}
&=\left\|\left[\sum_{l \leq|k| \leq m } \sum_{l \leq i \leq m}2^{k\tz_0}
\left(\one_{B_{i}^{k}}\right)^{2}\right]^{\frac{1}{2}}
\right\|_{X^{\frac{2}{\tz_0}}}^{\frac{2}{\tz_0}}\\
&\lesssim\left\|\left\{\sum_{l \leq|k| \leq m }
\sum_{l \leq i \leq m}{2^{k {\tz_0}}}\left[\cm\left(\one_{Q_{i}^{k}
\cap \Omega_{k}}\right)\right]^{2}\right\}^{\frac{1}{2}}
\right\|_{X^{\frac{2}{\tz_0}}}^{\frac{2}{\tz_0}}\notag\\
&\lesssim\left\|\sum_{l \leq|k| \leq m }
\sum_{l \leq i \leq m}\left(2^{k}
\one_{Q_{i}^{k}\cap\Omega_{k}}\right)^{\tz_0}\right\|
_{X^{\frac{1}{\tz_0}}}^{\frac{1}{\tz_0}}.\notag
\end{align}
In addition, from the fact that, for any $l,m \in \mathbb{N}$,
$\sum_{l \leq|k| \leq m } \sum_{l \leq i \leq m}
\lambda_{i}^{k} a_{i}^{k} \in H_{X}^{A}\left(\mathbb{R}^{n}\right)$,
Lemma \ref{finatomth}(i), and Definition \ref{Debf}(i), we deduce that
\begin{align}\label{s4e13}
\left\|\sum_{l \leq|k| \leq m}
\sum_{l \leq i \leq m} \lambda_{i}^{k} a_{i}^{k}\right\|_{H_{X}^{A}(\rn)}
&\lesssim\left\|\left\{\sum_{l \leq|k| \leq m} \sum_{l \leq i \leq m}
\left[\frac{\lambda_{i}^{k} \one_{B_{i}^{k}}}{\|\one_{B_{i}^{k}}\|_{X}}\right]
^{{\tz_0}}\right\}^{\frac{1}{\tz_0}}\right\|_{X}\\
&\sim\left\|\left[\sum_{l \leq|k| \leq m}
\sum_{l \leq i \leq m}\left(2^{k} \one_{B_{i}^{k}}\right)^{{\tz_0}}\right]
^{\frac{1}{\tz_0}}\right\|_{X}\notag\\
&=\left\|\sum_{l \leq|k| \leq m}
\sum_{l \leq i \leq m}\left(2^{k} \one_{B_{i}^{k}}\right)^{{\tz_0}}\right\|
_{X^{\frac{1}{\tz_0}}}^{\frac{1}{\tz_0}}\notag.
\end{align}
On the other hand, it follows from Definition \ref{BQBFS} that,
for any $l,m \in \mathbb{N}$,
\begin{align*}
\left\|\left[\sum_{l \leq|k| \leq m}\left(2^{k}
\one_{\Omega_{k}}\right)^{{\tz_0}}\right]
^{\frac{1}{\tz_0}}\right\|^{\tz_0}_X
&=\left\|\left[\sum_{l \leq|k| \leq m}\left(2^{k} \one_{\Omega_{k}
\setminus\Omega_{k+1}}+2^{k}
\one_{\Omega_{k+1}}\right)^{{\tz_0}}\right]^{\frac{1}{\tz_0}}
\right\|_{X}^{{\tz_0}}\\
&\lesssim\left\|\left[\sum_{l \leq|k| \leq m}
\left(2^{k} \one_{\Omega_{k} \setminus  \Omega_{k+1}}\right)
^{{\tz_0}}\right]^{\frac{1}{\tz_0}}\right\|_{X}^{{\tz_0}}\\
&\quad+\left(\frac{1}{2}\right)
^{{\tz_0}}\left\|\left[\sum_{l \leq|k| \leq m}
\left(2^{k+1} \one_{\Omega_{k+1}}\right)^{{\tz_0}}\right]
^{\frac{1}{\tz_0}}\right\|_{X}^{{\tz_0}} .
\end{align*}
Therefore, as $l \rightarrow \infty$, we have
\begin{align}\label{s4e14}
\left\|\left[\sum_{l \leq|k| \leq m}\left(2^{k} \one_{\Omega_{k}}\right)
^{{\tz_0}}\right]^{\frac{1}{\tz_0}}\right\|_{X}
\sim\left\|\left[\sum_{l \leq|k| \leq m}\left(2^{k} \one_{\Omega_{k}
\setminus  \Omega_{k+1}}\right)^{{\tz_0}}\right]
^{\frac{1}{\tz_0}}\right\|_{X} .
\end{align}
This, combined with \eqref{s4e12} and \eqref{s4e13}, further implies that,
as $l \rightarrow \infty$,
\begin{align*}
&\left\|\sum_{l \leq|k| \leq m}
\sum_{l \leq i \leq m} \lambda_{i}^{k} a_{i}^{k}\right\|
_{H_{X}^{A}\left(\mathbb{R}^{n}\right)}\\
&\quad\lesssim\left\|\sum_{l \leq|k| \leq m} \sum_{l \leq i \leq m}
\left(2^{k} \one_{Q_{i}^{k} \cap \Omega_{k}}\right)
^{{\tz_0}}\right\|_{X^{\frac{1}{\tz_0}}}^{\frac{1}{\tz_0}}
\lesssim\left\|\left[\sum_{l \leq|k| \leq m}\left(2^{k} \one_{\Omega_{k}}\right)
^{{\tz_0}}\right]^{\frac{1}{\tz_0}}\right\|_{X}\\
&\quad\sim\left\|\left[\sum_{l \leq|k| \leq m}\left(2^{k} \one_{\Omega_{k}
\setminus  \Omega_{k+1}}\right)^{{\tz_0}}\right]^{\frac{1}{\tz_0}}\right\|_{X}\\
&\quad\leq\left\|S_{\psi}(f)\left(\sum_{l \leq|k| \leq m} \one_{\Omega_{k}
\setminus  \Omega_{k+1}}\right)^{\frac{1}{\tz_0}}\right\|_{X} \rightarrow 0 .
\end{align*}
Thus, \eqref{s4e11} holds true and so \eqref{s4e10} does  in Case 1).

Case 2) $i \in\{1, \ldots, I\}$ with some $I \in \mathbb{N}$.
In this case, to show \eqref{s4e10}, by Lemma \ref{inclu},
it suffices to prove that
\begin{equation}\label{s4e15}
\lim _{l \rightarrow \infty}\left\|
\sum_{l\leq|k|\leq m}\sum_{i=1}^{I} \lambda_{i}^{k} a_{i}^{k}\right\|
_{H_{X}^{A}\left(\rn\right)}=0 .
\end{equation}
Indeed, by a proof similar to that of \eqref{s4e11},
it is easy to show that \eqref{s4e15} also holds true.
This finishes the proof of $\eqref{s4e10}$ in Case 2) and hence $\eqref{s4e10}$.

Step (2) In this step, we prove that
\begin{equation}\label{s4e16}
f=\sum_{k \in \mathbb{Z}} \sum_{i} \lambda_{i}^{k} a_{i}^{k}\
\text{in}\ \mathcal{S}^{\prime}\left(\mathbb{R}^{n}\right).
\end{equation}
To this end, for any $x \in \mathbb{R}^{n}$, let
$$
\widetilde{f}(x):=\sum_{k \in \mathbb{Z}} \sum_{i} h_{i}^{k}(x)
=\sum_{k \in \mathbb{Z}} \sum_{i} \int_{B_{k, i}}
(f \ast \psi_{t})(y) \phi_{k}(x-y)\,dy\,dm(t)
$$
in $\mathcal{S}^{\prime}\left(\mathbb{R}^{n}\right)$, where,
for any $k \in \mathbb{Z}$ and $i, B_{k, i}$ is the same
as in \eqref{s4e8}.
Then, to show \eqref{s4e16}, it suffices to prove that
\begin{equation}\label{s4e16plus}
f=\widetilde{f}\ \text{in}\ \mathcal{S}^{\prime}\left(\mathbb{R}^{n}\right).
\end{equation}
For this purpose, by the above assertion (i) and \eqref{s4e7},
we find that, for any given
$k,i\in\mathbb{Z}$, $q\in(\max\{p_0,1\},\infty)$, and $x \in \mathbb{R}^{n}$,
\begin{align}\label{s4e17}
h_{i}^{k}(x) &=\lim _{N \rightarrow \infty} \int_{0}^{\infty}
\int_{\mathbb{R}^{n}}(f\ast\psi_{t})(y)\phi_{k}(x-y)\\
&\quad\times\one_{\bigcup_{\substack{Q\subset Q_{i}^{k},
Q\in\mathcal{Q}_{k}\\|\ell(Q)| \leq N}}
\widehat{Q}}(y, t)\,dy\,dm(t)\noz\\
&=\lim _{N \rightarrow \infty} \int_{\gamma(N)}^{\eta(N)}
\int_{\mathbb{R}^{n}} (f \ast \psi_{t})(y)
\phi_{k}(x-y)\one_{B_{k, i}}(y, t)\,dy\,dm(t)\notag
\end{align}
holds true in $L^{q}\left(\mathbb{R}^{n}\right)$
and also in $\mathcal{S}^{\prime}\left(\mathbb{R}^{n}\right)$,
where, for any $N \in \mathbb{N}$,
$ \gamma(N):= b^{vN+u+1}$ and $\eta(N):=b^{-v(N+1)+u+1}$.
For the convenience of symbols, we rewrite $\widetilde{f}$ as,
for any $x \in \mathbb{R}^{n}$,
$$
\widetilde{f}(x)=\sum_{\ell \in \mathbb{N}} \int_{R^{(\ell)}}
(f \ast \psi_{t})(y) \phi_{t}(x-y)\,dy\,dm(t),
$$
where $\{R^{(\ell)}\}_{\ell \in \mathbb{N}}$ is an
arbitrary permutation of $\left\{B_{k, i}\right\}_{k \in \mathbb{Z}, i}$.
For any $L \in \mathbb{N}$ and $x \in \mathbb{R}^{n}$, let
$$
\widetilde{f}_{L}(x):=f(x)-\sum_{\ell=1}^{L}
\int_{R^{(\ell)}} (f \ast \psi_{t})(y) \phi_{t}(x-y)\,dy\,dm(t).
$$
Then, from \eqref{s4e8}, \eqref{s4e9}, and \eqref{s4e17},
it follows that, for any $L \in \mathbb{N}$ and $x \in \mathbb{R}^{n}$,

\begin{align}\label{s4e18}
\widetilde{f}_{L}(x)
=& \lim _{N \rightarrow \infty}
\int_{\gamma(N)}^{\eta(N)} \int_{\mathbb{R}^{n}}
(f \ast \psi_{t})(y) \phi_{t}(x-y)\one_{\cup_{\ell=1}^{\infty}
R^{(\ell)}}(y, t) \,dy\,dm(t)\\
& \quad-\lim_{N \rightarrow \infty} \int_{\gamma(N)}^{\eta(N)}
\int_{\mathbb{R}^{n}} (f \ast \psi_{t})(y) \phi_{t}(x-y)
\one_{\cup_{\ell=1}^{L} R^{(\ell)}}(y, t)\,dy\,dm(t) \notag\\
=& \lim_{N \rightarrow \infty} \int_{\gamma(N)}^{\eta(N)}
\int_{\mathbb{R}^{n}} (f \ast \psi_{t})(y) \phi_{t}(x-y)
\one_{\cup_{\ell=L+1}^{\infty} R^{(\ell)}}(y, t) \,dy\,dm(t)\notag
\end{align}
holds true in $\mathcal{S}^{\prime}\left(\mathbb{R}^{n}\right)$.

Note that $H_{X}^{A}\left(\mathbb{R}^{n}\right)$
is continuously embedded into $\mathcal{S}^{\prime}\left(\mathbb{R}^{n}\right)$
(Lemma \ref{inclu}). Thus, to prove \eqref{s4e16plus}, we only need to show that
\begin{equation}\label{s4e19}
\left\|\widetilde{f}_{L}\right\|_{H_{X}^{A}\left(\mathbb{R}^{n}\right)}
\rightarrow 0\  \text { as } \ L \rightarrow \infty .
\end{equation}
To do this, we borrow some ideas from the proof of the atomic characterization
of $H_{X}^{A}\left(\mathbb{R}^{n}\right)$ (see
the proof of \cite[Theorem 4.3]{wyy22}).
Indeed, for any $\varepsilon \in(0,1), L \in \mathbb{N}$, and $x \in \mathbb{R}^{n}$,
let
\begin{align*}
&\widetilde{f}_{L}^{(\varepsilon)}(x):=
\int_{\varepsilon}^{\alpha / \varepsilon} \int_{\mathbb{R}^{n}}
(f \ast \psi_{t})(y) \phi_{t}(x-y) \one_{\cup_{\ell=L+1}^{\infty}
R^{(\ell)}}(y, t) \,dy\,dm(t),
\end{align*}
where $\alpha:=b^{-v+2(u+1)}$.
Then, by the Lebesgue dominated convergence theorem,
we find that, for any $\varepsilon \in(0,1), L \in \mathbb{N}$,
and $x \in \mathbb{R}^{n}$,
\begin{align*}
\widetilde{f}_{L}^{(\varepsilon)}(x)
&=\sum_{\ell=L+1}^{\infty}
\int_{\varepsilon}^{\alpha / \varepsilon}
\int_{\mathbb{R}^{n}} (f \ast \psi_{t})(y) \phi_{t}(x-y)
\one_{R^{(\ell)}}(y, t)\,dy\,dm(t)\\
&=:\sum_{\ell=L+1}^{\infty} h_{\ell}^{(\varepsilon)}(x)
\end{align*}
in $\mathcal{S}^{\prime}\left(\mathbb{R}^{n}\right)$. Moreover,
by some arguments similar to those used in the proofs of assertions
(i) and (ii) in the proof of the sufficiency of \cite[Theorem 3.4(i)]{lhy20}
with some slight modifications, we conclude that,
for any $\varepsilon \in(0,1), q\in(\max\{p_0,1\},\infty), L \in \mathbb{N}$,
and $\ell \in \mathbb{N} \cap[L+1, \infty), h_{\ell}^{(\varepsilon)}$
is a multiple of an anisotropic $(X, q, d)$-atom, that is, there exists
a sequence
$\{\lambda_{\ell}\}_{\ell \in \mathbb{N} \cap(L+1, \infty)} \subset [0,\infty)$
and a sequence
$\{a_{\ell}^{(\varepsilon)}\}_{\ell \in \mathbb{N} \cap(L+1, \infty)}$
of anisotropic $(X, q, d)$-atoms
supported, respectively,
in $\{B^{(\ell)}\}_{\ell \in \mathbb{N} \cap(L+1, \infty)} \subset \CB$
such that,
for any $\ell \in \mathbb{N} \cap[L+1, \infty), h_{\ell}^{(\varepsilon)}
=\lambda_{\ell} a_{\ell}^{(\varepsilon)}$,
where, for any $\ell\in\mathbb{N}\cap[L+1, \infty),\lambda_{\ell}$
and $B^{(\ell)}$ are independent of $\varepsilon$.
Therefore, for any $\varepsilon\in(0,1), L\in\mathbb{N}$,
and $x\in\mathbb{R}^{n}$,
\begin{equation}\label{s4e20}
\widetilde{f}_{L}^{(\varepsilon)}(x)
=\sum_{\ell=L+1}^{\infty}\lambda_{\ell} a_{\ell}^{(\varepsilon)}(x)
\ \text{in}\ \mathcal{S}^{\prime}\left(\mathbb{R}^{n}\right)
\end{equation}
and
\begin{equation}\label{s4e21}
\left\|\left\{\sum_{\ell=L+1}^{\infty}\left[\frac{\lambda_{\ell}
\one_{B^{(\ell)}}}{\|\one_{B^{(\ell)}}\|_{X}}\right]^{{\tz_0}}\right\}
^{1 / {\tz_0}}\right\|_{X}<\infty .
\end{equation}
On the other hand, for any given
\begin{equation*}
N_{0} \in \mathbb{N} \cap\left[\left\lfloor\left(\frac{1}{\tz_0}-1\right)
\frac{\ln b}{\ln \lambda_{-}}\right\rfloor+2, \infty\right),
\end{equation*}
let $M_{N_{0}}^{0}$ denote the radial grand maximal function
in Definition \ref{radialM} with $N$ replaced by $N_{0}$.
Then, by the just proved conclusion
that, for any $\varepsilon \in(0,1)$ and
$L \in \mathbb{N},\{a_{\ell}^{(\varepsilon)}\}_{\ell \in\nn \cap(L+1, \infty)}$
is a sequence of anisotropic $(X, q, d)$-atoms and \cite[Lemma 4.7]{wyy22},
we find that, for any $\ell \in \mathbb{N} \cap[L+1, \infty)$ and
$x \in \mathbb{R}^{n}$,
\begin{align}\label{s4e22}
M_{N_{0}}^{0}\left(a_{\ell}^{(\varepsilon)}\right)(x)
&\lesssim M_{N_{0}}^{0}\left(a_{\ell}^{(\varepsilon)}\right)(x)
\one_{A^{\tau} B^{(\ell)}}(x)+\frac{1}{\|\one_{B^{(\ell)}}\|_{X}}
\left[\cm\left(\one_{B^{(\ell)}}\right)(x)\right]^{\beta},
\end{align}
where $\beta:=\left(\frac{\ln b}{\ln \lambda_{-}}+d+1\right)
\frac{\ln \lambda_{-}}{\ln b}>\frac{1}{{\tz_0}} $.
Moreover, since $q>1$, then, from the boundedness of $\cm$
on $L^{q}\left(\mathbb{R}^{n}\right)$ (see \cite[Lemma 3.3(ii)]{lyy17}),
we deduce that, for any $\varepsilon \in(0,1), L \in \mathbb{N}$,
and $\ell \in \mathbb{N} \cap[L+1, \infty)$,
\begin{align*}
\left\|M_{N_{0}}^{0}\left(a_{\ell}^{(\varepsilon)}\right)
\one_{A^{\tau} B^{(\ell)}}\right\|_{L^q(\rn)}
&\lesssim\left\|\cm\left(a_{\ell}^{(\varepsilon)}\right)
\one_{A^{\tau} B^{(\ell)}}\right\|_{L^q(\rn)}\lesssim
\frac{|B^{(\ell)}|^{1/q}}{\|\one_{B^{(\ell)}}\|_X},
\end{align*}
which, combined with Lemma \ref{s4l3}, further implies that
\begin{align}\label{s4e23}
&\left\|\liminf _{\varepsilon \rightarrow 0^{+}}
\left\{\sum_{\ell=L+1}^{\infty}\left[\lambda_{\ell}
M_{N_{0}}^{0}\left(a_{\ell}^{(\varepsilon)}\right)
\one_{A^{\tau} B^{(\ell)}}\right]^{{\tz_0}}\right\}^{1 / {\tz_0}}\right\|_{X}\\
&\quad\lesssim\left\|\left\{\sum_{\ell=L+1}^{\infty}
\left[\frac{\lambda_{\ell} \one_{B^{(\ell)}}}{\|\one_{B^{(\ell)}}\|_{X}}\right]
^{{\tz_0}}\right\}^{1 / {\tz_0}}\right\|_{X} \noz.
\end{align}
In addition, let $\varepsilon:=\gamma(N)$ with $N \in \mathbb{N}
\cap[\lfloor\frac{-u-1}{v}\rfloor+1, \infty)$.
Then, by \eqref{s4e18}, we obtain, for any $x\in\rn$,
\begin{align*}
M_{N_{0}}^{0}\left(\widetilde{f}_{L}\right)(x)
&=M_{N_{0}}^{0}
\left(\lim _{N \rightarrow \infty} \widetilde{f}_{L}^{(\gamma(N))}\right)(x)\\
&=\sup_{\varphi\in\cs_N(\rn)}\sup_{k \in \zz}\lf|\lim _{N \rightarrow \infty}
\widetilde{f}_{L}^{(\gamma(N))}\ast\varphi_{k}(x)\r|\\
&\leq  \liminf _{N \rightarrow \infty} \sup_{\varphi\in\cs_N(\rn)}\sup_{k \in \zz}
\lf|\widetilde{f}_{L}^{(\gamma(N))}\ast\varphi_{k}(x)\r|\\
&= \liminf _{N \rightarrow \infty} M_{N_{0}}^{0}\left(\widetilde{f}_{L}^{(\gamma(N))}
\right) .
\end{align*}
From this, \cite[p.\,12, Proposition 3.10]{Bownik},
\eqref{s4e20}, and \eqref{s4e22}, it follows that, for any $L \in \mathbb{N}$,
\begin{align*}
\left\|\widetilde{f}_{L}\right\|_{H_{X}^{A}\left(\rn\right)}
& \leq\left\|\liminf _{N \rightarrow \infty} M_{N_{0}}^{0}
\left(\widetilde{f}_{L}^{(\gamma(N))}\right)\right\|_{X}\\
&\leq\left\|\liminf _{N \rightarrow \infty} \sum_{\ell=L+1}
^{\infty}\lambda_{\ell} M_{N_{0}}^{0}\left(a_{\ell}^{(\gamma(N))}\right)
\right\|_{X} \\
& \lesssim\left\|\liminf _{N \rightarrow \infty} \sum_{\ell=L+1}^{\infty}
\lambda_{\ell} M_{N_{0}}^{0}\left(a_{\ell}^{(\gamma(N))}\right)
\one_{A^{\tau} B^{(\ell)}}\right\|_{X}\\
&\quad+\left\|\sum_{\ell=L+1}^{\infty} \frac{\lambda_{\ell}}
{\|\one_{A^{\tau} B^{(\ell)}}\|_{X}}\left[\cm\left(\one_{B^{(\ell)}}\right)\right]
^{\beta}\right\|_{X}.
\end{align*}
This, together with \eqref{s4e23},
Lemma \ref{basicine}, Definition \ref{BQBFS}(ii),
Assumption \ref{Assum-1}, and $\beta>\frac{1}{{\tz_0}}$,
further implies that, for any $L \in \mathbb{N}$,
\begin{align*}
\left\|\widetilde{f}_{L}\right\|_{H_{X}^{A}\left(\rn\right)}
&\lesssim\left\|\liminf _{N \rightarrow \infty}
\left\{\sum_{\ell=L+1}^{\infty}\left[\lambda_{\ell} M_{N_{0}}^{0}
\left(a_{\ell}^{(\gamma(N))}\right) \one_{A^{\tau} B^{(\ell)}}\right]
^{{\tz_0}}\right\}^{\frac{1}{\tz_0}}\right\|_{X}\\
&\quad+\left\|\left\{\sum_{\ell=L+1}^{\infty}
\frac{\lambda_{\ell}}{\|\one_{B^{(\ell)}}\|_{X}}
\left[\cm\left(\one_{B^{(\ell)}}\right)\right]
^{\beta}\right\}^{\frac{1}{\beta}}\right\|_{X^{\beta}}^{\beta}\\
&\lesssim\left\|\left\{\sum_{\ell=L+1}^{\infty}
\left[\frac{\lambda_{\ell} \one_{B^{(\ell)}}}{\|\one_{B^{(\ell)}}\|_{X}}\right]
^{{\tz_0}}\right\}^{\frac{1}{\tz_0}}\right\|_{X} .
\end{align*}
By this and \eqref{s4e21}, we conclude that \eqref{s4e19} holds true,
which completes the proof of \eqref{s4e16plus} and hence \eqref{s4e16}.

Step (3) By \eqref{s4e16}, \cite[Theorem 4.3]{wyy22},
and some arguments similar to those used in the estimations
of both \eqref{s4e12} and \eqref{s4e14}, we conclude that
\begin{align*}
\left\|f\right\|_{H_X^{A}(\rn)}&\sim\left\|
\left\{\sum_{k \in \mathbb{Z}}
\sum_{i}\left[\frac{\lambda_{i}^{k}\one_{B_{i}^{k}}}{\|\one_{B_{i}^{k}}\|_X}
\right]
^{{\tz_0}}\right\}^{\frac{1}{\tz_0}}\right\|_{X}=\left\|\left[\sum_{k \in \mathbb{Z}}
\sum_{i}\left(2^{k} \one_{B_{i}^{k}}\right)^{{\tz_0}}\right]
^{\frac{1}{\tz_0}}\right\|_{X}\\
&\lesssim\left\|\sum_{k \in \mathbb{Z}} \sum_{i}\left(2^{k} \one_{Q_{i}^{k}
\cap \Omega_{k}}\right)^{{\tz_0}}\right\|_
{X^{\frac{1}{\tz_0}}}^{\frac{1}{\tz_0}}\leq\left\|\left[\sum_{k \in \mathbb{Z}}
\left(2^{k} \one_{\Omega_{k}}\right)
^{{\tz_0}}\right]^{\frac{1}{\tz_0}}\right\|_{X}\\
&\sim\left\|\left[\sum_{k \in \mathbb{Z}}\left(2^{k} \one_{\Omega_{k}
\setminus  \Omega_{k+1}}\right)^{{\tz_0}}\right]
^{\frac{1}{\tz_0}} \right\|_{X}\leq\left\|S_{\psi}(f)\left[\sum_{k \in \mathbb{Z}}
\one_{\Omega_{k}
\setminus  \Omega_{k+1}}\right]^{\frac{1}{\tz_0}}\right\|_{X}\\
&=\left\|S_{\psi}(f)\right\|_{X},
\end{align*}
which further
implies that $f \in H_{X}^{A}\left(\mathbb{R}^{n}\right)$ and
\eqref{s4e6} holds true. This finishes the proof the sufficiency
and hence Theorem \ref{s4t1}.
\end{proof}
Now, we establish the anisotropic Littlewood--Paley $g$-function
characterization of $H_X^A(\rn)$. Recall that, for any given dilation
$A$, $\phi\in \mathcal{S}\left(\mathbb{R}^{n}\right)$,
$t \in(0, \infty)$, and
$j \in \mathbb{Z}$
and for any $f \in \mathcal{S}^{\prime}\left(\mathbb{R}^{n}\right)$,
the \emph{anisotropic Peetre maximal function} $(\phi_{j}^{*} f)_{t}$
is defined by setting, for any $x \in \mathbb{R}^{n}$,
$$
(\phi_{j}^{*} f)_{t}(x):=\esssup_{y \in \mathbb{R}^{n}}
\frac{|(\phi_{-j} \ast f)(x+y)|}{[1+b^{j} \rho(y)]^{t}}
$$
and the \emph{$g$-function associated with $(\phi_{j}^{*} f)_{t}$}
is defined by setting, for any $x \in \mathbb{R}^{n}$,
\begin{equation*}
g_{t, *}(f)(x):=\left\{\sum_{j \in \mathbb{Z}}
\left[\left(\phi_{j}^{*} f\right)_{t}(x)\right]^{2}\right\}^{1 / 2}.
\end{equation*}
To prove Theorem \ref{s4t1'}, we need the following estimate
which is just \cite[Lemma 3.6]{lwyy2019} originated from \cite[(2.66)]{U}.
\begin{lemma}\label{s4l4}
Let $\phi$ be a radial function the same as in Lemma \ref{s4l1}.
Then, for any given $N_0\in\nn$ and $\gamma\in(0,\infty)$,
there exists a positive constant $C_{(N_0,\gamma)}$, depending
only on $N_0$ and $\gamma$,
such that,
for any $t\in(0,N_0)$, $l\in\zz$,
$f\in\cs'(\rn)$, and $x\in\rn$,
\begin{align*}
\lf[\lf(\phi_{l}^\ast f\r)_t(x)\r]^\gamma\le C_{(N_0,\gamma
)}\sum_{k=0}^\infty b^{-kN_0\gamma}b^{k+l}
\int_\rn\frac{|\phi_{-(k+l)}\ast f(y)|^\gamma}{[1+b^l\rho(x-y)]^{t\gamma}}\,dy.
\end{align*}
\end{lemma}

We now prove Theorem \ref{s4t1'}.

\begin{proof}[Proof of Theorem \ref{s4t1'}]
First, let $f \in H_{X}^{A}\left(\mathbb{R}^{n}\right)$.
Then, by Lemma \ref{HXAvanish},
we find that $f$ vanishes weakly at infinity.
In addition, repeating the proof of the necessity of Theorem \ref{s4t1}
with some slight modifications,
we easily find that $g(f)\in X$ and $||g(f)||_X\ls||f||_{H_X^A(\rn)}$.
Thus, to prove the present theorem, by Theorem \ref{s4t1},
we only need to show that, for any $f \in$ $\mathcal{S}'(\rn)$
satisfying that $f$ vanishes weakly at infinity
and $g(f)\in X$,
\begin{equation}\label{s4e25}
\|S(f)\|_{X} \lesssim\|g(f)\|_{X}
\end{equation}
holds true. Notice that, for any  $f \in \mathcal{S}'(\rn)$
vanishing weakly at infinity, any $t \in(0, \infty)$,
and almost every $x \in \mathbb{R}^{n}$, $S(f)(x) \lesssim g_{t, *}(f)(x)$.
Thus, to show \eqref{s4e25}, it suffices to prove that,
for any $f \in \mathcal{S}'(\rn)$
vanishing weakly at infinity,
\begin{equation}\label{s4e26}
\left\|g_{t, *}(f)\right\|_{X} \lesssim\|g(f)\|_{X}
\end{equation}
holds true for some $t \in(1/r_+,\infty)$ with $r_+$
the same as in \eqref{3.25.x1}.
Now, we show \eqref{s4e26}. To this end, assume that
$\phi \in \mathcal{S}\left(\mathbb{R}^{n}\right)$ is a radial function
the same as in Lemma \ref{s4l1}. Obviously, $t \in\left(1/r_+, \infty\right)$
implies that there exists a $\tz_0 \in$ $\left(0, r_+\right)$ such that
$t \in(1 / \tz_0, \infty)$. Fix an $N_{0} \in(1 /\tz_0, \infty)$.
By this, Lemma \ref{s4l4}, and the Minkowski inequality, we find that,
for any $x \in \mathbb{R}^{n}$,
\begin{align*}
g_{t, *}(f)(x)
&=\left\{\sum_{k \in \mathbb{Z}}
\left[\left(\phi_{k}^{*} f\right)_{t}(x)\right]^{2}\right\}^{\frac{1}{2}} \\
& \lesssim\left[\sum_{k \in \mathbb{Z}}
\left\{\sum_{j \in \mathbb{Z}_{+}} b^{-j N_{0} r_+} b^{j+k} \int_{\mathbb{R}^{n}}
\frac{|(\phi_{-(j+k)} * f)(y)|^{r_+}}{[1+b^{k} \rho(x-y)]^{t r_+}}  \,dy\right\}
^{\frac{2}{r_+}}\right]^{\frac{1}{2}} \\
& \leq\left\{\sum_{j \in \mathbb{Z}_{+}} b^{-j\left(N_{0} r_+-1\right)}
\times\left[\sum_{k \in \mathbb{Z}} b^{\frac{2k}{r_+}}
\left\{\int_{\mathbb{R}^{n}}
\frac{|(\phi_{-(j+k)} * f)(y)|^{r_+}}{[1+b^{k} \rho(x-y)]^{t r_+}}  \,dy\right\}
^{\frac{2}{r_+}}\right]^{\frac{r_+}{2}}\right\}^{\frac{1}{r_+}},
\end{align*}
which further implies that
\begin{align*}
\left\|g_{t, *}(f)\right\|_{X}^{r_+ \tz_0}
&\lesssim\left\|\sum_{j \in \mathbb{Z}_{+}}
b^{-j\left(N_{0} r_+-1\right)}\left[\sum_{k \in \mathbb{Z}}
b^{\frac{2k}{r_+}}\left\{\int_{\mathbb{R}^{n}} \frac{|(\phi_{-(j+k)} * f)(y)|
^{r_+}}{[1+b^{k} \rho(\cdot-y)]^{t r_+}}  \,dy\right\}^{\frac{2}{r_+}}\right]
^{\frac{r_+}{2}}\right\|_{X^{\frac{1}{r_+}}}^{\tz_0}\\
&\leq \sum_{j \in \mathbb{Z}_{+}}
b^{-j\left(N_{0} r_+-1\right) \tz_0}
\left\Vert\left[\sum_{k \in \mathbb{Z}}
b^{\frac{2k}{r_+}}\left\{\int_{\mathbb{R}^{n}} \frac{|(\phi_{-(j+k)} * f)(y)|^{r_+}}
{[1+b^{k} \rho(\cdot-y)]^{t r_+}}  \,dy\right\}
^{\frac{2}{r_+}}\right]^{\frac{r_+}{2}}\right\|
_{X^{\frac{1}{r_+}}}^{\tz_0} \\
&\leq \sum_{j \in \mathbb{Z}_{+}}
b^{-j\left(N_{0} r_+-1\right) \tz_0}
\left\|\left\{\sum_{k \in \mathbb{Z}}b^{\frac{2k}{r_+}}
\right.\right.\\
&\quad\times\left[\left(\int_{\left\{y \in \mathbb{R}^{n}:\,\rho(\cdot-y)<b^{-k}
\right\}}+\sum_{i \in \mathbb{\zz_+}} b^{-i t r_+}
\int_{\left\{y \in \mathbb{R}^{n}:\,b^{i-k-1}<\rho(\cdot-y)<b^{i-k}\right\}}
\right)\r.\\
&\quad\times\left.\left.\left|\left(\phi_{-(j+k)}
\ast f\right)(y)\right|^{r_+}\,dy\Bigg]^{\frac{2}{r_+}}\right\}^{\frac{r_+}{2}}
\right\|_{X^{\frac{1}{r_+}}}^{\tz_0}\\
&\leq\sum_{j \in \mathbb{Z}_{+}} b^{-j\left(N_{0} r_+-1\right) \tz_0}
\left\|\sum_{k \in \mathbb{Z}} b^{\frac{2k}{r_+}}
\left\{\sum_{i \in \mathbb{N}} b^{-i t r_+}\r.\r.\\
&\quad\times\lf.\lf.\left[\int_{\left\{y \in \mathbb{R}^{n}:\,
\rho(\cdot-y)<b^{-k}\right\}}\left|\left(\phi_{-(j+k)} \ast f\right)(y)
\right|^{r_+}
\,dy\right]^{\frac{2}{r_+}}\right\}^{\frac{r_+}{2}}\right\|^{\tz_0}
_{X^{\frac{1}{r_+}}}.
\end{align*}
Then, from the Minkowski inequality again and Assumption \ref{Assum-1},
we further infer that
\begin{align*}
\left\|g_{t, *}(f)\right\|_{X}^{r_+\tz_0}
&\lesssim \sum_{j \in \mathbb{Z}_{+}} b^{-j\left(N_{0} r_+-1\right) \tz_0}
\left\|\sum_{i \in \mathbb{N}}
b^{-i t r_+}\left\{\sum_{k \in \mathbb{Z}} b^{k}\r.\r.\\
&\qquad\times\lf.\lf.
\left[\int_{\left\{y \in \mathbb{R}^{n}:\,\rho(\cdot-y)<b^{-k}\right\}}
\left|\left(\phi_{-(j+k)} * f\right)(y)\right|^{r_+}
\,dy\right]^{\frac{2}{r_+}}\right\}
^{\frac{r_+}{2}}\right\|^{\tz_0}_{X^{\frac{1}{r_+}}}\\
&\leq \sum_{j \in \mathbb{Z}_{+}}
b^{-j\left(N_{0} r_+-1\right) \tz_0}\\
&\quad\times\left\|\sum_{i \in \mathbb{N}} b^{(1-t r_+) i}
\left\{\sum_{k \in \mathbb{Z}}
\left[\cm\left(\left|\phi_{-(j+k)} * f\right|^{r_+}\right)\right]
^{\frac{2}{r_+}}\right\}^{\frac{r_+}{2}}\right\|_{X^{\frac{1}{r_+}}}^{\tz_0}\\
&\lesssim\sum_{j \in \mathbb{Z}_{+}} b^{-j\left(N_{0} r_+-1\right) \tz_0}
\sum_{i \in \mathbb{N}} b^{(1-t r_+) i\tz_0}\\
&\quad\times\left\|\left\{\sum_{k \in \mathbb{Z}}
\left[\left|\phi_{-(j+k)} * f\right|^{r_+}\right]^{\frac{2}{r_+}}\right\}
^{\frac{r_+}{2}}
\right\|_{X^{\frac{1}{r_+}}}^{\tz_0}\\
&\sim\left\|g(f)\right\|_{X}^{r_+\tz_0}.
\end{align*}
This further implies that \eqref{s4e26} holds true
and hence finishes the proof of Theorem \ref{s4t1'}.
\end{proof}

\begin{remark}\label{3.24.x3}
\begin{enumerate}
\item [{\rm(i)}] If $A:=2\,I_{n\times n}$, then Theorems \ref{s4t1},
\ref{s4t1'}, and \ref{s4t1''} were obtained in
\cite[Theorems 4.9, 4.11, and 4.13]{cwyz20}
(see also \cite[Theorem 3.21]{shyy17} and \cite[Theorem 2.10]{wyy}).
\item [{\rm(ii)}] As was mentioned in Remark \ref{3.24.x1}(ii),
although $(\rn,\rho,dx)$ is a space of homogeneous type,
Theorems \ref{s4t1}, \ref{s4t1'}, and \ref{s4t1''} can not
be deduced from \cite[Theorems 4.11, 5.1, and 5.3]{yhyy2} and,
actually, they can not cover each other.
\end{enumerate}
\end{remark}

\section{Carleson Measure Characterization of
$\mathcal{L}_{X,1,d,\tz_0}^{A}({\rn})$\label{s5}}
In this section, applying the results obtained in previous sections,
we establish the Carleson measure characterization
of the anisotropic ball Campanato-type function space
$\mathcal{L}_{X,1,d,\tz_0}^{A}({\rn})$.
To this end, we first introduce the following
\emph{anisotropic $X$--Carleson measure}.
\begin{definition}
Let $A$ be a dilation and $X$ a ball quasi-Banach function space.
A Borel measure $d\mu$
on $\rr^{n}\times\zz$ is called an \emph{anisotropic $X$--Carleson measure} if
\begin{align*}
\lf\|d\mu\r\|_{X}^{A}
:=&\,\sup
\lf\|\lf\{\sum_{i=1}^m
\lf[\frac{{\lambda}_i}{\|{\one}_{B^{(i)}}\|_X}\r]^s
{\one}_{B^{(i)}}\r\}^{\frac{1}{s}}\r\|_{X}^{-1}\\
&\quad\times\sum_{j=1}^m\lf\{\frac{{\lambda}_j|B^{(j)}|^{\frac{1}{2}}}
{\|{\one}_{B^{(j)}}
\|_{X}}\lf[\int_{\widehat{B^{(j)}}}\,|d\mu(x,k)|\r]^{\frac{1}{2}}\r\}\\
<&\,\infty,
\end{align*}
where $s\in(0,\infty)$, the supremum
is taken over all $m\in\nn$, $\{B^{(j)}\}_{j=1}^m\subset \CB$, and
$\{\lambda_j\}_{j=1}^m\subset[0,\infty)$ with $\sum_{j=1}^m\lambda_j\neq0$,
and, for any $j\in\{1,\ldots,m\}$,
$\widehat{B^{(j)}}$ denotes the \emph{tent} over $B^{(j)}$,
that is,
\begin{align}\label{hatB}
\widehat{B^{(j)}}:=\lf\{(y,k)\in\rr^{n}\times\zz:\ y+B_k\subset B^{(j)}\r\}.
\end{align}
\end{definition}

For the anisotropic $X$--Carleson measure,
we have the following equivalent characterization.

\begin{proposition}\label{s5p1}
Let $A$ be a dilation, $X$ a ball quasi-Banach function space,
$d\mu$ a Borel measure on $\rr^{n}\times\zz$, and
\begin{align*}
\widetilde{\lf\|d\mu\r\|}_{X}^{A}&:=\sup
\lf\|\lf\{\sum_{i\in\nn}
\lf[\frac{{\lambda}_i}{\|{\one}_{B^{(i)}}\|_X}\r]^s
{\one}_{B^{(i)}}\r\}^{\frac{1}{s}}\r\|_{X}^{-1}\\
&\quad\times\sum_{j\in\nn}\lf\{\frac{{\lambda}_j|B^{(j)}|^{\frac{1}{2}}}
{\|{\one}_{B^{(j)}}
\|_{X}}\lf[\int_{\widehat{B^{(j)}}}\,|d\mu(x,k)|\r]^{\frac{1}{2}}\r\},
\end{align*}
where $s\in(0,\infty)$ and the supremum is taken over all $\{B^{(j)}\}_{j\in\nn}
\subset \CB$ and
$\{\lambda_j\}_{j\in\nn}\subset[0,\infty)$ satisfying
\begin{align*}
\lf\|\lf\{\sum_{i\in\nn}
\lf[\frac{{\lambda}_i}{\|{\one}_{B^{(i)}}\|_X}\r]^s
{\one}_{B^{(i)}}\r\}^{\frac1s}\r\|_{X}\in(0,\infty).
\end{align*}
Then $\widetilde{\lf\|d\mu\r\|}_{X}^{A}=\lf\|d\mu\r\|_{X}^{A}$.
\end{proposition}
\begin{proof}
Let $d\mu$ be a Borel measure on $\rr^{n}\times\zz$.
Obviously, $\lf\|d\mu\r\|_{X}^{A}\le \widetilde{\lf\|d\mu\r\|}_{X}^{A}$.
We next show
\begin{align}\label{s5e1}
\widetilde{\lf\|d\mu\r\|}_{X}^{A}\le \lf\|d\mu\r\|_{X}^{A}.
\end{align}
Indeed, for any $\{B^{(j)}\}_{j\in\nn}\subset \CB$ and $\{\lambda_j\}_{j\in\nn}
\subset[0,\infty)$ as in the present proposition,
by Definition \ref{BQBFS}(iii), we find that
\begin{align*}
&\lim_{m\to\fz}
\lf\|\lf\{\sum_{i=1}^m
\lf[\frac{{\lambda}_i}{\|{\one}_{B^{(i)}}\|_X}\r]^s
{\one}_{B^{(i)}}\r\}^{\frac{1}{s}}\r\|_{X}^{-1}
\sum_{j=1}^m\frac{{\lambda}_j|B^{(j)}|^{\frac{1}{2}}}
{\|{\one}_{B^{(j)}}
\|_{X}}\lf[\int_{\widehat{B^{(j)}}}\,|d\mu(x,k)|\r]^{\frac{1}{2}}\\
&\quad=\lf\|\lf\{\sum_{i\in\nn}
\lf[\frac{{\lambda}_i}{\|{\one}_{B^{(i)}}\|_X}\r]^s
{\one}_{B^{(i)}}\r\}^{\frac{1}{s}}\r\|_{X}^{-1}
\sum_{j\in\nn}\frac{{\lambda}_j|B^{(j)}|^{\frac{1}{2}}}{\|{\one}_{B^{(j)}}
\|_{X}}\lf[\int_{\widehat{B^{(j)}}}\,|d\mu(x,k)|\r]^{\frac{1}{2}}.
\end{align*}
Therefore, for any given $\varepsilon\in(0,\fz)$,
there exists an $m_0\in\nn$ such that
$\sum_{j=1}^{m_0}\lambda_j\neq0$
and
\begin{align*}
&\lf\|\lf\{\sum_{i\in\nn}
\lf[\frac{{\lambda}_i}{\|{\one}_{B^{(i)}}\|_X}\r]^s
{\one}_{B^{(i)}}\r\}^{\frac{1}{s}}\r\|_{X}^{-1}
\sum_{j\in\nn}\frac{{\lambda}_j|B^{(j)}|^{\frac{1}{2}}}{\|{\one}_{B^{(j)}}
\|_{X}}\lf[\int_{\widehat{B^{(j)}}}\,|d\mu(x,k)|\r]^{\frac{1}{2}}\\
&\quad <\lf\|\lf\{\sum_{i=1}^{m_0}
\lf[\frac{{\lambda}_i}{\|{\one}_{B^{(i)}}\|_X}\r]^s
{\one}_{B^{(i)}}\r\}^{\frac{1}{s}}\r\|_{X}^{-1}
\sum_{j=1}^{m_0}\frac{{\lambda}_j|B^{(j)}|^{\frac{1}{2}}}
{\|{\one}_{B^{(j)}}
\|_{X}}\lf[\int_{\widehat{B^{(j)}}}\,|d\mu(x,k)|\r]^{\frac{1}{2}}+\varepsilon\\
&\quad\le\lf\|d\mu\r\|_{X}^{A}+\varepsilon.
\end{align*}
Combining this, the arbitrariness of both
$\{B^{(j)}\}_{j\in\nn}\subset \CB$ and
$\{\lambda_j\}_{j\in\nn}\subset[0,\infty)$ as in the present  proposition,
and $\varepsilon\in(0,\fz)$,
we further obtain \eqref{s5e1} and hence complete the proof of
Proposition \ref{s5p1}.
\end{proof}
In what follows, for any given $k\in\zz$, define
\begin{align*}
\delta_k(j):=
\begin{cases}
1
&\mbox{when}\ j=k,
\\0&\mbox{when}\ j\neq k.
\end{cases}
\end{align*}

Next, we state the main theorem of this section as follows.
\begin{theorem}\label{s5t1}
Let $A$, $X$, $d$, and $\tz_0$ be the same as in Definition \ref{deffin},
$p_0\in(\tz_0,2)$, and $\phi\in\cs(\rn)$ be
a radial real-valued function satisfying \eqref{s4l1e1} and \eqref{s4l1e2}.
\begin{enumerate}
\item[{\rm (i)}]
If $h\in \mathcal{L}_{X,1,d,\tz_0}^{A}({\rn})$, then,
for any $(x,k)\in\rn\times\zz$,
$d\mu(x,k):=\sum_{\ell\in\zz}|\phi_\ell\ast h(x)|^2dx\,\delta_\ell(k)$ is an
$X$--Carleson measure on $\rn\times\zz$;
moreover, there exists a positive constant $C$, independent of $h$, such that
$$\|d\mu\|_X^{A}\le C\|h\|_{\mathcal{L}_{X,1,d,\tz_0}^{A}({\rn})}.$$

\item[{\rm (ii)}]
If $h\in L^2_{\rm loc}(\rn)$ and, for any $(x,k)\in\rn\times\zz$,
$d\mu(x,k):=\sum_{\ell\in\zz}|\phi_\ell\ast h(x)|^2dx\,\delta_\ell(k)$ is an
$X$--Carleson measure on $\rn\times\zz$,
then $h\in \mathcal{L}_{X,1,d,\tz_0}^{A}({\rn})$ and, moreover,
there exists a positive
constant $C$, independent of $h$, such that
$$\|h\|_{\mathcal{L}_{X,1,d,\tz_0}^{A}({\rn})}\le C\|d\mu\|_X^{A}.$$
\end{enumerate}
\end{theorem}
\begin{remark}
\begin{enumerate}
\item[{\rm (i)}]
Note that, if $X$ is a concave ball quasi-Banach function space, then,
by Proposition \ref{s2p2}, Theorem \ref{s5t1} gives
the Carleson measure characterization of $\mathcal{L}_{X,1,d}^{A}({\rn})$.

\item[{\rm (ii)}]
If $A:=2\,I_{n\times n}$, then Theorem \ref{s5t1}
was obtained in \cite[Theorem 5.3]{zhyy21}.
\end{enumerate}

\end{remark}

To prove Theorem \ref{s5t1}, we need the anisotropic tent space associated
with ball quasi-Banach function space and its atomic decomposition.
We first recall the following concept.

\begin{definition}\label{defcone}
Let $A$ be a dilation and, for any $x\in\rn$, let
$$\Gamma(x):=\{(y,k)\in\rn\times\zz:\ y\in x+B_k\},$$ which is
called the \emph{cone} of aperture $1$
with vertex $x\in\rn$.
\end{definition}

Let $\alpha\in(0,\infty)$. For any measurable function
$F:\ \rn\times\zz\to\cc$ and
$x\in\rn$, define
\begin{equation*}
\mathscr{A}(F)(x):=
\lf[\sum_{\ell\in\zz}b^{-\ell}\int_{\{y\in\rn:\ (y,\ell)\in\,\Gamma(x)\}}
|F(y,\ell)|^2\,
dy\r]^\frac12,
\end{equation*}
where $\Gamma(x)$ is the same as in Definition \ref{defcone}.
A measurable function $F$ on $\rn\times\zz$
is said to belong to the \emph{anisotropic tent space}
$T_2^{A,p}(\rn\times\zz)$, with $p\in(0,\infty)$,
if $$\|F\|_{T_2^{A,p}(\rn\times\zz)}:=\|\mathscr{A}(F)\|_{L^p(\rn)}<\infty.$$
For any given ball quasi-Banach function space $X$,
the \emph{anisotropic} $X$-\emph{tent space} $T_X^{A}(\rn\times\zz)$
is defined to be the set of all the measurable
functions $F$ on $\rn\times\zz$ such that $\mathscr{A}(F)\in X$
and naturally equipped with the quasi-norm
$\|F\|_{T_X^{A}(\rn\times\zz)}:=\|\mathscr{A}(F)\|_{X}$.

We next give the definition of anisotropic $(T_X, p)$-atoms.

\begin{definition}\label{s5d1}
Let $p\in(1,\infty)$, $A$ be a dilation,
and $X$ a ball quasi-Banach function space.
A measurable function $a:\ \rn\times\zz\to\cc$ is called an
\emph{anisotropic $(T_X,p)$-atom} if there
exists a ball $B\subset\CB$ such that
\begin{enumerate}
\item[{\rm(i)}] $\supp a:=\{(x,k)\in\rn\times\zz:\ a(x,k)\neq0\}\subset
\widehat{B}$, where $\widehat{B}$ is the same as in \eqref{hatB}
with $B^{(j)}$ replaced by $B$.
\item[{\rm(ii)}] $\|a\|_{T_2^{A,p}(\rn\times\zz)}\le|B|^{1/p}/\|\one_B\|_{X}$.
\end{enumerate}
Moreover, if $a$ is an anisotropic $(T_X,p)$-atom for any $p\in(1,\infty)$,
then $a$ is called an \emph{anisotropic $(T_X,\infty)$-atom}.
\end{definition}

We have the following atomic decomposition
on the anisotropic $X$-tent space $T_X^{A}(\rn\times\zz)$.
\begin{lemma}\label{s5l1}
Let $A,\,X$, and $\tz_0$ be the same as in Definition \ref{deffin}
and $F:\ \rn\times\zz\to\cc$ a measurable function.
If $F\in T_X^A(\rn\times\zz)$, then there exists
a sequence $\{\lambda_j\}_{j\in\nn}\subset [0,\infty)$,
a sequence $\{B^{(j)}\}_{j\in\nn}\subset\CB$,
and a sequence $\{A_j\}_{j\in\nn}$ of anisotropic $(T_X,\infty)$-atoms
supported,
respectively, in
$\{\widehat{B^{(j)}}\}_{j\in\nn}$ such that,
for almost every $(x,k)\in\rn\times\zz$,
\begin{equation*}
F(x,k)=\sum_{j\in\nn}\lambda_jA_j(x,k),\,|F(x,k)|
=\sum_{j\in\nn}\lambda_j|A_j(x,k)|
\end{equation*}
pointwisely, and
\begin{align}\label{s5e2}
\lf\|\lf\{\sum_{j\in\nn}\lf(\frac{\lambda_j}
{\|\mathbf1_{B^{(j)}}\|_X}\r)^{\tz_0}\mathbf1_{B^{(j)}}\r\}^{\frac{1}{\tz_0}}\r\|_X
\ls \|F\|_{T_X^A(\rn\times\zz)},
\end{align}
where the implicit positive constant is independent of $F$.
\end{lemma}
\begin{proof}
For any $j\in\zz$, let
$$O_j:=\lf\{x\in\rn:\ \mathscr{A}(F)(x)>2^j\r\},$$
$F_j:=(O_j)^\complement$, and,
for any given $\gamma\in(0,1)$,
\begin{align*}
(O_j)_\gamma^\ast:=\lf \{x\in\rn: \cm(\one_{O_j})(x)>1-\gamma\r\}.
\end{align*}
Then, by an argument similar to
that used in the proof of \cite[(1.14)]{fl16},
we find that
\begin{align*}
\supp F\subset\left[\bigcup_{j\in\zz}\widehat{(O_j)_\gamma^\ast}\cup E\right],
\end{align*}
where $E\subset\rn\times\zz$ satisfies that
$$\sum_{\ell\in\zz}\int_{\{y\in\rn:\ (y,\ell)\in E\}}dy=0.$$
Moreover, applying \cite[(1.15)]{fl16}, we conclude that, for any $j\in\zz$,
there exists an integer
$N_j\in\nn\cup\{\infty\}$, $\{x_k^{(j)}\}_{k=1}^{N_j}\subset(O_j)_\gamma^\ast$,
and $\{l_k\}_{k=1}^{N_j}\subset\zz$ such that
$\{x_k^{(j)}+B_{l_k}^{(j)}\}_{k=1}^{N_j}$ has the
finite intersection property and
\begin{align}\label{s5l1e2}
\left(O_{j}\right)_{\gamma}^{*} &=\bigcup_{k=1}^{N_j}\left[x_k^{(j)}
+B_{l_k}^{(j)}\right] \\
&=\left[x_{1}^{(j)}+B_{l_{1}}^{(j)}\right] \cup\left\{\left[x_{2}^{(j)}
+B_{l_{2}}^{(j)}\right] \setminus \left[x_{1}^{(j)}+B_{l_{1}}^{(j)}\right]\right\}
\cup \cdots \noz\\
&\quad\cup\left\{\left[x_{N_{j}}^{(j)}+B_{l_{N_{j}}}^{(j)}\right]
\setminus  \bigcup_{i=1}^{N_{j}-1}\left[x_{i}^{(j)}+B_{l_{i}}^{(j)}\right]\right\}
\notag\\
&=: \bigcup_{k=1}^{N_{j}} B_{j, k}\noz.
\end{align}
Notice that, for any $j \in \mathbb{Z},\{B_{j, k}\}_{k=1}^{N_{j}}$
are mutually disjoint.
Thus, $\widehat{(O_{j})_{\gamma}^{*}}=\bigcup_{k=1}^{N_{j}} \widehat{B_{j, k}}$.
For any $j \in \mathbb{Z}$ and $k \in\{1, \ldots, N_{j}\}$, let
\begin{align}\label{s5l1e4}
C_{j, k}:=\widehat{B_{j, k}} \cap\left[\widehat{(O_{j})_{\gamma}^{*}}
\setminus \widehat{(O_{j+1})_{\gamma}^{*}}\right],\
A_{j, k}:=2^{-j}\left\|\one_{x_k^{(j)}+B_{l_k}^{(j)}}\right\|_{X}^{-1}
F \one_{C_{j, k}},
\end{align}
and $\lambda_{j, k}:=2^{j}||\one_{x_k^{(j)}+B_{l_k}^{(j)}}||_{X}$.
Therefore, from \eqref{s5l1e2}, it follows that
$$
F=\sum_{j \in \mathbb{Z}} \sum_{k=1}^{N_{j}} \lambda_{j, k} A_{j, k}
\  \text { and } \ |F|=\sum_{j \in \mathbb{Z}}
\sum_{k=1}^{N_{j}} \lambda_{j, k}\left|A_{j, k}\right|
$$
almost everywhere on $\mathbb{R}^{n} \times \mathbb{Z}$.
We now show that, for any $j \in \mathbb{Z}$
and $k \in\{1, \ldots, N_{j}\}$, $A_{j, k}$
is an anisotropic $(T_X^{A},\infty)$-atom
supported in $\widehat{x_k^{(j)}+B_{l_k}^{(j)}}$
up to a harmless constant multiple. Obviously,
$$
\supp  A_{j, k} \subset C_{j, k} \subset \widehat{B_{j, k}}
\subset \widehat{x_k^{(j)}+B_{l_k}^{(j)}}.
$$
In addition, let $p \in(1, \infty)$ and
$h \in T_{2}^{A, {p^\prime}}\left(\mathbb{R}^{n} \times \mathbb{Z}\right)$
satisfy $\|h\|_{T_{2}^{A, {p^\prime}}\left(\mathbb{R}^{n}
\times \mathbb{Z}\right)}
\leq 1$. Notice that
$$C_{j, k} \subset\widehat{(O_{j+1})_{\gamma}^{*}}^\complement
=\bigcup_{x \in{\left(O_{j+1}\right)_{\gamma}^{*}}^\complement} \Gamma(x).$$
Applying this, \cite[Lemma 1.3]{fl16}, the H\"older inequality,
and \eqref{s5l1e4},
we find that
\begin{align*}
\left|\left\langle A_{j, k}, h\right\rangle\right|
&=\left|\sum_{\ell \in \mathbb{Z}}
\int_{\rn} A_{j, k}(y, \ell) h(y, \ell) \one_{C_{j, k}}(y,\ell) \,dy\right| \\
&\leq \sum_{\ell \in \mathbb{Z}}
\int_{(y,\ell)\in\bigcup_{x \in{(O_{j+1})_{\gamma}^{*}}^
\complement} \Gamma(x)} \left|A_{j, k}(y, \ell) h(y, \ell)\right|
\,dy \,\delta_i(\ell)\\
&\lesssim \int_{{\left(O_{j+1}\right)}^\complement}
\lf[\sum_{\ell \in \mathbb{Z}} \int_{\left\{y \in \mathbb{R}^{n}:\
(y, \ell )
\in \Gamma(x)\right\}} b^{-\ell}
\left|A_{j, k}(y, \ell) h(y, \ell)\right|  \,dy \r] \,dx \\
&\leq \int_{\left(O_{j+1}\right)^{\complement}}
\mathscr{A}\left(A_{j, k}\right)(x) \mathscr{A}(h)(x)\,dx\\
&\leq\left\{\int_{\left(O_{j+1}\right)^{\complement}}
\left[\mathscr{A}\left(A_{j, k}\right)(x)\right]^{p}  \,dx\right\}^{\frac{1}{p}}
\left\{\int_{\left(O_{j+1}\right)^{\complement}}
[\mathscr{A}(h)(x)]^{p^{\prime}}
\,dx\right\}^{\frac{1}{p^{\prime}}} \\
&\leq 2^{-j}\left\|\one_{x_k^{(j)}+B_{l_k}^{(j)}}\right\|_{X}^{-1}
\left\{\int_{(x_k^{(j)}+B_{l_k}^{(j)}) \cap\left(O_{j+1}\right)^{\complement}}
[\mathscr{A}(F)(x)]^{p}  \,dx\right\}^{\frac{1}{p}} \\
&\quad\times\|h\|_{T_{2}^{A, p^{\prime}}
\left(\mathbb{R}^{n} \times \mathbb{Z}\right)} \\
&\lesssim \frac{|x_k^{(j)}+B_{l_k}^{(j)}|^{\frac{1}{p}}}
{\|\one_{x_k^{(j)}+B_{l_k}^{(j)}}
\|_{X}},
\end{align*}
which, combined with $(T_{2}^{A, p}\left(\mathbb{R}^{n}
\times \mathbb{Z}\right))^{*}=T_{2}^{A, p^{\prime}}
\left(\mathbb{R}^{n} \times \mathbb{Z}\right)$ (see \cite[Theorem 2]{CMS85}),
further implies that
\begin{align*}
\left\|A_{j, k}\right\|_{T_{2}^{A, p}\left(\mathbb{R}^{n}
\times \mathbb{Z}\right)} \lesssim \frac{|x_k^{(j)}+B_{l_k}^{(j)}|
^{\frac{1}{p}}}{\| \one_{x_k^{(j)}+B_{l_k}^{(j)}}\|_{X}}.
\end{align*}
Using this, we find that, for any $j \in \mathbb{Z}$
and $k \in$ $\{1, \ldots, N_{j}\},$ $A_{j, k}$
is an anisotropic $(T_X^{A},p)$-atom up to a harmless
constant multiple for any $p \in(1, \infty)$.
Thus, for any $j \in \mathbb{Z}$ and $k \in\{1, \ldots, N_{j}\}$, $A_{j, k}$
is an anisotropic $(T_X^{A},\infty)$-atom up to a harmless constant multiple.

We next prove \eqref{s5e2}. To achieve this, from \eqref{s5l1e2},
the finite intersection property of $\{x_k^{(j)}+B_{l_k}^{(j)}\}_{k=1}^{N_{j}}$,
the estimate that $\one_{(O_{j})_{\gamma}^{*}} \lesssim[\cm(\one_{O_{j}})]
^{\frac{1}{\tz_0}}$,
and Assumption \ref{Assum-1}, we deduce that
\begin{align*}
&\left\|\left\{\sum_{j \in \mathbb{Z}}
\sum_{k=1}^{N_{j}}\left[\frac{\lambda_{j, k}}
{||\one_{x_k^{(j)}+B_{l_k}^{(j)}}||_X}\right]^{\tz_0}\one_{x_k^{(j)}+B_{l_k}
^{(j)}}
\right\}^{\frac{1}{\tz_0}}\right\|_{X}\\
&\quad=\left\|\left\{\sum_{j \in \mathbb{Z}}
\sum_{k=1}^{N_{j}}\left[2^{j}
\one_{x_k^{(j)}+B_{l_k}^{(j)}}\right]^{\tz_0}\right\}^{\frac{1}{\tz_0}}
\right\|_{X}\\
&\quad\lesssim\left\|\left\{\sum_{j \in \mathbb{Z}}
\left[2^{j} \one_{\left(O_{j}\right)_{\gamma}^{*}}\right]^{\tz_0}\right\}
^{\frac{1}{\tz_0}}\right\|_{X}\lesssim\left\|\sum_{j \in \mathbb{Z}}
\left\{2^{j} \left[\cm\left(\one_{O_{j}}\right)\right]^{\frac{1}{\tz_0}}
\right\}
^{\tz_0}\right\|_{X^{\frac{1}{\tz_0}}}^{\frac{1}{\tz_0}}\\
&\quad\lesssim \left\|
\left\{\sum_{j \in \mathbb{Z}}\left(2^{j} \one_{O_{j}}\right)^{\tz_0}\right\}
^{\frac{1}{\tz_0}}\right\|_{X}
\sim\left\|
\left\{\sum_{j \in \mathbb{Z}}\left(2^{j} \one_{O_{j} \setminus  O_{j+1}}\right)
^{\tz_0}\right\}^{\frac{1}{\tz_0}} \right\|_{X}\\
&\quad\leq \left\|\mathscr{A}(F)\left[\sum_{j \in \mathbb{Z}} \one_{O_{j}
\setminus  O_{j+1}}\right]^{\frac{1}{\tz_0}}\right\|_{X}
=\|\mathscr{A}(F)\|_{X} = \|F\|_{T_X^{A}\left(\mathbb{R}^{n}
\times \mathbb{Z}\right)} .
\end{align*}
This further implies that \eqref{s5e2} holds true and hence
finishes the proof of Lemma \ref{s5l1}.
\end{proof}

We now prove Theorem \ref{s5t1}.
\begin{proof}[Proof of Theorem \ref{s5t1}]
We first show (i). To this end, let $h\in\mathcal{L}_{X,1,d,\tz_0}^{A}({\rn})$
and $\{x_j+B_{l_j}\}_{j=1}^m\subset \CB$ with
$m\in\nn$, $\{x_j\}_{j=1}^m\subset\rn$, and $\{l_j\}_{j=1}^m\subset\zz$.
Then we easily find that, for any $j\in\{1,\ldots,m\}$,
\begin{align}\label{s5e3}
h&=P^d_{x_j+B_{l_j}}h+\lf(h-P^d_{x_j+B_{l_j}}h\r)\one_{x_j+B_{l_j+\tau}}+
\lf(h-P^d_{x_j+B_{l_j}}h\r)\one_{(x_j+B_{l_j+\tau})^{\com}}\\
&=:h_{j}^{(1)}+h_{j}^{(2)}+h_{j}^{(3)},\noz
\end{align}
where $\tau$ is the same as in \eqref{tau}.
For $h_{j}^{(1)}$, by the fact that $\int_{\rn}\phi(x)x^{\az}\,dx=0$
for any $\az\in\zz_+^n$ with $|\alpha|\leq d$, we conclude
that, for any $k\in\zz$,
$\phi_k\ast h_{j}^{(1)}\equiv 0$ and hence
\begin{equation}\label{s5e4}
\sum_{k\in\zz}\int_{\{x\in\rn:\,(x,k)\in\widehat{x_j+B_{l_j}}\}}\lf|\phi_k\ast
h_{j}^{(1)}(x)\r|^2dx=0.
\end{equation}
For $h_{j}^{(2)}$, from the Tonelli theorem and the boundedness
on $L^2(\rn)$ of the anisotropic $g$-function
$$g(h_{j}^{(2)}):=\lf[\sum_{k \in \zz}\lf|\phi_k\ast h_{j}^{(2)}\r|
^2\r]^{\frac{1}{2}}$$
(see, for instance, \cite[Theorem 6.3]{hlyy20}),
we infer that
\begin{align}\label{s5e5}
&\sum_{k\in\zz}\int_{\{x\in\rn:\,(x,k)
\in\widehat{x_j+B_{l_j}}\}}\lf|\phi_k\ast h_{j}^{(2)}(x)\r|^2 \,dx\noz\\
&\quad\le \int_{\rn}\sum_{k \in \mathbb{Z}}\lf|\phi_k\ast h_{j}^{(2)}(x)\r|^2
\,dx
\ls \lf\|h_{j}^{(2)}\r\|_{L^2(\rn)}^2\\
&\quad= \int_{x_j+B_{l_j+\tau}}\lf|h(x)-P^d_{x_j+B_{l_j}}h(x)\r|^2\,dx\noz\\
&\quad\leq \int_{x_j+B_{l_j+\tau}}\lf|h(x)-P^d_{x_j+B_{l_j+\tau}}h(x)\r|^2
\,dx\noz\\
&\qquad+\int_{x_j+B_{l_j+\tau}}\lf|P^d_{x_j+B_{l_j+\tau}}h(x)-P^d_{x_j+B_{l_j}}
h(x)\r|
^2\,dx.\noz
\end{align}
In addition, using Lemma \ref{s3l1}, we obtain, for any
$x\in x_j+B_{l_j+\tau}$,
\begin{align*}
&\lf|P^d_{x_j+B_{l_j+\tau}}h(x)-P^d_{x_j+B_{l_j}}h(x)\r|\\
&\quad=\lf|P^d_{x_j+B_{l_j}+\tau}\lf(h-P^d_{x_j+B_{l_j}}h\r)(x)\r|\\
&\quad\ls \frac1{|x_j+B_{l_j}|}
\int_{x_j+B_{l_j+\tau}}\lf|h(y)-P^d_{x_j+B_{l_j+\tau}}h(y)\r|\,dy.
\end{align*}
Thus, combining this with \eqref{s5e5},
Lemma \ref{s3l2}, and Definition \ref{BQBFS}(ii),
we find that, for any $m\in\nn$, $\{x_j+B_{l_j}\}_{j=1}^m\subset \CB$
with both $\{x_j\}_{j=1}^m\subset\rn$ and $\{l_j\}_{j=1}^m\subset\zz$,
and $\{\lz_j\}_{j=1}^m\subset [0,\fz)$ with $\sum_{j=1}^m\lambda_j\neq0$,
\begin{align*}
&\lf\|\lf\{\sum_{i=1}^m
\lf(\frac{{\lambda}_i}{\|{\one}_{x_i+B_{l_i}}\|_X}\r)^{\tz_0}
{\one}_{x_i+B_{l_i}}\r\}^{\frac{1}{\tz_0}}\r\|_{X}^{-1}\sum_{j=1}^m
\frac{{\lambda}_j|x_j+B_{l_j}|^{\frac{1}{2}}}{\|{\one}_{x_j+B_{l_j}}
\|_{X}}\\
&\qquad\times\lf[\sum_{k\in\zz}
\int_{\{x\in\rn:\,(x,k)\in\widehat{x_j+B_{l_j}}\}}\lf|\phi_k\ast
h_{j}^{(2)}(x)\r|
^2\,dx\r]^{\frac{1}{2}}\\
&\quad\ls J_1
\lf\|\lf\{\sum_{i=1}^m
\lf(\frac{{\lambda}_i}{\|{\one}_{x_i+B_{l_i+\tau}}\|_X}\r)^{\tz_0}
{\one}_{x_i+B_{l_i+\tau}}\r\}^{\frac{1}{\tz_0}}\r\|_{X}^{-1}\\
&\qquad\times \sum_{j=1}^m J_2^{(j)}
\left\{
\lf[\int_{x_j+B_{l_j+\tau}}
\lf|h(x)-P^d_{x_j+B_{l_j+\tau}}h(x)\r|^2\,dx\r]^{\frac{1}{2}}\right.\\
&\qquad+\left.
\frac1{|x_j+B_{l_j}|^{\frac{1}{2}}}\int_{x_j+B_{l_j+\tau}}
\lf|h(x)-P^d_{x_j+B_{l_j+\tau}}h(x)\r|\,dx
\right\}\noz\\
&\quad\ls \lf\|\lf\{\sum_{i=1}^m
\lf(\frac{{\lambda}_i}{\|{\one}_{x_i+B_{l_i+\tau}}\|_X}\r)^{\tz_0}
{\one}_{x_i+B_{l_i+\tau}}\r\}^{\frac{1}{\tz_0}}\r\|_{X}^{-1}\\
&\qquad\times\sum_{j=1}^m
\frac{{\lambda}_j|x_j+B_{l_j+\tau}|^{\frac{1}{2}}}
{\|{\one}_{x_j+B_{l_j+\tau}}\|_{X}}
\lf\{\lf[\int_{x_j+B_{l_j+\tau}}
\lf|h(x)-P^d_{x_j+B_{l_j+\tau}}h(x)\r|^2\,dx\r]^{\frac{1}{2}}\right.\\
&\qquad+\left.\frac1{|x_j+B_{l_j}|^{\frac{1}{2}}}
\int_{x_j+B_{l_j+\tau}}
\lf|h(x)-P^d_{x_j+B_{l_j+\tau}}h(x)\r|\,dx\r\}\\
&\quad\leq\|h\|_{\mathcal{L}_{X,2,d,\tz_0}^{A}({\rn})}+\|h\|_
{\mathcal{L}_{X,1,d,\tz_0}^{A}({\rn})},\noz
\end{align*}
where $$ J_1:=\frac{\|\{\sum_{i=1}^m
(\frac{{\lambda}_i}{\|{\one}_{x_i+B_{l_i+\tau}}\|_X})^{\tz_0}
{\one}_{x_i+B_{l_i+\tau}}\}^{\frac{1}{\tz_0}}\|_{X}}
{\|\{\sum_{i=1}^m
(\frac{{\lambda}_i}{\|{\one}_{x_i+B_{l_i}}\|_X})^{\tz_0}
{\one}_{x_i+B_{l_i}}\}^{\frac{1}{\tz_0}}\|_{X}}$$
and, for any $j\in\{1,\ldots,m\}$,
$$J_2^{(j)}:=
\frac{\|{\one}_{x_j+B_{l_j+\tau}}\|_{X}}{\|{\one}_{x_j+B_{l_j}}\|_{X}}
\frac{{\lambda}_j|x_j+B_{l_j+\tau}|^{\frac{1}{2}}}
{\|{\one}_{x_j+B_{l_j+\tau}}\|_{X}}.$$
This, combined with $p_0\in(\tz_0,2)$ and Corollary \ref{s2c1},
further implies that
\begin{align}\label{s5e6}
&\lf\|\lf\{\sum_{i=1}^m
\lf(\frac{{\lambda}_i}{\|{\one}_{x_i+B_{l_i}}\|_X}\r)^{\tz_0}
{\one}_{x_i+B_{l_i}}\r\}^{\frac{1}{\tz_0}}\r\|_{X}^{-1}\sum_{j=1}^m
\frac{{\lambda}_j|x_j+B_{l_j}|^{\frac{1}{2}}}{\|{\one}_{x_j+B_{l_j}}
\|_{X}}\\
&\qquad\times\lf[\sum_{k\in\zz}\int_{\{x\in\rn:\,(x,k)
\in\widehat{x_j+B_{l_j}}\}}\lf|\phi_k\ast h_{j}^{(2)}(x)\r|^2dx\r]^{\frac{1}{2}}\noz\\
&\quad\ls\|h\|_{\mathcal{L}_{X,1,d,\tz_0}^{A}({\rn})}.\noz
\end{align}
Finally, we deal with $h_{j}^{(3)}$. To do this, letting $s\in(0,\tz_0)$
and $\vaz\in(\frac{\ln b}{\ln(\lambda_-)}
[\frac2s+d\frac{\ln(\lambda_+)}{\ln b}],\fz)$, we have,
for any
$j\in\{1,\ldots,m\}$
and $(x,k)\in
\widehat{x_j+B_{l_j}}$,
\begin{align*}
\lf|\phi_k\ast h_{j}^{(3)}(x)\r|
&\ls \int_{(x_j+B_{l_j+\tau})^{\com}}
\frac{b^{\vaz k\frac{\ln \lambda_-}{\ln b}}}{[b^k+\rho(x-y)]
^{1+\vaz\frac{\ln \lambda_-}{\ln b}}}
\lf|h(y)-P^d_{x_j+B_{l_j}}h(y)\r|\,dy\\
&\sim \int_{(x_j+B_{l_j+\tau})^{\com}}
\frac{b^{\vaz k\frac{\ln \lambda_-}{\ln b}}}{[b^k+\rho(x_j-y)]
^{1+\vaz\frac{\ln \lambda_-}{\ln b}}}
\lf|h(y)-P^d_{x_j+B_{l_j}}h(y)\r|\,dy\noz\\
&\leq \frac{{b^{\vaz k \frac{\ln \lambda_-}{\ln b}}}}
{{b^{\vaz l_j \frac{\ln \lambda_-}{\ln b}}}}\int_{(x_j+B_{l_j+\tau})^{\com}}
\frac{b^{\vaz l_j \frac{\ln \lambda_-}{\ln b}}}
{[\rho(x_j-y)]^{1+\vaz\frac{\ln \lambda_-}{\ln b}}}
\lf|h(y)-P^d_{x_j+B_{l_j}}h(y)\r|\,dy\noz\\
&\ls \frac{{b^{\vaz k \frac{\ln \lambda_-}{\ln b}}}}
{{b^{\vaz l_j \frac{\ln \lambda_-}{\ln b}}}} \int_{(x_j+B_{l_j+\tau})^{\com}}
\frac{b^{\vaz l_j \frac{\ln \lambda_-}{\ln b}}}
{b^{l_j(1+\vaz\frac{\ln \lambda_-}{\ln b})}+[\rho(x_j-y)]
^{1+\vaz\frac{\ln \lambda_-}{\ln b}}}\\
&\quad\times\lf|h(y)-P^d_{x_j+B_{l_j}}h(y)\r|\,dy.\noz
\end{align*}
From this and Theorem \ref{s3t1}, it follows that,
for any $m\in\nn$, $\{x_j+B_{l_j}\}_{j=1}^m\subset \CB$
with both $\{x_j\}_{j=1}^m\subset\rn$ and
$\{l_j\}_{j=1}^m\subset\zz$, and
$\{\lz_j\}_{j=1}^m\subset [0,\fz)$ with $\sum_{j=1}^m\lambda_j\neq0$,
\begin{align*}
&\lf\|\lf\{\sum_{i=1}^m
\lf(\frac{{\lambda}_i}{\|{\one}_{x_i+B_{l_i}}\|_X}\r)^{\tz_0}
{\one}_{x_i+B_{l_i}}\r\}^{\frac{1}{\tz_0}}\r\|_{X}^{-1}\sum_{j=1}^m
\frac{{\lambda}_j|x_j+B_{l_j}|^{\frac{1}{2}}}{\|{\one}_{x_j+B_{l_j}}
\|_{X}}\\
&\qquad\times\lf[\sum_{k\in\zz}
\int_{\{x\in\rn: (x,k)\in\widehat{x_j+B_{l_j}}\}}\lf|\phi_k\ast h_{j}^{(3)}(x)\r|^2
\,dx\r]^{\frac{1}{2}}\\
&\quad\ls \lf\|\lf\{\sum_{i=1}^m
\lf(\frac{{\lambda}_i}{\|{\one}_{x_i+B_{l_i}}\|_X}\r)^{\tz_0}
{\one}_{x_i+B_{l_i}}\r\}^{\frac{1}{\tz_0}}\r\|_{X}^{-1}\sum_{j=1}^m
\frac{{\lambda}_j|x_j+B_{l_j}|}{\|{\one}_{x_j+B_{l_j}}
\|_{X}}\\
&\qquad\times\sum_{k=-\infty}^{l_j}b^{-(l_j-k)\vaz\frac{\ln\lambda_{-}}{\ln b}}
\int_{(x_j+B_{l_j+\tau})^{\com}}
\frac{b^{\vaz l_j \frac{\ln \lambda_-}{\ln b}}|h(x)-P^d_{x_j+B_{l_j}}h(x)|}
{b^{l_j(1+\vaz\frac{\ln \lambda_-}{\ln b}})
+[\rho(x_j-x)]^{1+\vaz\frac{\ln \lambda_-}{\ln b}}}\,dx\noz\\
&\quad\ls\|h\|_{\mathcal{L}_{X,1,d,\tz_0}^{A,\vaz}(\rn)}
\sim\|h\|_{\mathcal{L}_{X,1,d,\tz_0}^{A}({\rn})}.\noz
\end{align*}
Combining this, \eqref{s5e3}, \eqref{s5e4}, and \eqref{s5e6}, we
conclude that
\begin{align*}
&\lf\|\lf\{\sum_{i=1}^m
\lf(\frac{{\lambda}_i}{\|{\one}_{x_i+B_{l_i}}\|_X}\r)^{\tz_0}
{\one}_{x_i+B_{l_i}}\r\}^{\frac{1}{\tz_0}}\r\|_{X}^{-1}\sum_{j=1}^m
\frac{{\lambda}_j|x_j+B_{l_j}|^{\frac{1}{2}}}{\|{\one}_{x_j+B_{l_j}}
\|_{X}}\\
&\qquad\times\lf[\sum_{k\in\zz}
\int_{\{x\in\rn:\,(x,k)\in\widehat{x_j+B_{l_j}}\}}\lf|\phi_k\ast h(x)\r|^2\,dx\r]
^{\frac{1}{2}}\\
&\quad\ls\|h\|_{\mathcal{L}_{X,1,d,\tz_0}^{A}({\rn})},
\end{align*}
which, together with the arbitrariness of $m\in\nn$, $\{x_j+B_{l_j}\}_{j=1}
^m\subset \CB$ with both $\{x_j\}_{j=1}^m\subset\rn$
and $\{l_j\}_{j=1}^m\subset\zz$,
and $\{\lz_j\}_{j=1}^m\subset [0,\fz)$ with $\sum_{j=1}^m\lambda_j\neq0$,
further implies that, for any
$(x,k)\in\rn\times\zz$,
$$d\mu(x,k):=|\phi_k\ast h(x)|^2\,dx$$
is an
$X$--Carleson measure on $\rn\times\zz$.
Moreover, there exists a positive constant $C$, independent of $b$, such that
$\|d\mu\|_X^{A}\ls\|h\|_{\mathcal{L}_{X,1,d,\tz_0}^{A}({\rn})}$.
This finishes the proof
of (i).

We now prove (ii). To this end, let
$f\in H_{X, \rm fin}^{A,\fz,d,\tz_0}({{\rr}^n})$
with the quasi-norm greater than zero.
Then $f\in L^{\fz}(\rn)$ with compact support. From this, the assumption
that
$h\in L^2_{\rm loc}(\rn)$, and \cite[(2.10)]{fl16}, it follows that
\begin{equation}\label{s5e7}
\lf|\int_{\rn}f(x)\overline{h(x)}\,dx\r|
\sim \lf|\sum_{k \in \zz}\int_{\rn}\phi_k\ast f(x)\overline{\phi_k\ast h(x)}
\,dx\r|.
\end{equation}
In addition, by the assumption
that $f\in H_X^{A}(\rn)$ and Theorem \ref{s4t1},
we find that
\begin{align*}
\lf\|\phi_k\ast f\r\|_{T_X^A(\rn\times\zz)}\sim\|f\|_{H_X^{A}(\rn)}<\fz,
\end{align*}
which, combined with Lemma \ref{s5l1},
further implies that there exists a sequence
$\{\lz_j\}_{j\in\nn}\subset [0,\fz)$
and a sequence $\{A_j\}_{j\in\nn}$ of anisotropic $(T_X^{A},\fz)$-atoms supported,
respectively, in
$\{\widehat{x_j+B_{l_j}}\}_{j\in\nn}$ with $\{x_j+B_{l_j}\}_{j\in\nn}\subset \CB$
such that, for almost every $(x,k)\in\rn\times\zz$,
$$\phi_k\ast f(x)=\sum_{j\in\nn}\lz_jA_j(x,k)$$
and
\begin{align*}
0<\lf\|\lf\{\sum_{j\in\nn}\lf(\frac{\lambda_j}
{\|\one_{x_j+B_{l_j}}\|_X}\r)^{\tz_0}
\one_{x_j+B_{l_j}}\r\}^{\frac{1}{\tz_0}}\r\|_X
\ls\|f\|_{H_X^{A}(\rn)}.
\end{align*}
From this, \eqref{s5e7}, the H\"older inequality, the size condition
of $A_j$, and the Tonelli theorem, we infer that, for any
$f\in H_{X,\fin}^{A,\infty,d}({{\rr}^n})$,
\begin{align*}
&\lf|\int_{\rn}f(x)\overline{h(x)}\,dx\r|\\
&\quad\ls \sum_{k\in\zz}\sum_{j\in\nn}\lz_j\int_{\rn}\lf|A_j(x,k)\r|
\lf|\phi_k\ast h(x)\r|\,dx\\
&\quad\leq \sum_{j\in\nn}\lz_j\lf[\sum_{k\in\zz}\int_{\{x\in\rn:\,(x,k)
\in\widehat{x_j+B_{l_j}}\}}\lf|A_j(x,k)\r|^2
\,dx\r]^{\frac{1}{2}}\\
&\qquad\times\lf[\sum_{k\in\zz}\int_{\{x\in\rn:\,(x,k)
\in\widehat{x_j+B_{l_j}}\}}\lf|\phi_k\ast h(x)\r|^2
\,dx\r]^{\frac{1}{2}}\\
&\quad=\sum_{j\in\nn}\lz_j\lf[\sum_{k \in \mathbb{Z}}b^{-k}
\int_{\{x\in\rn:\,(x,k)
\in\Gamma(y)\}}\lf|A_j(x,k)\r|^2
\,dx\int_{\{y\in\rn:\,y\in x+B_k\}}dy\r]^{\frac{1}{2}}\\
&\qquad\times\lf[\sum_{k\in\zz}\int_{\{x\in\rn:\,(x,k)\in\widehat{x_j+B_{l_j}}\}}
\lf|\phi_k\ast h(x)\r|^2 \,dx\r]^{\frac{1}{2}}\\
&\quad= \sum_{j\in\nn}\lz_j\lf\|A_j\r\|_{T^{A,2}_2(\rn\times\zz)}
\lf[\sum_{k\in\zz}\int_{\{x\in\rn:\,(x,k)\in\widehat{x_j+B_{l_j}}\}}\lf|\phi_k
\ast h(x)\r|^2\,dx\r]^{\frac{1}{2}}\\
&\quad\leq \sum_{j\in\nn}\frac{\lz_j|x_j+B_{l_j}|^{\frac{1}{2}}}
{\|\one_{x_j+B_{l_j}}\|_{X}}
\lf[\sum_{k\in\zz}\int_{\{x\in\rn:\,(x,k)\in\widehat{x_j+B_{l_j}}\}}\lf|\phi_k
\ast h(x)\r|^2
\,dx\r]^{\frac{1}{2}}\\
&\quad\ls \|f\|_{H_X^{A}(\rn)}\widetilde{\|d\mu\|}_X^{A},
\end{align*}
which, together with Theorem \ref{s2t1}, Proposition \ref{s5p1},
and Corollary \ref{s2c1}, further implies that
$$\|h\|_{\mathcal{L}_{X,1,d,\tz_0}^{A}({\rn})}\ls\|d\mu\|_X^{A}.$$
This finishes the proof of (ii) and hence Theorem \ref{s5t1}.
\end{proof}
\section{Several Applications \label{s6}}
In this section, we apply Theorems \ref{s2t1}, \ref{s3t1},
\ref{s3t2}, \ref{s4t1}, \ref{s4t1'}, \ref{s4t1''}, and \ref{s5t1}
as well as Corollary \ref{s2c1} to
seven concrete examples of ball quasi-Banach function spaces,
namely
Morrey spaces (see Subsection \ref{s6-appl1} below),
Orlicz-slice spaces (see Subsection \ref{s6-appl2} below),
Lorentz spaces (see Subsection \ref{s6-appl3} below),
variable Lebesgue spaces (see Subsection \ref{s6-appl4} below),
mixed-norm Lebesgue spaces (see Subsection \ref{s6-appl5} below),
weighted Lebesgue spaces (see Subsection \ref{s6-appl6} below), and
Orlicz spaces (see Subsection \ref{s6-appl7} below).

\subsection{Morrey Spaces\label{s6-appl1}}

Recall that the classical Morrey space
$M_{q}^{p}(\mathbb{R}^{n})$ with $0<q\leq p<\infty$,
originally introduced by Morrey \cite{morrey38} in 1938,
plays a fundamental role in harmonic analysis
and partial differential equations. From then on,
various variants of Morrey spaces over different
underlying spaces have been investigated and developed
(see, for instance, \cite{cltl20,st09}).

\begin{definition}
Let $A$ be a dilation and $ 0<q\leq p<\infty $.
The {\it anisotropic Morrey space} $M_{q,A}^p(\rn)$ is defined
to be the set of all the measurable functions $f$ on $\rn$ such that
\begin{equation*}
\|f\|_{M_{q,A}^p(\rn)}:=\sup_{B\in\CB}
\lf[|B|^{\frac{1}{p}-\frac{1}{q}}\|f\|_{L^q(B)}\r]<\infty,
\end{equation*}
where $\CB$ is the same as in \eqref{ball-B}.
\end{definition}

It is easy to show that,
$M_{q,A}^p(\rn)$ is a ball quasi-Banach function space.
From this and \cite[Remark 8.4]{wyy22},
we deduce that $M_{q,A}^{p}(\rn)$ satisfies both
Assumptions \ref{Assum-1} and \ref{Assum-2} with $X:=M_{q,A}^{p}(\rn)$,
$p_{-}\in(0,q]$, $\tz_0\in(0,\underline{p})$, and $p_0\in(p,\infty)$,
where $\underline{p}:=\min\{p_{-},1\}$. In what follows,
we always let $HM_{q,A}^p(\rn)$ be the
\emph{anisotropic Hardy--Morrey space}
which is defined to be the same
as in Definition \ref{HXA} with $X:=M_{q,A}^p(\rn)$.
Then, applying Theorems \ref{s4t1}, \ref{s4t1'}, and \ref{s4t1''},
we obtain the following characterizations of $HM_{q,A}^p(\rn)$,
respectively, in terms of the anisotropic
Lusin area function, the anisotropic Littlewood--Paley $g$-function,
and the anisotropic Littlewood--Paley $g_\lambda^*$-function.

\begin{theorem}\label{Thsm}
Let $A$ be a dilation and $0<q\leq p<\infty$.
Then Theorems \ref{s4t1}, \ref{s4t1'},
and \ref{s4t1''} with $X:=M_{q,A}^p(\rn)$
and $\lambda\in(2/\min\{1,q\},\infty)$ hold true.
\end{theorem}

\begin{remark}\label{3.23.x1}
\begin{enumerate}
\item[{\rm(i)}] We point out that Theorem \ref{Thsm} is completely new.
\item[{\rm(ii)}] However, Theorems \ref{s2t1},
\ref{s3t1}, \ref{s3t2}, and \ref{s5t1} as well as
Corollary \ref{s2c1} can not be applied to the anisotropic
Morrey space $M_{q,A}^p(\rn)$ because $M_{q,A}^p(\rn)$ does not
have an absolutely continuous quasi-norm.
\end{enumerate}
\end{remark}

\subsection{Orlicz-Slice Spaces\label{s6-appl2}}

Recently, Zhang et al. \cite{zyyw} originally introduced the
Orlicz-slice space on $\rn$, which generalizes both the slice
space in \cite{am19} and the Wiener-amalgam space in \cite{dj17}.
They also introduced the Orlicz-slice (local) Hardy spaces
and developed a complete real-variable theory of these spaces
in \cite{zyy21,zyyw}. For more studies about Orlicz-slice spaces,
we refer the reader to \cite{ho21,ho22}.

Recall that a function $\Phi:[0,\infty)\to[0,\infty)$
is called an {\it Orlicz function} if it is non-decreasing,
$\Phi(0)=0$, $\Phi(t)>0$ for any $t\in(0,\infty)$,
and $\lim_{t\to\infty}\Phi(t)=\infty$.
The function $ \Phi $ is said to be of {\it upper}
(resp. {\it lower}) {\it type} $p$ for some $p\in[0,\infty)$
if there exists a positive constant $C$ such that,
for any $s\in[1,\infty)$ (resp. $s\in[0,1]$)
and $t\in[0,\infty)$, $\Phi(st)\leq Cs^p\Phi(t)$.
The \emph{Orlicz space} $L^{\Phi}(\rn)$ is defined to be
the set of all the measurable functions $f$ on $\rn$ such that
\begin{align*}
&\|f\|_{L^{\Phi}(\rn)}:=\inf\lf\{\lz\in(0,\fz)
:\ \int_{\rn}\Phi\lf(\frac{|f(x)|}{\lz}\r)\,dx\le1\r\}<\infty.
\end{align*}

\begin{definition}
Let $A$ be a dilation, $\ell\in\zz$, $q\in(0,\infty)$,
and $\Phi$ be an Orlicz function.
The {\it anisotropic Orlicz--slice space}
$(E_\Phi^q)_{\ell,A}\lf(\rn\r)$ is defined to be the
set of all the measurable functions $f$ on $\rn$ such that
\begin{align*}
\|f\|_{(E_\Phi^q)_{\ell,A}\lf(\rn\r)}
:=\lf\{\int_{\rn}\lf[\frac{\|f\one_{x+B_{\ell}}\|_{L^\Phi(\rn)}}
{\|\one_{x+B_{\ell}}\|_{L^\Phi(\rn)}}\r]^q\,dx\r\}^{\frac{1}{q}}
<\infty,
\end{align*}
where $B_{\ell}$ is the same as in \eqref{B_k}.
\end{definition}
Let $A$ be a dilation, $\ell\in\zz$, $q\in(0,\infty)$,
and $\Phi$ be an Orlicz function with positive lower type
$p_\Phi^-$ and positive upper type $p_\Phi^+$.
Then, by the arguments similar to those used in
the proofs of \cite[Lemmas 2.28 and 4.5]{zyyw},
we find that $(E_\Phi^q)_{\ell,A}\lf(\rn\r)$ is
a ball quasi-Banach function space and
has an absolutely continuous quasi-norm.
From these and \cite[Remark 8.14]{wyy22},
we deduce that $(E_\Phi^q)_{\ell,A}(\rn)$ satisfies both
Assumptions \ref{Assum-1} and \ref{Assum-2} with $X:=(E_\Phi^q)_{\ell,A}(\rn)$,
$p_{-}\in (0,\min\{p_\Phi^-,q\}]$, $\tz_0\in(0,\underline{p})$,
and $p_0\in(\max\{p_\Phi^+,q\},\infty)$, where $\underline{p}:=\min\{p_{-},1\}$.
In what follows, we always let $(HE_\Phi^q)_{\ell,A}(\rn)$
denote the \emph{anisotropic Orlicz-slice Hardy space}
which is defined to be the same as in Definition \ref{HXA} with
$X:=(E_\Phi^q)_{\ell,A}(\rn)$.
Moreover, by Theorems \ref{s2t1}, \ref{s3t1},
\ref{s3t2}, \ref{s4t1}, \ref{s4t1'}, \ref{s4t1''},
and \ref{s5t1} as well as Corollary \ref{s2c1}
with $X$ replaced by $(E_\Phi^q)_{\ell,A}(\rn)$,
we obtain the following conclusion.

\begin{theorem}\label{thorliczslice}
Let $A$ be a dilation, $\ell\in\zz$, $q\in(0,\infty)$,
and $\Phi$ be an Orlicz function with positive lower
type $p_\Phi^-$.
Then
\begin{enumerate}
\item[{\rm(i)}] Theorems \ref{s2t1}, \ref{s3t1},
\ref{s3t2}, and \ref{s5t1} as well as Corollary \ref{s2c1}
with $X:=(E_\Phi^q)_{\ell,A}(\rn)$ hold true;

\item[{\rm(ii)}] Theorems \ref{s4t1}, \ref{s4t1'},
and \ref{s4t1''} with $X:=(E_\Phi^q)_{\ell,A}(\rn)$
and $\lambda\in(\frac{2}{\min\{1,p_\Phi^-,q\}},\infty)$
also hold true.
\end{enumerate}
\end{theorem}

\begin{remark}\label{3.23.x2}
We point out that Theorem \ref{thorliczslice} is completely new.
\end{remark}

\subsection{Lorentz Spaces\label{s6-appl3}}
Let $p\in(0,\fz]$ and $q\in(0,\fz]$.
Recall that the {\it Lorentz space} $L^{p,q}(\rn)$
is defined to be the set of all the measurable functions
$f$ on $\rn$ with the following finite quasi-norm
\begin{align*}
\|f\|_{L^{p,q}(\rn)}:=
\begin{cases}
\lf[\displaystyle{\frac{q}{p}\int_0^{\infty}}
\lf\{t^{\frac{1}{p}}f^*(t)\r\}^q\,\frac{dt}{t}\r]^{\frac{1}{q}}
\ &{\rm if}\ q\in(0,\infty),\\
\displaystyle{\sup_{t\in(0,\infty)}}
\lf[t^{\frac{1}{p}}f^*(t)\r]&{\rm if}\ q=\infty
\end{cases}
\end{align*}
with the usual modification made when $p=\fz$,
where $f^*$ denotes the \emph{non-increasing rearrangement} of $f$,
that is, for any $t\in(0,\infty)$,
\begin{equation*}
f^*(t):=\lf\{\alpha\in(0,\infty):\ d_f(\alpha)\leq t\r\}
\end{equation*}
with $d_f(\alpha):=|\{x\in\rn:\ |f(x)|>\alpha\}|$
for any $\az\in(0,\fz)$.

Then, by \cite[Remarks 2.7(ii), 4.21(ii), and 6.8(iv)]{yhyy},
we conclude that $L^{p,q}(\rn)$ satisfies all the assumptions
of Definition \ref{HXA} with $X:=L^{p,q}(\rn)$,
$p_{-}\in (0,p]$, $\tz_0\in(0,\underline{p})$,
and $p_0\in(p,\infty)$, where $\underline{p}:=\min\{p_{-},1\}$,
and has an absolutely continuous quasi-norm.
In what follows, we always let $H_{A}^{p,q}(\rn)$
be the \emph{anisotropic Hardy--Lorentz space}
which is defined to be the same
as in Definition \ref{HXA} with $X:=L^{p,q}(\rn)$.
By Theorems \ref{s2t1}, \ref{s3t1},
\ref{s3t2}, \ref{s4t1}, \ref{s4t1'}, \ref{s4t1''},
and \ref{s5t1} as well as Corollary \ref{s2c1}
with $X$ replaced by $L^{p,q}(\rn)$,
we obtain the following conclusion.

\begin{theorem}\label{thlorentz}
Let $A$ be a dilation, $p\in(0,\fz)$, and $q\in(0,\fz]$. Then
\begin{enumerate}
\item[{\rm(i)}] Theorems \ref{s2t1}, \ref{s3t1},
\ref{s3t2}, and \ref{s5t1} as well as Corollary \ref{s2c1}
with $X:=L^{p,q}(\rn)$ hold true;

\item[{\rm(ii)}] Theorems \ref{s4t1}, \ref{s4t1'},
and \ref{s4t1''} with $X:=L^{p,q}(\rn)$ and
$\lambda\in(2/\min\{1,p\},\infty)$ also hold true.
\end{enumerate}
\end{theorem}

\begin{remark}\label{3.23.x3}
Let $p(\cdot)\in C^{\log}(\rn)$ and $q\in(0,\fz)$,
where $C^{\log}(\rn)$ is the same as in Subsection \ref{s6-appl4} below.
We point out that Theorem \ref{thlorentz}(i) is a special
case of \cite[Theorems 1 and 2]{llh2023} with $p(\cdot)\equiv p\in(0,\infty)$
therein and that Theorem \ref{thlorentz}(ii) improves the corresponding
results in \cite[Theorems 2.7, 2.8, and 2.9]{lyy2018}
by widening the range of $p\in(0,1]$ into $p\in(0,\fz)$.
Although the variable Hardy--Lorentz space $L^{p(\cdot),q}(\rn)$
is also a ball quasi-Banach function space,
\cite[Theorems 1 and 2]{llh2023}
can not be deduced from Theorems \ref{s2t1} and \ref{s5t1}.
This is because the boundedness of the powered Hardy--Littlewood
maximal operator on the associate space of
$L^{p(\cdot),q}(\rn)$ is still unknown,
which makes to verify Assumption \ref{Assum-2}
with $X:=L^{p(\cdot),q}(\rn)$ become impossible presently.
\end{remark}

\subsection{Variable Lebesgue Spaces\label{s6-appl4}}
Denote by $\cp(\rn)$ the {\it set of all the measurable functions}
$p(\cdot)$ on $\rn$ satisfying
\begin{align}\label{s6e1}
0<\widetilde{p_-}:=\mathop\mathrm{ess\,inf}_{x\in\rn}p(x)\le
\mathop\mathrm{ess\,sup}_{x\in\rn}p(x)=:\widetilde{p_+}<\fz.
\end{align}
For any $p(\cdot)\in \cp(\rn)$, the \emph{variable Lebesgue space}
$L^{p(\cdot)}(\rn)$ is defined to be the set of all the measurable
functions $f$ on $\rn$ such that
\begin{equation*}
\int_{\rn}|f(x)|^{p(x)}\,dx<\infty,
\end{equation*}
equipped with the \emph{quasi-norm}
$\|f\|_{L^{p(\cdot)}(\rn)}$ defined by setting
\begin{equation*}
\|f\|_{L^{p(\cdot)}(\rn)}:=\inf\lf\{\lambda\in(0,\infty):
\  \int_{\rn}\lf[\frac{|f(x)|}
{\lambda}\r]^{p(x)}\,dx\leq 1\r\}.
\end{equation*}

Denote by $C^{\log}(\rn)$ the set of all the functions $p(\cdot)\in\cp(\rn)$
satisfying the {\it globally log-H\"older continuous condition}, that is,
there exist $C_{\log}(p),C_\infty\in(0,\infty)$ and $p_\infty\in\rr$
such that, for any $x,y\in\rn$,
\begin{equation*}
|p(x)-p(y)|\leq\frac{C_{\log}(p)}{\log {(e+1/|x-y|)}}
\end{equation*}
and
\begin{equation*}
|p(x)-p_\infty|\leq\frac{C_\infty}{\log(e+\rho(x))}.
\end{equation*}

Let $p(\cdot)\in C^{\log}(\rn)$.
Then, by \cite[Remarks 2.7(iv), 4.21(v), and 6.8(vii)]{yhyy},
we conclude that $L^{p(\cdot)}(\rn)$ satisfies all the
assumptions of Definition \ref{HXA} with $X:=L^{p(\cdot)}(\rn)$,
$p_{-}:=\widetilde{p_-}$, $\tz_0\in(0,\widetilde{\unp})$, and
$p_0\in(\widetilde{p_+},\infty]$,
where $\widetilde{p_-}$ and $\widetilde{p_+}$ are the same as in \eqref{s6e1}
and $\widetilde{\unp}:=\min\{1,\widetilde{p_-}\}$,
and has an absolutely continuous quasi-norm.
In what follows,
we always let $H_{A}^{p(\cdot)}(\rn)$ be the
\emph{anisotropic variable Hardy space}
which is defined to be the same
as in Definition \ref{HXA} with $X:=L^{p(\cdot)}(\rn)$.
Moreover, by Theorems \ref{s2t1}, \ref{s3t1},
\ref{s3t2}, \ref{s4t1}, \ref{s4t1'}, \ref{s4t1''},
and \ref{s5t1} as well as Corollary \ref{s2c1}
with $X$ replaced by $L^{p(\cdot)}(\rn)$, we obtain the following conclusion.

\begin{theorem}\label{thvariable}
Let $A$ be a dilation and $p(\cdot)\in C^{\log}(\rn)$. Then
\begin{enumerate}
\item[{\rm(i)}] Theorems \ref{s2t1}, \ref{s3t1},
\ref{s3t2}, and \ref{s5t1} as well as Corollary \ref{s2c1}
with $X:=L^{p(\cdot)}(\rn)$ hold true;

\item[{\rm(ii)}] Theorems \ref{s4t1}, \ref{s4t1'},
and \ref{s4t1''} with $X:=L^{p(\cdot)}(\rn)$ and
$\lambda\in(2/\min\{1,\widetilde p_-\},\infty)$ also hold true,
where $\widetilde{p_-}$ is the same as in \eqref{s6e1}.
\end{enumerate}
\end{theorem}

\begin{remark}\label{3.23.x4}
We point out that Theorem \ref{thvariable}(i)
was also obtained in \cite[Theorems 1, 2, and 3, and Corollary 1]{hw22} and
Theorem \ref{thvariable}(ii) improves the
corresponding results in \cite[Theorems 6.1, 6.2, and 6.3]{lwyy2018}
by widening the range
of $\lambda\in(1+2/\min\{2,\widetilde p_-\},\infty)$
into $\lambda\in(2/\min\{1,\widetilde p_-\},\infty)$.
\end{remark}

\subsection{Mixed-Norm Lebesgue Spaces\label{s6-appl5}}

Let $\vec{p}:=(p_1,\ldots,p_n)\in(0,\fz]^n$.
Recall that the \emph{mixed-norm Lebesgue space $L^{\vec{p}}(\rn)$}
is defined to be the set of all the measurable
functions $f$ on $\rn$ such that
\begin{align*}
&\|f\|_{L^{\vec{p}}(\rn)}\\
&\quad:=\lf\{\int_{\mathbb{R}}\cdots
\lf[\int_{\mathbb{R}}
\lf\{\int_{\mathbb{R}}|f(x_1,\ldots,x_n)|^{p_1}\,dx_1\r\}
^{\frac{p_2}{p_1}}\,dx_2\r]^{\frac{p_3}{p_2}}\cdots\,dx_n\r\}
^{\frac{1}{p_n}}
\end{align*}
is finite with the usual modifications made
when $p_i=\fz$ for some $i\in\{1,\ldots,n\}$.

Let $\vec{p}\in(0,\fz]^n$.
Then, by both \cite[p.\,2047]{zwyy}
and \cite[Lemmas 7.22 and 7.26]{zwyy},
we conclude that $L^{\vec{p}}(\rn)$ satisfies
all the assumptions of Definition \ref{HXA}
with $X:=L^{\vec{p}}(\rn)$, $p_{-}:=\widehat{p_-}$,
$\tz_0\in(0,\widehat{\unp})$, and $p_0\in(\tz_0,\infty)$,
where $\widehat{p_-}:=\min\{p_1,...,p_n\}$
and $\widehat{\unp}:=\min\{1,\widehat{p_-}\}$, and has
an absolutely continuous quasi-norm.
In what follows, we always let $H^{\vec{p}}_{A}(\rn)$
be the
\emph{anisotropic mixed-norm Hardy space}
which is defined to be the same
as in Definition \ref{HXA} with $X:=L^{\vec{p}}(\rn)$.
Moreover, by Theorems \ref{s2t1}, \ref{s3t1},
\ref{s3t2}, \ref{s4t1}, \ref{s4t1'}, \ref{s4t1''},
and \ref{s5t1} as well as Corollary \ref{s2c1}
with $X$ replaced by $L^{\vec{p}}(\rn)$,
we obtain the following conclusion.

\begin{theorem}\label{thmix}
Let $A$ be a dilation and $\vec{p}\in(0,\fz)^n$. Then
\begin{enumerate}
\item[{\rm(i)}] Theorems \ref{s2t1}, \ref{s3t1},
\ref{s3t2}, and \ref{s5t1} as well as Corollary \ref{s2c1}
with $X:=L^{\vec{p}}(\rn)$ hold true;

\item[{\rm(ii)}] Theorems \ref{s4t1}, \ref{s4t1'},
and \ref{s4t1''} with
$X:=L^{\vec{p}}(\rn)$ and
$\lambda\in(2/\min\{1,\widehat{p_-}\},\infty)$ also hold true,
where $\widehat{p_-}:=\min\{p_1,...,p_n\}$.
\end{enumerate}
\end{theorem}

\begin{remark}\label{3.23.x5}
\begin{enumerate}
\item[{\rm(i)}] We point out that Theorem \ref{thmix}(i) was also obtained in
\cite[Theorems 3.4, 4.1, and 5.3, and Corollary 3.9]{hyy21}
and Theorem \ref{thmix}(ii) improves the corresponding results
in \cite[Theorems 6.2, 6.3, and 6.4]{hlyy20} by widening the range of
$\lambda\in(1+2/\min\{2,\widehat{p_-}\},\infty)$
into $\lambda\in(2/\min\{1,\widehat{p_-}\},\infty)$.
\item[{\rm(ii)}] Let $\vec{a}:=(a_1,\ldots,a_n)\in[1,\infty]^n$.
Then Theorem \ref{thmix}(i) with
$$A:=
\begin{pmatrix}
 2^{a_1}&0&\cdots&0\\
 0&2^{a_2}&\cdots&0\\
 \vdots&\vdots& &\vdots\\
 0&0&\cdots&2^{a_n}
\end{pmatrix}$$
gives the dual space of the anisotropic mixed-norm Hardy space
$H_{\vec{a}}^{\vec{p}}(\rn)$ which was introduced in
\cite[Definition 3.3]{cgn17} and completely answers
the open problem on the dual space of
$H_{\vec{a}}^{\vec{p}}(\rn)$ proposed in \cite{cgn17}.
\end{enumerate}
\end{remark}

\subsection{Weighted Lebesgue Spaces\label{s6-appl6}}

Let $p\in(0,\fz]$ and $ w\in\ca_\infty(A)$.
From \cite[Remarks 2.7(iii), 4.21(iii), and 6.8(v)]{yhyy},
we deduce that $L_w^p(\rn)$ satisfies all the
assumptions of Definition \ref{HXA} with $X:=L_w^p(\rn)$,
$p_{-}\in(0,p/q_w]$, $\tz_0\in(0,\min\{1,p_{-}\})$, and $p\in(\tz_0,\infty)$,
where $q_w$ is the same as in \eqref{qw},
and has an absolutely continuous quasi-norm.
In what follows, we always let $H_w^p(\rn)$
be the
\emph{anisotropic weighted Hardy space} which
is defined to be the same as in Definition \ref{HXA} with $X:=L_w^p(\rn)$.
By Theorems \ref{s2t1}, \ref{s3t1},
\ref{s3t2}, \ref{s4t1}, \ref{s4t1'}, \ref{s4t1''},
and \ref{s5t1} as well as Corollary \ref{s2c1}
with $X$ replaced by $L_w^p(\rn)$,
we obtain the following conclusion.

\begin{theorem}\label{thwei}
Let $A$ be a dilation, $p\in(0,\fz)$, and
$w\in\ca_\infty(A)$.
Then
\begin{enumerate}
\item[{\rm(i)}] Theorems \ref{s2t1}, \ref{s3t1},
\ref{s3t2}, and \ref{s5t1} as well as Corollary \ref{s2c1}
with $X:=L_w^p(\rn)$ hold true;

\item[{\rm(ii)}] Theorems \ref{s4t1}, \ref{s4t1'},
and \ref{s4t1''} with $X:=L_w^p(\rn)$ and
$\lambda\in(2/\min\{1,{q_w}/{p}\},\infty)$ also hold true,
where $q_w$ is the same as in \eqref{qw}.
\end{enumerate}
\end{theorem}

\begin{remark}\label{3.27.x1}
We point out that Theorem \ref{thwei}(i) is completely new and
Theorem \ref{thwei}(ii) improves the corresponding results in
\cite[Theorems 2.14, 3.1, and 3.9]{lfy15} by
widening the range of $p\in(0,1]$ into $p\in(0,\fz)$.
\end{remark}

\subsection{Orlicz Spaces\label{s6-appl7}}	
Let $\Phi$ be an Orlicz function with positive lower type
$p_{\Phi}^-$ and positive upper type $p_{\Phi}^+$.
From \cite[Remarks 2.7(iii), 4.21(iv), and 6.8(vi)]{yhyy},
we deduce that $L^{\Phi}(\rn)$ satisfies all the assumptions
of Definition \ref{HXA} with $X:=L^{\Phi}(\rn)$,
$p_{-}\in(0,p_\Phi^-]$, $\tz_0\in(0,\min\{p_\Phi^-,1\})$,
and $p_0\in(\max\{p_\Phi^+,1\},\infty)$,
and has an absolutely continuous quasi-norm.
In what follows, we always let $H^{\Phi}_{A}(\rn)$
be the
\emph{anisotropic Orlicz--Hardy space} which is
defined to be the same as in Definition \ref{HXA} with $X:=L^{\Phi}(\rn)$.
Moreover, by Theorems \ref{s2t1}, \ref{s3t1},
\ref{s3t2}, \ref{s4t1}, \ref{s4t1'}, \ref{s4t1''},
and \ref{s5t1} as well as Corollary \ref{s2c1}
with $X$ replaced by $L^{\Phi}(\rn)$,
we obtain the following conclusion.

\begin{theorem}\label{thorlicz}
Let $A$ be a dilation and
$\Phi$ an Orlicz function with lower type $p_{\Phi}^-\in(0,\fz)$. Then
\begin{enumerate}
\item[{\rm(i)}] Theorems \ref{s2t1}, \ref{s3t1},
\ref{s3t2}, and \ref{s5t1} as well as Corollary \ref{s2c1} with
$X:=L^{\Phi}(\rn)$ hold true;

\item[{\rm(ii)}] Theorems \ref{s4t1}, \ref{s4t1'},
and \ref{s4t1''}
with $X:=L^{\Phi}(\rn)$ and
$\lambda\in(2/\min\{1,p_{\Phi}^-\},\infty)$ also hold true.
\end{enumerate}
\end{theorem}

\begin{remark}\label{3.27.x2}
We point out that Theorem \ref{thorlicz}(i) is completely new and
Theorem \ref{thorlicz}(ii) improves the corresponding results in
\cite[Theorems 2.14, 3.1, and 3.9]{lfy15}
by widening the range of $p_{\Phi}^-\in(0,1]$ into $p_{\Phi}^-\in(0,\fz)$.
\end{remark}

\medskip

\noindent\textbf{Author Contributions}\quad All authors developed and discussed the results and contributed to the final manuscript.

\medskip

\noindent\textbf{Data Availibility Statement}\quad Data sharing 
is not applicable to this article as no data sets were generated or
analysed.

\medskip

\noindent\textbf{Declarations}

\medskip

\noindent\textbf{Conflict of interest}\quad All authors state no conflict of interest.

\medskip

\noindent\textbf{Informed consent}\quad Informed consent has been obtained from all individuals included in this research work.


\bigskip

\noindent Chaoan Li and Dachun Yang (Corresponding author)

\medskip

\noindent Laboratory of Mathematics and Complex Systems
(Ministry of Education of China),
School of Mathematical Sciences, Beijing Normal University,
Beijing 100875, The People's Republic of China

\smallskip

\noindent {\it E-mails}:
\texttt{cali@mail.bnu.edu.cn} (C. Li)

\noindent\phantom{{\it E-mails:}}
\texttt{dcyang@bnu.edu.cn} (D. Yang)

\bigskip

\noindent Xianjie Yan

\medskip

\noindent Institute of Contemporary Mathematics,
School of Mathematics and Statistics, Henan University,
Kaifeng 475004, The People's Republic of China

\smallskip

\noindent{{\it E-mail:}}
\texttt{xianjieyan@henu.edu.cn}



\begin{thebibliography}{99}

\bibitem{abr2017}
V. Almeida, J. J. Betancor
and L. Rodr\'iguez-Mesa,
Anisotropic Hardy--Lorentz spaces with variable exponents,
Canad. J. Math. 69 (2017), 1219--1273.

\vspace{-0.3cm}

\bibitem{am19}
P. Auscher and M. Mourgoglou,
Representation and uniqueness for boundary value
elliptic problems via first order systems,
Rev. Mat. Iberoam. 35 (2019), 241--315.

\vspace{-.3cm}

\bibitem{bs88}
C. Bennett and R. Sharpley,
Interpolation of Operators,
Pure Appl. Math. 129, Academic Press, Boston, MA, 1988.

\vspace{-0.3cm}

\bibitem{Bownik}
M. Bownik,
Anisotropic Hardy spaces and wavelets,
Mem. Amer. Math. Soc. 164 (781) (2003), vi+122 pp.

\vspace{-0.3cm}

\bibitem{Bownik2}
M. Bownik,
Duality and interpolation of anisotropic Triebel--Lizorkin spaces,
Math. Z. 259 (2008), 131--169.

\vspace{-0.3cm}

\bibitem{bh06}
M. Bownik and K.-P. Ho,
Atomic and molecular decompositions of
anisotropic Triebel--Lizorkin spaces,
Trans. Amer. Math. Soc. 358 (2006), 1469--1510.

\vspace{-0.3cm}

\bibitem{blyz08weight}
M. Bownik, B. Li, D. Yang and Y. Zhou,
Weighted anisotropic Hardy spaces and their
applications in boundedness of sublinear operators,
Indiana Univ. Math. J. 57 (2008), 3065--3100.

\vspace{-0.3cm}

\bibitem{blyz10weight}
M. Bownik, B. Li, D. Yang and Y. Zhou,
Weighted anisotropic product Hardy spaces
and boundedness of sublinear operators,
Math. Nachr. 283 (2010), 392--442.

\vspace{-0.3cm}

\bibitem{bdl18}
T. A. Bui, X.-T. Duong and F. K. Ly,
Maximal function characterizations for new local
Hardy-type spaces on spaces of homogeneous type,
Trans. Amer. Math. Soc. 370 (2018), 7229--7292.

\vspace{-0.3cm}

\bibitem{bdl20}
T. A. Bui, X.-T. Duong and F. K. Ly,
Maximal function characterizations for Hardy spaces on spaces
of homogeneous type with finite measure and applications,
J. Funct. Anal. 278 (2020), 108423, 55 pp.

\vspace{-0.3cm}

\bibitem{bl11}
T. A. Bui and J. Li,
Orlicz--Hardy spaces associated to operators satisfying bounded
$H_\infty$ functional calculus and Davies--Gaffney estimates,
J. Math. Anal. Appl. 373 (2011), 485--501.

\vspace{-0.3cm}

\bibitem{calderon77}
A.-P. Calder\'on,
An atomic decomposition of distributions in parabolic $H^p$ spaces,
Adv. Math. 25 (1977), 216--225.

\vspace{-0.3cm}

\bibitem{ct75}
A.-P. Calder\'on and A. Torchinsky,
Parabolic maximal functions associated with a distribution,
Adv. Math. 16 (1975), 1--64.

\vspace{-0.3cm}

\bibitem{ct77}
A.-P. Calder\'on and A. Torchinsky,
Parabolic maximal functions associated with a distribution. II,
Adv. Math. 24 (1977), 101--171.

\vspace{-0.3cm}

\bibitem{c64}
S. Campanato,
Proprieti una famiglia di spazi funzionali,
Ann. Scuola Norm. Sup. Pisa (3) 18 (1964), 137--160.

\vspace{-0.3cm}

\bibitem{cwyz20}
D.-C. Chang, S. Wang, D. Yang and Y. Zhang,
Littlewood--Paley characterizations of Hardy-type spaces
associated with ball quasi-Banach function spaces,
Complex Anal. Oper. Theory 14 (2020), Paper No. 40, 33 pp.

\vspace{-0.3cm}

\bibitem{cltl20}
J. Chou, X. Li, Y. Tong and H. Lin,
Generalized weighted Morrey spaces on RD-spaces,
Rocky Mountain J. Math. 50 (2020), 1277--1293.

\vspace{-0.3cm}

\bibitem{christ90}
M. Christ,
A $T(b)$ theorem with remarks on analytic capacity and the Cauchy integral,
Colloq. Math. 60 (1990), 601–628.

\vspace{-0.3cm}

\bibitem{cgn17}
G. Cleanthous, A. G. Georgiadis and M. Nielsen,
Anisotropic mixed-norm Hardy spaces,
J. Geom. Anal. 27 (2017), 2758--2787.

\vspace{-0.3cm}

\bibitem{cgn19}
G. Cleanthous, A. G. Georgiadis and M. Nielsen,
Molecular decomposition of anisotropic homogeneous
mixed-norm spaces with applications to the boundedness of operators,
Appl. Comput. Harmon. Anal. 47 (2019), 447--480.

\vspace{-0.3cm}

\bibitem{cgp22}
G. Cleanthous, A. G. Georgiadis and E. Porcu,
Oracle inequalities and upper bounds for kernel
density estimators on manifolds and more general metric spaces,
J. Nonparametr. Stat. 34 (2022), 734--757.

\vspace{-0.3cm}

\bibitem{CMS85}
R. R. Coifman, Y. Meyer and E. M. Stein,
Some new function spaces and their applications to harmonic analysis,
J. Funct. Anal. 62 (1985), 304–335.

\vspace{-0.3cm}

\bibitem{CWhomo}
R. R. Coifman and G. Weiss,
Analyse Harmonique Non-commutative sur Certains Espaces Homog\`enes.
(French) \'Etude de Certaines Int\'egrales Singuli\`eres,
Lecture Notes in Math. 242,
Springer-Verlag, Berlin--New York, 1971.

\vspace{-0.3cm}

\bibitem{dj17}
Z. V. de Paul Abl\'e and J. Feuto,
Atomic decomposition of Hardy-amalgam spaces,
J. Math. Anal. Appl. 455 (2017), 1899--1936.

\vspace{-0.3cm}

\bibitem{dfmn21}
R. del Campo, A. Fern\'andez, F. Mayoral and F. Naranjo,
Orlicz spaces associated to a quasi-Banach function space:
applications to vector measures and interpolation,
Collect. Math. 72 (2021), 481--499.

\vspace{-0.3cm}

\bibitem{drs69}
P. L. Duren, B. W. Romberg and A. L. Shields,
Linear functionals on $H^p$ spaces with $0<p<1$,
J. Reine Angew. Math. 238 (1969), 32--60.

\vspace{-0.3cm}

\bibitem{fl16} X. Fan and B. Li,
Anisotropic tent spaces of Musielak--Orlicz type and their applications,
Adv. Math. (China) 45 (2016), 233--251.

\vspace{-0.3cm}

\bibitem{fs72}
C. Fefferman and E. M. Stein,
$H^p$ spaces of several variables,
Acta Math. 129 (1972), 137--193.

\vspace{-0.3cm}

\bibitem{Folland}
G. B. Folland,
Real Analysis, Modern Techniques and Their Applications,
Second Edition, Pure and Applied Mathematics (New York),
Wiley, New York, 1999.

\vspace{-0.3cm}

\bibitem{fs82}
G. B. Folland and E. M. Stein,
Hardy Spaces on Homogeneous Groups, Mathematical Notes 28,
Princeton Univ. Press, Princeton, 1982.

\vspace{-0.3cm}

\bibitem{gkp21}
A. G. Georgiadis, G. Kyriazis and P. Petrushev,
Product Besov and Triebel--Lizorkin spaces
with application to nonlinear approximation,
Constr. Approx. 53 (2021), 39--83.

\vspace{-0.3cm}

\bibitem{gn18}
A. G. Georgiadis and M. Nielsen,
Spectral multipliers on spaces of distributions
associated with non-negative self-adjoint operators,
J. Approx. Theory 234 (2018), 1--19.

\vspace{-0.3cm}

\bibitem{GTM249}
L. Grafakos,
Classical Fourier Analysis, Third edition,
Graduate Texts in Mathematics 249,
Springer, New York, 2014.

\vspace{-0.3cm}

\bibitem{hhllyy19}
Z. He, Y. Han, J. Li, L. Liu, D. Yang and W. Yuan,
A complete real-variable theory of Hardy spaces
on spaces of homogeneous type,
J. Fourier Anal. Appl. 25 (2019), 2197--2267.

\vspace{-0.3cm}

\bibitem{ho03}
K.-P. Ho,
Frames associated with expansive matrix dilations,
Collect. Math. 54 (2003), 217--254.

\vspace{-0.3cm}

\bibitem{ho21}
K.-P. Ho,
Operators on Orlicz-slice spaces and Orlicz-slice Hardy spaces,
J. Math. Anal. Appl. 503 (2021), Paper No. 125279, 18 pp.

\vspace{-0.3cm}

\bibitem{ho22}
K.-P. Ho,
Fractional integral operators on Orlicz slice Hardy spaces,
Fract. Calc. Appl. Anal. 25 (2022), 1294--1305.

\vspace{-0.3cm}

\bibitem{hlyy20}
L. Huang, J. Liu, D. Yang and W. Yuan,
Real-variable characterizations of new
anisotropic mixed-norm Hardy spaces,
Commun. Pure Appl. Anal. 19 (2020), 3033--3082.

\vspace{-0.3cm}

\bibitem{hw22}
L. Huang and X. Wang,
Anisotropic variable Campanato-type spaces and their
Carleson measure characterizations,
Fract. Calc. Appl. Anal. 25 (2022), 1131--1165.

\vspace{-0.3cm}

\bibitem{hyy21}
L. Huang, D. Yang and W. Yuan,
Anisotropic mixed-norm Campanato-type spaces with applications
to duals of anisotropic mixed-norm Hardy spaces,
Banach J. Math. Anal. 15 (2021), Paper No. 62, 36 pp.

\vspace{-0.3cm}

\bibitem{hpr12} T. Hyt\"onen, C. P\'erez and E. Rela,
Sharp reverse H\"older property for $A_{\infty}$
weights on spaces of homogeneous type,
J. Funct. Anal. 263 (2012), 3883--3899.

\vspace{-0.3cm}

\bibitem{its19}
M. Izuki, T. Noi and Y. Sawano,
The John--Nirenberg inequality in ball Banach function spaces
and application to characterization of BMO,
J. Inequal. Appl. 2019 (2019), Paper No. 268, 11 pp.

\vspace{-0.3cm}

\bibitem{is17} M. Izuki and Y. Sawano,
Characterization of BMO via ball Banach function spaces,
Vestn. St.-Peterbg. Univ. Mat. Mekh. Astron. 4(62) (2017), 78--86.

\vspace{-0.3cm}

\bibitem{jtyyz2022} H. Jia, J. Tao, D. Yang, W. Yuan amd Y. Zhang,
Boundedness of Calder\'on--Zygmund operators on special
John--Nirenberg--Campanato and Hardy-type spaces via congruent cubes,
Anal. Math. Phys. 12 (2022), Paper No. 118, 35 pp.

\vspace{-0.3cm}

\bibitem{jn87}
R. Johnson and C. J. Neugebauer,
Homeomorphisms preserving $A_p$,
Rev. Mat. Iberoam. 3 (1987), 249--273.

\vspace{-0.3cm}

\bibitem{lby14}
B. Li, M. Bownik and D. Yang,
Littlewood--Paley characterization and duality of
weighted anisotropic product Hardy spaces,
J. Funct. Anal. 266 (2014), 2611--2661.

\vspace{-0.3cm}

\bibitem{lffy15}
B. Li, X. Fan,  Z. Fu and D. Yang,
Molecular characterization of anisotropic
Musielak--Orlicz Hardy spaces and their applications,
Acta Math. Sin. (Engl. Ser.) 32 (2016), 1391--1414.

\vspace{-0.3cm}

\bibitem{lfy15}
B. Li, X. Fan and D. Yang,
Littlewood--Paley theory of anisotropic Hardy
spaces of Musielak--Orlicz type,
Taiwanese J. Math. 19 (2015), 279--314.

\vspace{-0.3cm}

\bibitem{libaode14}
B. Li, D. Yang and W. Yuan,
Anisotropic Hardy spaces of Musielak--Orlicz type with
applications to boundedness of sublinear operators,
The Scientific World Journal 2014, Article ID 306214,
19 pp., doi: 10.1155/2014/306214.

\vspace{-0.3cm}

\bibitem{lj10}
J. Li,
Atomic decomposition of weighted Triebel--Lizorkin
spaces on spaces of homogeneous type,
J. Aust. Math. Soc. 89 (2010), 255--275.

\vspace{-0.3cm}

\bibitem{lj11}
J. Li,  L. Song and C. Tan,
Various characterizations of product Hardy space,
Proc. Amer. Math. Soc. 139 (2011), 4385--4400.

\vspace{-0.3cm}

\bibitem{lj13}
J. Li and L. A. Ward,
Singular integrals on Carleson measure spaces ${\rm CMO}^p$
on product spaces of homogeneous type,
Proc. Amer. Math. Soc. 141 (2013), 2767--2782.

\vspace{-0.3cm}

\bibitem{lhy12}
Y. Liang, J. Huang and D. Yang,
New real-variable characterizations of Musielak--Orlicz Hardy spaces,
J. Math. Anal. Appl. 395 (2012), 413--428.

\vspace{-0.3cm}


\bibitem{liu2020}
J. Liu,
Molecular characterizations of variable anisotropic Hardy spaces
with applications to boundedness of Calder\'on--Zygmund operators,
Banach J. Math. Anal. 15 (2021), 1--24.

\vspace{-0.3cm}

\bibitem{lhy20}
J. Liu, D. D. Haroske and D. Yang,
A survey on some anisotropic Hardy-type function spaces,
Anal. Theory Appl. 36 (2020), 373--456.

\vspace{-0.3cm}

\bibitem{lhyy20}
J. Liu, D. D. Haroske, D. Yang and W. Yuan,
Dual spaces and wavelet characterizations of
anisotropic Musielak--Orlicz Hardy spaces,
Appl. Comput. Math. 19 (2020), 106--131.

\vspace{-0.3cm}

\bibitem{llh2023}
J. Liu, Y. Lu and L. Huang,
Dual spaces of anisotropic variable Hardy--Lorentz
spaces and their applications.
Fract. Calc. Appl. Anal. 26 (2023), 913--942.

\vspace{-0.3cm}

\bibitem{lwyy2018}
J. Liu, F. Weisz, D. Yang and W. Yuan,
Variable anisotropic Hardy spaces and their applications,
Taiwanese J. Math. 22 (2018), 1173--1216.

\vspace{-0.3cm}

\bibitem{lwyy2019}
J. Liu, F. Weisz, D. Yang and W. Yuan,
Littlewood--Paley and finite atomic characterizations of
anisotropic variable Hardy--Lorentz spaces and their applications,
J. Fourier Anal. Appl. 25 (2019), 874--922.

\vspace{-0.3cm}

\bibitem{lyy16}
J. Liu, D. Yang and W. Yuan,
Anisotropic Hardy--Lorentz spaces and their applications,
Sci. China Math. 59 (2016), 1669--1720.

\vspace{-0.3cm}

\bibitem{lyy17}
J. Liu, D. Yang and W. Yuan,
Anisotropic variable Hardy--Lorentz spaces and their real interpolation,
J. Math. Anal. Appl. 456 (2017), 356--393.

\vspace{-0.3cm}

\bibitem{lyy2018}
J. Liu, D. Yang and W. Yuan,
Littlewood--Paley characterizations of anisotropic Hardy--Lorentz spaces,
Acta Math. Sci. Ser. B (Engl. Ed.) 38 (2018), 1--33.

\vspace{-0.3cm}

\bibitem{lyz2023}
J. Liu, D. Yang and M. Zhang,
Sharp bilinear decomposition for products of both anisotropic
Hardy spaces and their dual spaces with its applications to
endpoint boundedness of commutators,
Sci. China Math. (Submitted).

\vspace{-0.3cm}

\bibitem{morrey38}
C. B. Morrey,
On the solutions of quasi-linear elliptic partial differential equations,
Trans. Amer. Math. Soc. 43 (1938), 126--166.

\vspace{-0.3cm}

\bibitem{mu94}
S. M\"uller,
Hardy space methods for nonlinear partial differential equations,
Tatra Mt. Math. Publ. 4 (1994), 159--168.

\vspace{-0.3cm}

\bibitem{s18} Y. Sawano, Theory of Besov Spaces,
Developments in Mathematics 56, Springer, Singapore, 2018.

\vspace{-0.3cm}

\bibitem{shyy17} Y. Sawano, K.-P. Ho, D. Yang and S. Yang,
Hardy spaces for ball quasi-Banach function spaces,
Dissertationes Math. 525 (2017), 1--102.

\vspace{-0.3cm}

\bibitem{st09}
Y. Sawano and H. Tanaka,
Predual spaces of Morrey spaces with non-doubling measures,
Tokyo J. Math. 32 (2009), 471--486.

\vspace{-0.3cm}

\bibitem{st15} Y. Sawano and H. Tanaka,
The Fatou property of block spaces,
J. Math. Sci. Univ. Tokyo 22 (2015), 663--683.

\vspace{-0.3cm}

\bibitem{st87} H.-J. Schmeisser and H. Triebel,
Topics in Fourier Analysis and Function Spaces,
John Wiley \& Sons, Ltd., Chichester, 1987.

\vspace{-0.3cm}


\bibitem{syy1}
J. Sun, D. Yang and W. Yuan,
Weak Hardy spaces associated with ball quasi-Banach function
spaces on spaces of homogeneous type: Decompositions,
real interpolation, and Calder\'on--Zygmund operators,
J. Geom. Anal. 32 (2022), Paper No. 191, 85 pp.

\vspace{-0.3cm}

\bibitem{tw80}
M. H. Taibleson and G. Weiss,
The molecular characterization of certain Hardy spaces,
Representation theorems for Hardy spaces, pp. 67-149,
Ast\'{e}risque, 77, Soc. Math. France, Paris, 1980.

\vspace{-0.3cm}

\bibitem{tl15}
C. Tan and J. Li,
Littlewood--Paley theory on metric spaces with
non doubling measures and its applications,
Sci. China Math. 58 (2015), 983--1004.

\vspace{-0.3cm}

\bibitem{tyyz21}
J. Tao, D. Yang, W. Yuan and Y. Zhang,
Compactness characterizations of commutators on ball Banach
function spaces, Potential Anal. 58 (2023), 645--679.

\vspace{-0.3cm}

\bibitem{tribelfs83}
H. Triebel, Theory of Function Spaces,
Birkh\"{a}user Verlag, Basel, 1983.

\vspace{-0.3cm}

\bibitem{tribelfs92}
H. Triebel, Theory of Function Spaces. II,
Birkh\"{a}user Verlag, Basel, 1992.

\vspace{-0.3cm}

\bibitem{U} T. Ullrich, Continuous characterization of Besov--Lizorkin--Triebel
space and new interpretations as coorbits, J. Funct. Spaces Appl.
2012 (2012), Art. ID 163213, 47 pp.

\vspace{-0.3cm}

\bibitem{w73}
T. Walsh,
The dual of $H^p(\mathbb{R}^{n+1}_+)$ for $p<1$,
Canad. J. Math. 25 (1973), 567--577.	

\vspace{-0.3cm}

\bibitem{wyy} F. Wang, D. Yang and S. Yang,
Applications of Hardy spaces associated with ball
quasi-Banach function spaces, Results Math. 75 (2020), Paper No. 26, 58 pp.

\vspace{-0.3cm}

\bibitem{wyyz} S. Wang, D. Yang, W. Yuan and Y. Zhang,
Weak Hardy-type spaces associated with ball quasi-Banach
function spaces II: Littlewood--Paley characterizations
and real interpolation, J. Geom. Anal. 31 (2021), 631--696.

\vspace{-0.3cm}

\bibitem{wyy22}
Z. Wang, X. Yan and D. Yang, Anisotropic Hardy spaces associated with ball
quasi-Banach function spaces and their applications, Kyoto J. Math. (to appear).

\vspace{-0.3cm}

\bibitem{yhyy}
X. Yan, Z. He, D. Yang and W. Yuan, Hardy spaces associated with
ball quasi-Banach function spaces on spaces of homogeneous type:
Characterizations of maximal functions, decompositions, and dual spaces,
Math. Nachr. (2023), https://doi.org/10.1002/mana.202100432.

\vspace{-0.3cm}

\bibitem{yhyy2}
X. Yan, Z. He, D. Yang and W. Yuan,
Hardy spaces associated with ball quasi-Banach function spaces
on spaces of homogeneous type: Littlewood--Paley characterizations
with applications to boundedness of Calder\'on--Zygmund operators,
Acta Math. Sin. (Engl. Ser.) 38 (2022), 1133--1184.

\vspace{-0.3cm}

\bibitem{yyy20}
X. Yan, D. Yang and W. Yuan,
Intrinsic square function characterizations of Hardy
spaces associated with ball quasi-Banach function spaces,
Front. Math. China 15 (2020), 769--806.

\vspace{-0.3cm}

\bibitem{yyy20b}
X. Yan, D. Yang and W. Yuan,
Intrinsic square function characterizations of several
Hardy-type spaces -- a survey,
Anal. Theory Appl. 37 (2021), 426--464.

\vspace{-0.3cm}

\bibitem{zhyy21} Y. Zhang, L. Huang, D. Yang and W. Yuan,
New ball Campanato-type function spaces and their applications,
J. Geom. Anal. 32 (2022), Paper No. 99, 42 pp.

\vspace{-0.3cm}

\bibitem{zwyy} Y. Zhang, S. Wang, D. Yang and W. Yuan,
Weak Hardy-type spaces associated with ball quasi-Banach
function spaces I: Decompositions with applications to
boundedness of Calder\'on--Zygmund operators,
Sci. China Math. 64 (2021), 2007--2064.

\vspace{-0.3cm}

\bibitem{zyy21} Y. Zhang, D. Yang and W. Yuan,
Real-variable characterizations of local Orlicz-slice Hardy
spaces with application to bilinear decompositions,
Commun. Contemp. Math. 24 (2022), Paper No. 2150004, 35 pp.

\vspace{-0.3cm}

\bibitem{zyyw} Y. Zhang, D. Yang, W. Yuan and S. Wang,
Real-variable characterizations of Orlicz-slice Hardy spaces,
Anal. Appl. (Singap.) 17 (2019), 597--664.

\end{thebibliography}
\end{document}